\documentclass{article}

\usepackage[a4paper, total={6in, 8in}]{geometry}
\usepackage{emptypage}
\usepackage[utf8]{inputenc}
\usepackage[T1]{fontenc}
\usepackage{lmodern}
\usepackage[english]{babel}
\usepackage{amsmath, amssymb}
\usepackage{amsthm}
\usepackage{longtable}
\usepackage{placeins}

\usepackage{algorithm}
\usepackage{algpseudocode}
\algrenewcommand\algorithmicrequire{\textbf{Input:}}
\algrenewcommand\algorithmicensure{\textbf{Output:}}

\usepackage{enumerate}
\usepackage{xcolor}
\RequirePackage[colorlinks,allcolors=blue]{hyperref}
\RequirePackage{graphicx}
\usepackage[nottoc]{tocbibind}

\usepackage{titlesec}
\titleformat{\section}{\large\bfseries}{\thesection}{1em}{}

\usepackage[round]{natbib} 
\let\oldcite\cite
\renewcommand{\cite}[1]{\mbox{\oldcite{#1}}}

\newtheorem{theorem}{Theorem}
\newtheorem{lemma}{Lemma}
\newtheorem{corollary}{Corollary}
\newtheorem{proposition}{Proposition}

\theoremstyle{definition}
\newtheorem{assumption}{Assumption}

\theoremstyle{remark}
\newtheorem{remark}{Remark}
\newtheorem*{remark*}{Remark}

\newtheorem{model}{Model}

\DeclareMathOperator*{\esssup}{ess\,sup}

\title{A Bootstrap Test for  Independence of Time Series Based on the Distance Covariance.}
\date{\today}
\author{Annika Betken\footnote{Faculty of Electrical Engineering, Mathematics and Computer Science (EEMCS), University of Twente, \href{mailto:a.betken@utwente.nl}{a.betken@utwente.nl}} \and Herold Dehling\footnote{Faculty of Mathematics, Ruhr-Universität Bochum, \href{mailto:herold.dehling@rub.de}{herold.dehling@rub.de}} \and Marius Kroll\footnote{Faculty of Mathematics, Ruhr-Universität Bochum, \href{mailto:marius.kroll@rub.de}{marius.kroll@rub.de}}}
\begin{document}
\maketitle

\begin{abstract}
  We present a test for independence of two strictly stationary time series based on a bootstrap procedure for the distance covariance. Our test detects any kind of dependence between the two time series within an arbitrary maximum lag $L$. In simulation studies, our test outperforms alternative testing procedures. In proving the validity of the underlying bootstrap procedure, we generalise bounds for the Wasserstein distance between an empirical measure and its marginal distribution under the assumption of $\alpha$-mixing. Previous results of this kind only existed for i.i.d. processes.
\end{abstract}

\section{Introduction}
Consider a strictly stationary process $(X_k,Y_k)_{k\in \mathbb{N}}$, where $X_k$  and $Y_k$ take  values in $\mathbb{R}^{\ell_1}$ and $\mathbb{R}^{\ell_2}$, respectively.  For ease of notation, we denote by $(X,Y)$ a random vector that has the same distribution as $(X_k,Y_k)$. Our goal is to test the hypothesis that the processes
$(X_k)_{k\in \mathbb{N}}$  and  $(Y_k)_{k\in \mathbb{N}}$ are independent against the alternative that these two processes are dependent. We construct a test based on the distance covariance, defined by \cite{srb} as
\begin{equation}
  \mathrm{dcov}(X,Y) := \iint \left| \phi_{X,Y}(s,t)-\phi_X(s) \phi_Y(t)   \right|^2 w(s,t)\, \mathrm{d}s\, \mathrm{d}t,
\end{equation}
where $w(s,t)$ is some positive weight function, and where $\phi_{X,Y}(s,t)$, $\phi_X(s)$, and $\phi_Y(t)$ denote the joint characteristic function of $(X,Y)$, and the marginal characteristic functions of $X$ and $Y$, respectively. In contrast with the classical Pearson correlation, which is only able to detect linear dependence between $X$ and $Y$, distance covariance will detect any kind of dependence, i.e. $\mathrm{dcov}(X,Y)=0$ if and only if the random variables $X$ and $Y$ are independent.  In our paper, we will take the specific weight function $w(s,t)=c(\|s\|_2^{\ell_1+1} \|t\|_2^{\ell_2+1})^{-1}$, recommended by \cite{srb} because of  a number of desirable properties.

Given the sample $(X_1,Y_1),\ldots, (X_n,Y_n)$, we define the empirical distance covariance
\begin{equation}
  \mathrm{dcov}_n(X,Y) := \iint \left| \phi_{X,Y}^{(n)} (s,t) - \phi_X^{(n)}(s) \phi_Y^{(n)}(t)  \right|^2 w(s,t) \, \mathrm{d}s\, \mathrm{d}t,
\end{equation}
where $\phi_{X,Y}^{(n)}(s,t)=\frac{1}{n} \sum_{j=1}^n e^{ i(sX_j+tY_j) }$, $\phi_X^{(n)}(s) =\frac{1}{n} \sum_{j=1}^n e^{isX_j}$, and $\phi_Y^{(n)}=\frac{1}{n}\sum_{j=1}^n e^{itY_j}$ are the empirical (joint) characteristic functions of the data. For i.i.d. processes $(X_k,Y_k)_{k\geq 1}$ with finite first moments, \cite{srb} show that $\mathrm{dcov}_n(X,Y)$ is a consistent estimator of $\mathrm{dcov}(X,Y)$. \cite{davistime} extend this result to stationary ergodic processes with values in a Euclidean space, and \cite{kroll} establishes consistency for stationary ergodic processes with values in  separable metric spaces.

\cite{srb} derive the large sample distribution of the empirical distance covariance in the case of i.i.d. processes $(X_k,Y_k)_{k\geq 1}$, and under the hypothesis that $X_i$ and $Y_i$ are independent. Under suitable moment conditions they show that, as $n\rightarrow \infty$,
\begin{equation}
  \label{eq:schwache_konergenz_unter_h0}
  n\, \mathrm{dcov}_n(X,Y) \xrightarrow[n \to \infty]{\mathcal{D}} \zeta,
\end{equation}
where $\zeta$ is some non-degenerate limiting distribution. This result is extended to the case of mixing data  by \cite{davistime} for processes with values in a Euclidean space, and by \cite{kroll} to general separable metric spaces.

By Eq.\@ \eqref{eq:schwache_konergenz_unter_h0}, the test that rejects the independence hypothesis when $n\, \mathrm{dcov}_n(X,Y) \geq \zeta_\alpha$, where $\zeta_\alpha$ denotes the upper $\alpha$-quantile of the limiting distribution  $\zeta$, has asymptotically level $\alpha$. Unfortunately this approach is not practically feasible, because $\zeta$ depends in a complex way on the distribution of $(X,Y)$, which in addition is unknown. In order to remedy this situation, \cite{dehlingprozesse} propose a bootstrap procedure to calculate asymptotic critical values of $\zeta$. In their paper, the authors establish the validity of the bootstrap in the case of i.i.d. processes $(X_k,Y_k)_{k\geq 1}$. However, their bootstrap does not capture the serial dependence in the data, and is thus not valid for time series. In addition, their technique of proof depends heavily on the i.i.d. assumption and cannot easily be generalized to dependent data.

In the present paper, we propose a block bootstrap technique to determine the quantiles of $\zeta$  that works for a large class of weakly dependent data. More precisely, we perform two independent non-overlapping block bootstrap procedures, namely one for each of the samples $X_1,\ldots,X_n$, and $Y_1,\ldots, Y_n$. In this way, we obtain two independent bootstrap samples $X_1^\ast,\ldots, X_n^\ast$  and $Y_1^\ast,\ldots, Y_n^\ast$. Under the hypothesis that the processes $(X_k)_{k\geq 1}$ and $(Y_k)_{k\geq 1}$ are independent,  the distribution of the paired bootstrap sample $(X_1^\ast,Y_1^\ast),\ldots, (X_n^\ast,Y_n^\ast)$ mimics the unknown distribution of the original sample $(X_1,Y_1),\ldots, (X_n,Y_n)$. In our theoretical analysis, we derive a bound for the Wasserstein distance between the bootstrap and the sample distribution of the empirical distance covariance. An important technical ingredient in the proof is a bound for the Wasserstein distance between the empirical distribution and the theoretical distribution of samples from a multivariate weakly dependent process. Our results generalize earlier results obtained by \cite{dereich} for i.i.d. data.

In its original version, our test is only able to capture dependence between $X_i$ and $Y_i$, but not between $X_i$ and $Y_{i+k}$ for $k > 0$. This problem is not new, and is treated in various contexts in the literature. \cite{haugh:1976} presents a Portmanteau-type test for independence based on linear combinations of Pearson's cross correlations for short-range dependent time series, while \cite{shao:2009} covers long-range dependent processes. \cite{betkendehling:2022} extend these ideas to tests for independence via linear combinations of distance cross covariances, again in the case of long-range dependent data. \cite{wang_et_al:2021} propose a test based on sums of Hilbert-Schmidt independence criteria (HSIC), which by \cite{sejdinovic:2013} is equivalent to the distance covariance in a certain sense. Most recently, \cite{chu:2023} constructs a test based on a variation of the distance covariance which, like that of \cite{wang_et_al:2021}, presupposes specific data generating models. \cite{edelman_et_al:2019} give an extensive literature survey of the distance covariance in the context of time series.

In the present paper, we present an alternative approach by applying our original bootstrap test to the vectors
$X_k^\prime=(X_k,\ldots,X_{k+L})$ and $Y_k^\prime=(Y_k,\ldots,Y_{k+L})$ of $L$-lagged data. In this way, we obtain a test that is consistent against arbitrary deviations from independence, as long as they occur within a time lag of at most $L$. Furthermore, we do not require any specific model assumptions other than absolute regularity and some mild conditions on the Lebesgue densities of the samples.

The remainder of this paper is structured as follows: In the next section, we give an alternative representation of the distance covariance as a V-statistic, and present details of the bootstrap procedure which is based on this representation. Section \ref{sec:chapter3_mains} contains the statements of the main theoretical results and a discussion of some technical assumptions, preceded by a brief summary of the mathematical tools used in this article. In Section \ref{sec:chapter3_simulations_and_data_analysis}, we present the results of an extensive simulation study, as well as an application to financial data. Section \ref{sec:chapter3_proof_outlines} contains outlines of the proofs of our main theoretical results. Details of proofs and auxiliary results are available in an online supplement to this paper.

\section{Test Statistic and Bootstrap Technique}
\label{sec:test_statistic_and_bs_technique}
We will base our analysis of the distance covariance on a V-statistic representation introduced by \cite{lyons}. We denote by $\theta=\mathcal{L}(X,Y)$ the joint distribution of $(X,Y)$, and by $\mu=\mathcal{L}(X)$ and $\nu=\mathcal{L}(Y)$ the distributions of $X$ and $Y$, respectively. By $\theta_n, \mu_n$ and $\nu_n$, we denote the corresponding empirical distributions.  As the distance covariance depends only on the joint distribution of $(X,Y)$, we will also use the notation $\mathrm{dcov}(\theta)$ instead of $\mathrm{dcov}(X,Y)$. In this notation, the empirical distance covariance is simply the distance covariance of the empirical distribution $\theta_n$, i.e. $\mathrm{dcov}_n(X,Y)=\mathrm{dcov}(\theta_n)$.

In order to obtain the V-statistics representation of the empirical distance covariance, we define the functions $f_\ell:(\mathbb{R}^\ell)^4\to \mathbb{R}$ by $f_\ell(s_1,s_2,s_3,s_4) := \|s_1-s_2\|_2-\|s_1-s_3\|_2 -\|s_2-s_4\|_2+\|s_3-s_4\|_2$, $\ell \in \{\ell_1, \ell_2\}$, and the kernel $h^\prime: (\mathbb{R}^{\ell_1+\ell_2})^6\rightarrow \mathbb{R}$ by
\[
  h^\prime((s_1,t_1),\ldots, (s_6,t_6)):= f_{\ell_1}(s_1,s_2,s_3,s_4)\, f_{\ell_2}(t_1,t_2,t_5, t_6).
\]
For the distance covariance with the specific weight function
\[
  w(s,t) = \frac{ \Gamma\left(\frac{\ell_1+1}{2}\right) \Gamma\left(\frac{\ell_2+1}{2}\right) }{\pi^{(1+\ell_1+\ell_2)/2} \|s\|_2^{l_1+1} \|t\|_2^{l_2+1}},
\]
\cite{lyons} showed that $ \mathrm{dcov}(\theta)=\int h^\prime ~\mathrm{d}\theta^6$.
As a consequence, one obtains a V-statistic representation of the empirical distance covariance
\[
  \mathrm{dcov}(\theta_n) = \int h^\prime ~\mathrm{d}\theta_n^6
  = \sum_{1\leq i_1,\ldots,i_6 \leq n} h^\prime( (X_{i_1},Y_{i_1}), \ldots, (X_{i_6},Y_{i_6}) )
\]
We will use this representation in our bootstrap procedure and also in the proofs of our main results. To simplify some calculations, we will often use the symmetrisation of $h^\prime$, defined as
\[
  h((s_1,t_1),\ldots, (s_6,t_6) ) := \frac{1}{6!}  \sum_{\sigma\in \mathcal{S}_6}
  h^\prime( (s_{\sigma(1)},t_{\sigma(1)}),\ldots, (s_{\sigma(6)},t_{\sigma(6)})  ),
\]
where $\mathcal{S}_6$ denotes the symmetric group of order $6$. Since any V-statistic is equal to the V-statistic with the corresponding symmetrized kernel, the empirical distance covariance may also be expressed as V-statistic with the above symmetric kernel $h$.

As it is central to our test, let us also explain the weak limit $\zeta$ from Eq.\@ \eqref{eq:schwache_konergenz_unter_h0} in more detail. \cite{davistime} and \cite{kroll} show (for Euclidean and metric data, respectively) that
$$
  \zeta = \sum_{k=1}^\infty \lambda_k \zeta_k^2,
$$
where $(\lambda_k, \varphi_k)_{k \in \mathbb{N}}$ is a sequence of non-negative eigenvalues and matching eigenfunctions of a certain integral operator, and $(\zeta_k)_{k \in \mathbb{N}}$ is a centred Gaussian process whose covariance function is given by
\begin{equation}
  \label{eq:kovarianz_von_zeta_formel}
  \mathrm{Cov}(\zeta_i, \zeta_j) = \lim_{n \to \infty} \frac{1}{n} \sum_{s,t=1}^n \mathrm{Cov}\left(\varphi_i(X_s, Y_s), \varphi_j(X_t, Y_t)\right).
\end{equation}
The integral operator mentioned depends on the joint distribution $\theta$. Thus, the distribution of $\zeta$ depends on the unknown distribution $\theta$ as well as the unknown covariances occurring in Eq.\@ \eqref{eq:kovarianz_von_zeta_formel}.

Let us now present the block bootstrap procedure as well as the resulting test in detail. In the following pseudocode, $d = d(n)$ will denote the block length that the two separate non-overlapping block bootstraps are based upon. We give instructions on how to choose this block length in Section \ref{sec:chapter3_mains}.

\begin{algorithm}
  \caption{\texttt{IBB} (Independent Block Bootstrap)}
  \label{alg:ibb}
  \begin{algorithmic}[1]
    \Require $(X_1, \ldots, X_n)$, $(Y_1, \ldots, Y_n)$, $d$
    \State $N \gets \lfloor n/d \rfloor$
    \For{$k = 1, \ldots, N$}
    \State $B_{X,k} \gets (X_{(k-1)d + 1}, \ldots, X_{kd})$
    \State $B_{Y,k} \gets (Y_{(k-1)d + 1}, \ldots, Y_{kd})$
    \EndFor
    \For{$k = 1, \ldots, N$}
    \State $B_{X,k}^* \gets$ random element from $\{B_{X,1}, \ldots, B_{X,N}\}$ drawn with replacement
    \State $B_{Y,k}^* \gets$ random element from $\{B_{Y,1}, \ldots, B_{Y,N}\}$ drawn with replacement
    \EndFor
    \State $\left(X_1^*, \ldots, X_{Nd}^*\right) \gets \left(B_{X,1}^*, \ldots, B_{X,N}^*\right)$
    \State $\left(Y_1^*, \ldots, Y_{Nd}^*\right) \gets \left(B_{Y,1}^*, \ldots, B_{Y,N}^*\right)$
    \Ensure $\left(X_1^*, \ldots, X_{Nd}^*\right)$, $\left(Y_1^*, \ldots, Y_{Nd}^*\right)$
  \end{algorithmic}
\end{algorithm}

The resulting bootstrap sample is of length $Nd$, where $N$ denotes the number of blocks. Thus, if the sample size $n$ is not divisible by the block length $d$, we disregard some of our observations (but always fewer than $d$). With an appropriate choice of $d = d(n)$, this does not impact the consistency of our bootstrap procedure.

In the pseudocode for the $L$-lag-test (Algorithm \ref{alg:llt}), $\mathrm{dcov}_{n-L}(X', Y')$ denotes the empirical distance covariance of the samples $X'_1, \ldots, X'_{n-L}$ and $Y'_1, \ldots, Y'_{n-L}$. Likewise, $\mathrm{dcov}_{n-L}\left(X^*, Y^*\right)$ is the empirical distance covariance of the corresponding bootstrap samples $X^*$ and $Y^*$. Note that the number of bootstrap blocks is given by $N = \lfloor (n-L)/d\rfloor$ because we are performing the bootstrap procedure on the vectorised sequence $X'_1, \ldots, X'_{n-L}$ and $Y'_1, \ldots, Y'_{n-L}$, not on the original observed sample of length $n$. \texttt{IBB} denotes the independent block bootstrap as described in Algorithm \ref{alg:ibb}.

\begin{algorithm}
  \caption{$L$-lag-test (empirical version)}
  \label{alg:llt}
  \begin{algorithmic}[1]
    \Require $(X_1, \ldots, X_n)$, $(Y_1, \ldots, Y_n)$, $d$, $\alpha$, $L$, $B$
    \For{$k = 1, \ldots, n-L$}
    \State $X'_k \gets \left(X_k \ldots, X_{k+L}\right)$
    \State $Y'_k \gets \left(Y_k, \ldots, Y_{k+L}\right)$
    \EndFor
    \State $N \gets \lfloor (n-L)/d\rfloor$
    \For{$b = 1, \ldots, B$}
    \State $\left(X^*, Y^*\right) \gets \texttt{IBB}\left((X'_1, \ldots, X'_{n-L}), (Y'_1, \ldots, Y'_{n-L}), d\right)$
    \State $D_b \gets Nd \, \mathrm{dcov}_{n-L}\left(X^*, Y^*\right)$\label{algline:dcov1}
    \EndFor
    \State $c_\alpha^* \gets$ empirical upper $\alpha$-quantile of $\{D_1, \ldots, D_B\}$
    \If{$(n-L) \, \mathrm{dcov}_{n-L}(X', Y') > c_\alpha^*$}\label{algline:dcov2}
    \State $\texttt{Decision} \gets$ `Reject $H_0$'
    \Else
    \State $\texttt{Decision} \gets$ `Do not reject $H_0$'
    \EndIf
    \Ensure \texttt{Decision}
  \end{algorithmic}
\end{algorithm}

In Algorithm \ref{alg:llt}, we are performing the independent block bootstrap $B$ times and take the empirical upper $\alpha$-quantile as our critical value. This is standard in the application of bootstrap procedures. However, one can also treat the result of the independent block bootstrap as a random variable, where the randomness comes from the sampling involved, as well as from the observed (random) sample $X_1, \ldots, X_n$ and $Y_1, \ldots, Y_n$. Denoting the empirical measure of this random bootstrap sample by $\theta_n^*$, we define
$$
  V^* := \mathrm{dcov}\left(\theta_n^*\right).
$$
The main results of this paper mostly concern the asymptotic behaviour of this bootstrap statistic.

\section{Main Theoretical Results}
\label{sec:chapter3_mains}
\subsection{Summary}
\label{sec:summaryofmains}
In Theorem \ref{thm:hyp_bootstrap} we show that, under certain technical assumptions, most importantly absolute regularity, $nV^*$ has the same limiting distribution in probability as $n\,\mathrm{dcov}(\theta_n)$ under $H_0$. From this it follows that the non-vectorised test has asymptotic level $\alpha$. Furthermore, because $\theta \neq \mu \otimes \nu$ implies that $n\,\mathrm{dcov}(\theta_n) \xrightarrow[n \to \infty]{a.s.} \infty$, this test is consistent against every alternative in which $X$ and $Y$ are not independent. Thus, the only alternatives against which this test is not consistent are those in which the marginals $X_k$ and $Y_k$ are independent for every $k \in \mathbb{N}$ but the entire processes $(X_k)_{k \in \mathbb{N}}$ and $(Y_k)_{k \in \mathbb{N}}$ are not independent. By first vectorising the observations, i.e.\@ by considering $X'_k = (X_k, \ldots, X_{k+L})$ and $Y'_k = (Y_k, \ldots, Y_{k+L})$, we are able to achieve consistency against any alternative for which $X_s$ and $Y_t$ are dependent for some pair $s,t \in \mathbb{N}$ such that $|s-t| \leq L$, where $L$ is a pre-specified parameter. Corollaries \ref{cor:asymptoticlevelalpha} and \ref{cor:L-lag-test} present the necessary details.

In Theorem \ref{thm:wasserstein_gesamt} we give an explicit bound on $\mathbb{E}d_p^p(\xi_n, \xi)$, where $d_p$ denotes the Wasserstein distance of order $p$, $\xi$ is some measure on $\mathbb{R}^d$ and $\xi_n$ is the empirical measure of a stationary and strongly mixing process with marginal distribution $\xi$. It is known that for growing dimension $d = d(n)$, this expected value does not necessarily converge to $0$, since the number of observations required for an approximation of fixed precision grows with $d$ -- this phenomenon is sometimes known as the curse of dimensionality. Since the upper bound in Theorem \ref{thm:wasserstein_gesamt} is explicit, it allows us to find a growth rate for $d = d(n)$ that still results in $\mathbb{E}d_p^p(\xi_n, \xi)$ converging to $0$. Previous results of this type, such as those in \cite{dereich}, require i.i.d. data as opposed to our weaker assumption of stationary and strongly mixing sample data. Finally, Corollary \ref{cor:boundwasserstein} is a handy consequence of Theorem \ref{thm:wasserstein_gesamt} for the special case where $\xi$ is the distribution of the first $d'$ observations of a strictly stationary process.

\subsection{Technical Preliminaries}
The following definitions are taken from \cite{bradley}. Let $(\Omega, \mathcal{F}, \mathbb{P})$ be a probability space. For any two sub-$\sigma$-algebras $\mathcal{A}, \mathcal{B} \subseteq \mathcal{F}$, we define the following objects:
\begin{align*}
  \alpha(\mathcal{A}, \mathcal{B}) & := \sup_{A \in \mathcal{A}, B \in \mathcal{B}} \left| \mathbb{P}(A \cap B) - \mathbb{P}(A)\mathbb{P}(B)\right|,                       \\
  \phi(\mathcal{A}, \mathcal{B})   & := \sup_{A \in \mathcal{A}, B \in \mathcal{B}, \mathbb{P}(A) > 0} \left| \mathbb{P}(B|A) - \mathbb{P}(B)\right|,                      \\
  \beta(\mathcal{A}, \mathcal{B})  & := \sup \left\{ \frac{1}{2} \sum_{i=1}^I \sum_{j=1}^J \left|\mathbb{P}(A_i \cap B_j) - \mathbb{P}(A_i)\mathbb{P}(B_j)\right|\right\},
\end{align*}
where the last supremum is taken over all possible partitions $\{A_1, \ldots, A_I\} \subseteq \mathcal{A}$ and $\{B_1, \ldots, B_J\} \subseteq \mathcal{B}$ with $I, J \in \mathbb{N}$. These coefficients characterise independence of $\mathcal{A}$ and $\mathcal{B}$ in the sense that, for any $\gamma \in \{\alpha, \beta, \phi\}$, $\gamma(\mathcal{A}, \mathcal{B}) = 0$ if and only if $\mathcal{A}$ and $\mathcal{B}$ are independent (\cite{bradley}, Proposition 3.4).

Now let $U := (U_k)_{k \in \mathbb{N}}$ be a stochastic process on $(\Omega, \mathcal{F}, \mathbb{P})$. For $1 \leq i,j \leq \infty$ denote by $\mathcal{F}_i^j$ the $\sigma$-algebra $\sigma(U_k, i \leq k \leq j) \subseteq \mathcal{F}$ generated by $U_i, U_{i+1}, \ldots, U_j$. Then, for any $\gamma \in \{\alpha, \beta, \phi\}$, we define $\gamma(n) := \sup_{j \in \mathbb{N}} \gamma\left(\mathcal{F}_1^j, \mathcal{F}_{j+n}^\infty\right)$, and we say that the process $(U_k)_{k \in \mathbb{N}}$ is $\gamma$-mixing if $\gamma(n) \to 0$ for $n \to \infty$. Instead of saying $\alpha$-mixing, we may use the term `strongly mixing'. Instead of saying $\beta$-mixing, we may use the term `absolutely regular'. By assuming a certain rate of growth for the mixing coefficients, one can control how much an observation may influence another observation a fixed distance away. A common assumption in our results will be absolute regularity with polynomial decay, i.e.\@ $\beta(n) = \mathcal{O}\left(n^{-r}\right)$ for some growth rate $r > 0$. Some of our results only require a strong mixing condition. This is a weaker assumption than absolute regularity; in fact, for any two sub-$\sigma$-algebras $\mathcal{A}, \mathcal{B}$ it holds that $2\alpha(\mathcal{A}, \mathcal{B}) \leq \beta(\mathcal{A}, \mathcal{B}) \leq \phi(\mathcal{A}, \mathcal{B})$. This is Proposition 3.11 in \cite{bradley}.

Finally, let us define the Wasserstein distance, which is a certain metric on the space of measures in which we will prove our convergence results. Let $\eta$ and $\xi$ be two measures on some separable metric space $(\mathcal{S}, d_\mathcal{S})$ and $p \geq 1$ a real constant. If $\eta$ and $\xi$ have finite $p$-th moments, then we define the Wasserstein distance $d_p(\eta, \xi)$ as
$$
  d_p(\eta, \xi) := \left(\inf_{\gamma \in \Gamma} \left\{ \int d_\mathcal{S}(s,s')^p ~\mathrm{d}\gamma(s,s')\right\}\right)^\frac{1}{p},
$$
where $\Gamma = \Gamma(\eta, \xi)$ is the set of all couplings of $\eta$ and $\xi$, i.e.\@ the set of all measures $\gamma$ on $\mathcal{S}^2$ with marginals $\eta$ and $\xi$. By abuse of notation, for any two random variables $U$ and $V$, we may write $d_p(U,V)$ instead of $d_p(\mathcal{L}(U), \mathcal{L}(V))$.

Convergence in the Wasserstein distance is stronger than weak convergence. By Theorem 6.9 in \cite{villani:optimaltransport}, if $(\xi_n)_{n \in \mathbb{N}}$ is a sequence of measures on some Polish space $(\mathcal{S}, d_\mathcal{S})$, then $(\xi_n)_{n \in \mathbb{N}}$ converges to some measure $\xi$ in the Wasserstein distance $d_p$ if and only if it converges weakly and for any (and thus all) $s_0 \in \mathcal{S}$ it holds that
$$
  \int d_\mathcal{S}(s,s_0)^p ~\mathrm{d}\xi_n(s) \to \int d_\mathcal{S}(s,s_0)^p ~\mathrm{d}\xi(s),
$$
i.e.\@ convergence in $d_p$ is equivalent to weak convergence plus convergence of the $p$-th moments.

\subsection{Results on the Test for Independence}
Our central result is the validity of our proposed bootstrap method, stated in Theorem \ref{thm:hyp_bootstrap}. More precisely, we prove convergence of the bootstrap in the Wasserstein distance.

For a sequence of measures $(\xi_d)_{d \in \mathbb{N}}$ on $\mathbb{R}^{\ell_d}$, we define the following assumption.
\begin{assumption}
  \label{ass:folgecubeassumption}
  Every $\xi_d$ has an essentially bounded Lebesgue density, and the corresponding sequence of essential suprema $M(d)$ grows at most exponentially in $\ell_d$.
\end{assumption}

To state our main result, we adopt the notation $Z_k := (X_k, Y_k)$. We will use this notation throughout the remainder of this article.
\begin{theorem}
  \label{thm:hyp_bootstrap}
  Suppose that $X_1$ and $Y_1$ both have finite $(4+\delta)$-th moments for some $\delta > 0$, and that the process $(Z_k)_{k \in \mathbb{N}}$ is strictly stationary and absolutely regular with $\beta(n) = \mathcal{O}\left(n^{-r}\right)$ for some $r > 18$. Furthermore, suppose that the joint distributions of the vectors $(Z_1, \ldots, Z_d)$ fulfil Assumption \ref{ass:folgecubeassumption}, and that $d = \mathcal{O}\left(\log(n)^\gamma\right)$ for some $0 < \gamma < 1/2$ and $d \to \infty$ as $n \to \infty$.

  Then it holds that
  $$
    d_1(\zeta,nV^{*}) \xrightarrow[n \to \infty]{\mathbb{P}} 0,
  $$
  where $\zeta$ is the weak limit of $n\, \mathrm{dcov}(\theta_n)$ under the hypothesis
  $$
    H_0 : (X_k)_{k \in \mathbb{N}} \textrm{ and } (Y_k)_{k \in \mathbb{N}} \textrm{ are independent}.
  $$
\end{theorem}

Theorem \ref{thm:hyp_bootstrap} also holds if the distributions of the vectors $(Z_1, \ldots, Z_d)$ do not fulfil Assumption \ref{ass:folgecubeassumption}, if instead we assume the sample generating process $(Z_k)_{k \in \mathbb{N}}$ to be $\phi$-mixing. We provide the necessary theoretical tools for this in Appendix \ref{app:proofs}.

One can also use the independent block bootstrap to approximate the limiting distribution of different $V$-statistics of absolutely regular sample data. For instance, the empirical Pearson covariance of $X$ and $Y$ (with normalisation $1/n$ instead of $1/(n-1)$) can be expressed as a $V$-statistic with kernel function
$$
  (z,z') \mapsto \frac{1}{2}(x - x')(y - y').
$$
Using an appropriate normalisation factor of $\sqrt{n}$, the bootstrap procedure yields the appropriate limiting distribution (in this case, a Gaussian distribution).

\begin{corollary}
  \label{cor:asymptoticlevelalpha}
  Suppose the assumptions from Theorem \ref{thm:hyp_bootstrap} are satisfied. Consider a test which rejects $H_0$ if $n\,\mathrm{dcov}(\theta_n) > c_\alpha^*$, where $c_\alpha^*$ is the upper $\alpha$-quantile of $nV^*$. Then this test has asymptotic level $\alpha$, and the only alternatives against which it is not consistent are those for which $X_k$ and $Y_k$ are independent for every $k \in \mathbb{N}$, but the entire processes $(X_k)_{k \in \mathbb{N}}$ and $(Y_k)_{k \in \mathbb{N}}$ are not independent.
\end{corollary}

\begin{corollary}
  \label{cor:L-lag-test}
  Suppose the assumptions from Theorem \ref{thm:hyp_bootstrap} are satisfied. Let $L \in \mathbb{N}$ be some arbitrary but fixed parameter. For any $1 \leq k \leq n-L$, let $X'_k = (X_k, \ldots, X_{k+L})$ and $Y'_k = (Y_k, \ldots, Y_{k+L})$. Then, performing the test from Corollary \ref{cor:asymptoticlevelalpha} on the new sample $(X'_1, Y'_1), \ldots, (X'_{n-L}, Y'_{n-L})$ yields a test of asymptotic level $\alpha$ that is consistent against any alternative for which $X_s$ and $Y_t$ are dependent for some $s,t \in \mathbb{N}$ such that $|s-t| \leq L$.
\end{corollary}

\subsection{A Bound for the Wasserstein Distance}
The following theorem is concerned with bounding the Wasserstein distance between an empirical measure of strongly mixing sample data and its marginal distribution. It provides a valuable tool in determining a rate of growth of the block length $d = d(n)$ which is slow enough that it does not impair the convergence of $\xi_n$ to $\xi$. Traditional Glivenko-Cantelli type results usually break down when one allows for a growing dimension since they implicitly depend on this dimension for their asymptotic bounds. Prior results of this kind, e.g.\@ those in \cite{dereich}, require i.i.d. sample data.

\begin{theorem}
  \label{thm:wasserstein_gesamt}
  Let $d \in \mathbb{N}$ and $1 \leq p \leq d/2$ and $q > p$ be fixed. Let $\xi$ be a probability measure on $\mathbb{R}^d$ with finite $q$-moments whose Lebesgue density exists and is essentially bounded by some $M > 0$. Then, for any $n,K \in \mathbb{N}$, it holds that
  $$
    \mathbb{E}d_p^p(\xi_n, \xi) \leq 3^{p-1}\left\{2^{p}\left(\left(1-\xi\left(U_K\right)\right)^\frac{q-p}{q}m_q^{p/q} + \left(1-\xi\left(U_K\right)\right)K^p\right) + K^{d/2}\cdot\mathfrak{M}^p\right\},
  $$
  where $U_K$ is the open sphere centred around the origin with radius $K$, $m_q$ is the $q$-th moment of $\xi$, and $\mathfrak{M}$ is given by
  $$
    \mathfrak{M}^p = c_0n^{-\frac{p-2}{2d}} 2^{3d/2-p} d^\frac{p}{2}\left(\frac{1 + M^\frac{d/2 - p}{d}}{1-2^{p - d/2}} + \frac{1}{1 - 2^{-p}}  +  4M^\frac{1}{d}\right),
  $$
  for some uniform constant $c_0$, where $\xi_n$ is the empirical measure of a strictly stationary and $\alpha$-mixing process $(U_i)_{i \in \mathbb{N}}$ with marginal distribution $\xi$ and $\alpha(n) \leq f(n) = \mathcal{O}(n^{-r_0})$ for some function $f$ and some constant $r_0 > 1$. The constant $c_0$ only depends on $f$ and $r_0$.
\end{theorem}

\begin{remark}
  \begin{enumerate}
    \item If $\xi$ is the measure of a random vector $(U_1', \ldots, U'_{d'})$ with equal marginal distributions, then $m_q$ can be bounded by $(d')^{q/2}\cdot\|U_1'\|_{L_q}^q$.
    \item $\xi$ having finite $q$-moments implies that $\xi\left(U_K^C\right) = o\left(K^{-q}\right)$, where $A^C$ denotes the complement of a set $A$.
    \item For a useful bound, we want to choose $K = n^{\delta(p-2)/(d^2)}$ for some $0 < \delta < 1$. This ensures that $K^{d/2}\cdot \mathfrak{M}^p$ will still converge to $0$.
    \item The observations above, combined with the fact that $d$ needs to be of order $\log(n)^\gamma$ with $0 < \gamma < 1/2$ for $\mathfrak{M}$ to converge to $0$, yield a useful bound for $\mathbb{E}d_p^p(\xi_n, \xi)$ which allows for changing $d$.
  \end{enumerate}
\end{remark}

Corollary \ref{cor:boundwasserstein} is an application of Theorem \ref{thm:wasserstein_gesamt} to segments taken from strictly stationary processes.

\begin{corollary}
  \label{cor:boundwasserstein}
  Suppose that the assumptions of Theorem \ref{thm:wasserstein_gesamt} are satisfied, and that $\xi$ is the measure of a random vector $(U_1', \ldots, U'_{d'})$ with equal marginal distributions (such as the first $d'$ observations of a stationary process) with finite $q$-moments. Then there is some $n_0 \in \mathbb{N}$ such that for all $n \geq n_0$,
  $$
    \mathbb{E}d_p^p(\xi, \xi_n) \leq c_p(d) \,n^{-\frac{p-2}{4d}} + c_{p,q}(d) \,n^{(p-2)(p-q)/(2d^2)},
  $$
  where $c_{p,q}(d) = 2c' 6^p d^{1+q/2}$ for some constant $c'$ and
  \begin{align*}
    c_{p}(d) & = 6^p c_0 2^{3d/2 - p} d^\frac{p}{2} \left(\frac{1 + M^\frac{d/2 - p}{d}}{1 - 2^{p-d/2}} + \frac{1}{1 - 2^{-p}} + 4M^\frac{1}{d}\right).
  \end{align*}
  The threshold $n_0$ and the constant $c'$ only depend on $\mathcal{L}(U_1')$, and are therefore independent of $d$ and $d'$. The constant $c_0$ only depends on $f$ and $r_0$.
\end{corollary}

\subsection{The Condition of Bounded Densities}
\label{subsec:bounded_densities}
In our results we make use of Assumption \ref{ass:folgecubeassumption}. We now show that this assumption is fulfilled most importantly by Gaussian processes, but also by elements of the more general class of elliptical distributions, i.e.\@ distributions whose Lebesgue densities are of the form
$$
  f(x) = c\cdot g\left((x-\Lambda)^T \Sigma^{-1} (x-\Lambda)\right),
$$
where $c$ is some normalisation constant, $\Sigma$ is a covariance matrix and $\Lambda$ is a location parameter (i.e.\@ a multidimensional expected value). Here, we assume that the random vectors are centred for the sake of simplicity. Some analysis shows that
$$
  \int g\left(x^T \Sigma^{-1} x\right) ~\mathrm{d}x = \int g\left(\left\|\sqrt{\Sigma^{-1}}x\right\|_2^2\right) ~\mathrm{d}x = \frac{1}{\sqrt{|\mathrm{det}\left(\Sigma^{-1}\right)|}} \int g\left(\|x\|_2^2\right) ~\mathrm{d}x,
$$
so in order for the density to integrate to a total measure of one, the normalisation constant $c$ must be given by
$$
  c = \sqrt{|\mathrm{det}\left(\Sigma^{-1}\right)|} \left(\int g\left(\|x\|_2^2\right) ~\mathrm{d}x\right)^{-1} = \sqrt{|\mathrm{det}\left(\Sigma^{-1}\right)|} \cdot \frac{\Gamma\left(\frac{d}{2}\right)}{2\pi^\frac{d}{2}} \left(\int_0^\infty t^{d-1} g\left(t^2\right) ~\mathrm{d}t\right)^{-1},
$$
where in the last equality we have used identity 3.3.2.1, Chapter 5, in \cite{intandseries} and $\Gamma$ denotes the $\Gamma$-function. The function $g$ is known, and so the factor
\begin{equation}
  \label{eq:normconstant}
  \Gamma\left(\frac{d}{2}\right)\left(\int_0^\infty t^{d-1} g\left(t^2\right) ~\mathrm{d}t\right)^{-1}
\end{equation}
can be evaluated.

The most important members of this class of distributions are the Gaussian distributions. In this case, the function $g$ is given by $g(t) = \exp(-t/2)$, and so the product in Eq.\@ \eqref{eq:normconstant} is equal to $1$. This gives us the bound of $\mathrm{det}\left(\Sigma^{-1/2}\right)\pi^{-d/2}$ for the multivariate Gaussian density. The following lemma therefore implies that the condition of bounded densities is fulfilled by any process whose first $d$ observations are elliptically distributed with a sufficiently well-behaved function $g$, most importantly Gaussian processes.

\begin{lemma}
  Let $\Sigma$ be the covariance matrix of the first $d$ observations $U_1, \ldots, U_d$ of a strictly stationary and absolutely regular process $(U_k)_{k \in \mathbb{N}}$ such that $\sum_{h=0}^\infty |\mathrm{Cov}(U_1, U_{1+h})| < \infty$.  Then it holds that
  $$
    \mathrm{det}\left(\Sigma^{-1}\right) \leq K^d,
  $$
  where $K > 0$ is a constant independent of $d$.
\end{lemma}
\begin{proof}
  Since the covariances are absolutely summable, there exists a spectral density $f$ such that
  $$
    \Sigma_{jk} = \mathrm{Cov}(U_j, U_k) = \frac{1}{2\pi}\int_{-\pi}^{\pi} f(\lambda)\exp(-i(j-k)\lambda) ~\mathrm{d}\lambda.
  $$
  The determinants of such Toeplitz forms are well studied. In particular, Szeg\H o's limit theorem (cf.\@ \cite{toeplitz}) gives us
  $$
    \lim_{d \to \infty} \mathrm{det}(\Sigma)^\frac{1}{d} = \exp\left(\frac{1}{2\pi} \int _{-\pi}^{\pi} \log f(\lambda) ~\mathrm{d}\lambda\right),
  $$
  which is finite by Theorem 5.8.1 in \cite{brockwell_davis:time_series} combined with Theorem 3.3.2 in \cite{bradley}. Thus,
  $$
    \mathrm{det}\left(\Sigma^{-1}\right) \leq \exp\left(-\frac{1}{2\pi} \int _{-\pi}^{\pi} \log f(\lambda) ~\mathrm{d}\lambda\right)^d
  $$
  for sufficiently large $d$.
\end{proof}

\section{Simulations and Data Analysis}
\label{sec:chapter3_simulations_and_data_analysis}
\subsection{Simulations}
\label{subsec:sims}

\begin{figure}[t]
  \begin{center}
    \includegraphics[width=\textwidth]{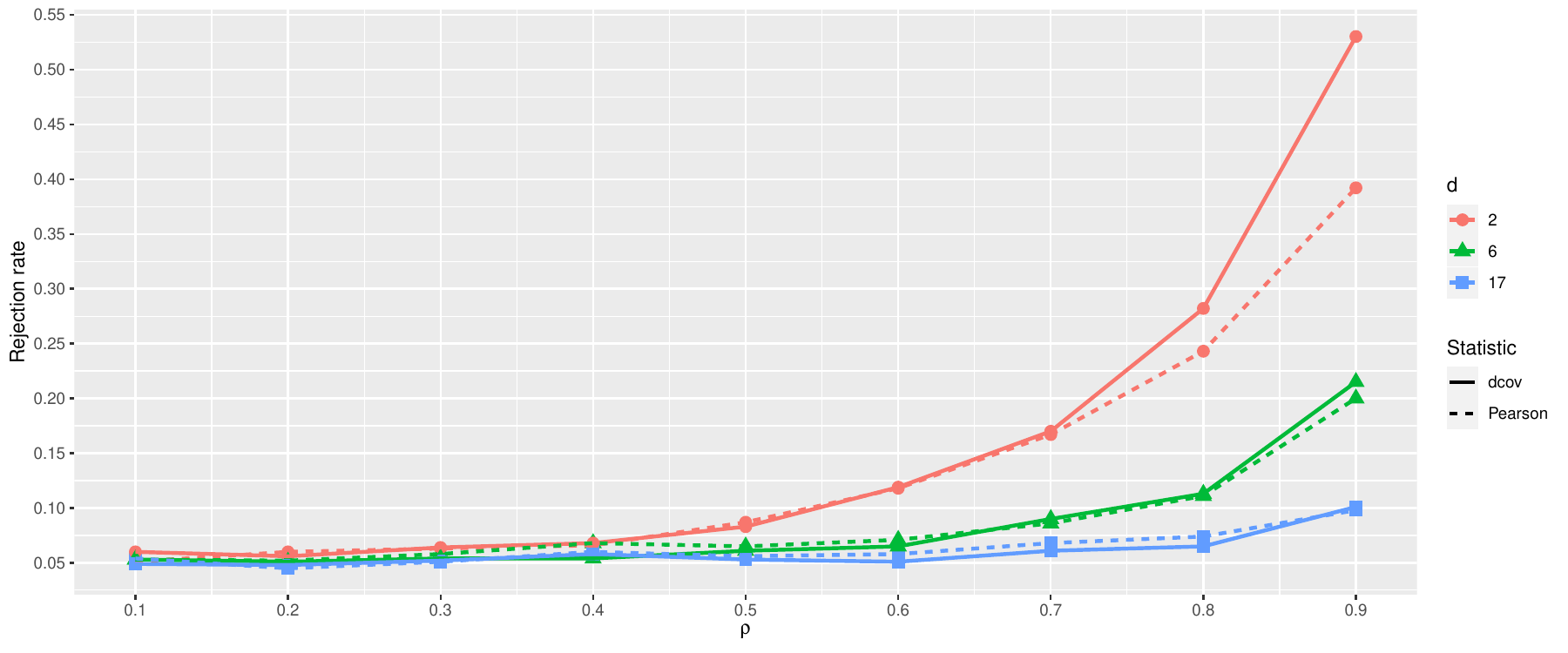}
    \caption{\small Rejection rates of the test based on Pearson's correlation and the distance covariance for two independent AR(1) processes with parameter $\rho$  and length $n = 300$. As block length we choose $d\in \{2, n^{1/3}, \sqrt{n}\} = \{2, 6, 17\}$.}
    \label{fig:ar1}
  \end{center}
\end{figure}

\begin{figure}[t]
  \begin{center}
    \includegraphics[width=\textwidth]{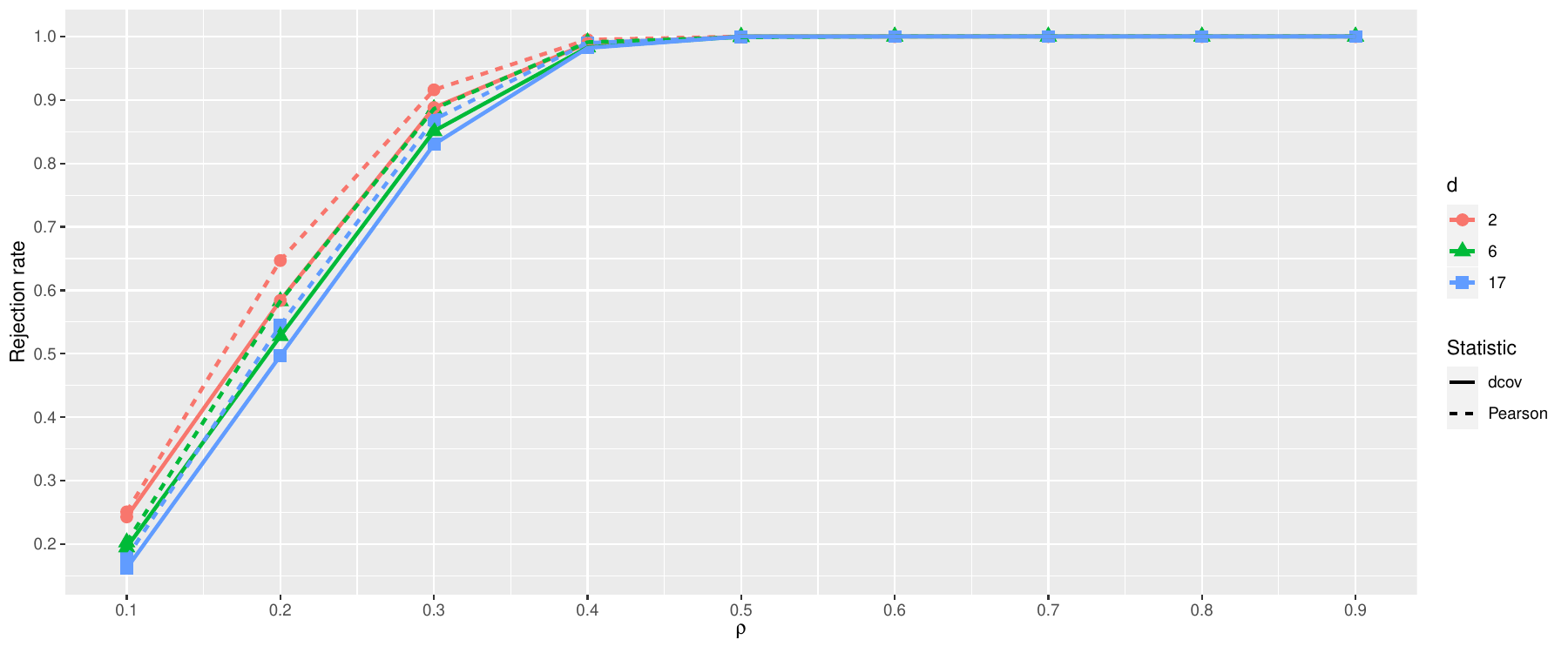}
    \caption{\small Rejection rates of the test based on Pearson's correlation and the distance covariance for  $\mathrm{VAR}(1)$ processes with parameters $a=0.5$, $b=0$, correlation matrix $\Gamma(0)$ as in Eq.\@ \eqref{eq:corr_matrix}, and length $n = 300$. As block length we choose $d\in \{2, n^{1/3}, \sqrt{n}\} = \{2, 6, 17\}$.}
    \label{fig:var1}
  \end{center}
\end{figure}

\begin{figure}[t]
  \begin{center}
    \includegraphics[width=\textwidth]{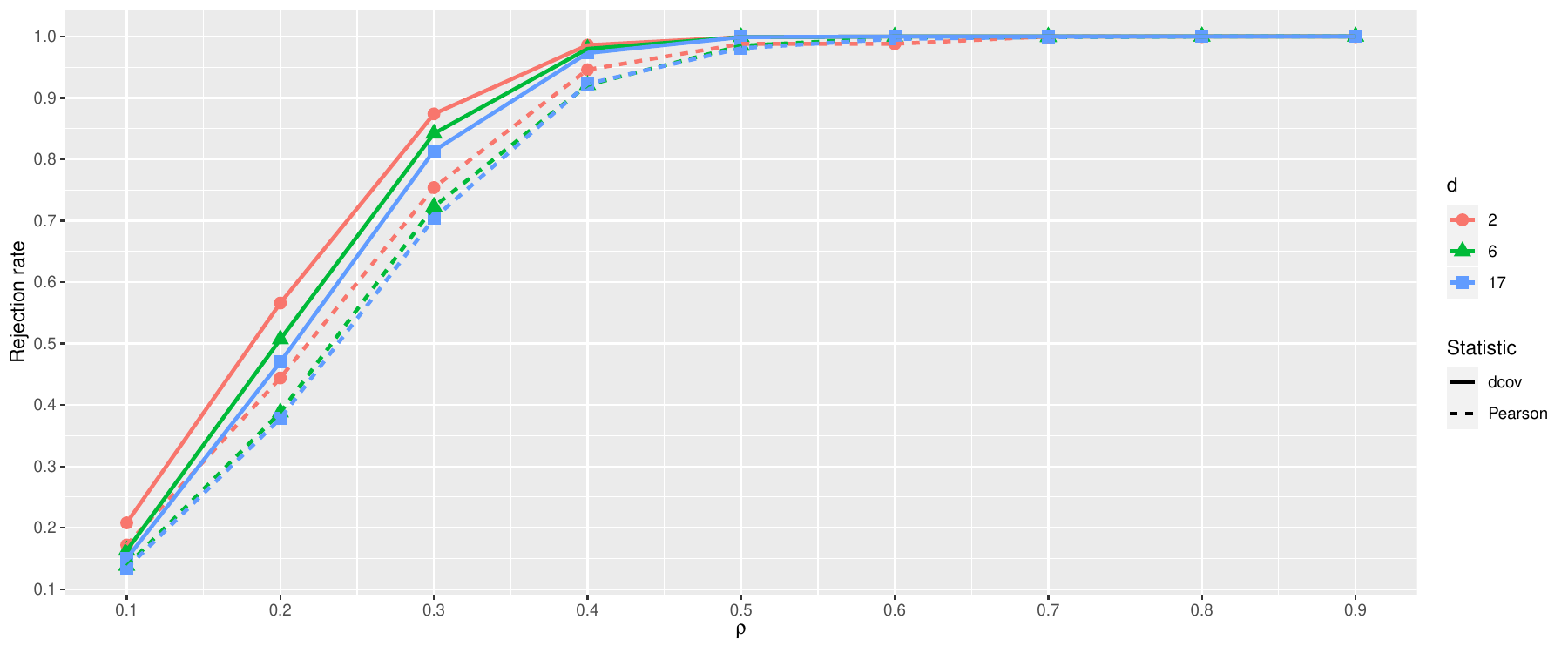}
    \caption{\small Rejection rates of the test based on Pearson's correlation and the distance covariance for  $\mathrm{VAR}(1)$ processes with parameters $a=0.5$, $b=0$, correlation matrix $\Gamma(0)$ as in Eq.\@ \eqref{eq:corr_matrix}, and length $n = 300$. We consider a quantile transformation that yields a $t$-distribution with $\nu = 3$ degrees of freedom. As block length we choose $d\in \{2, n^{1/3}, \sqrt{n}\} = \{2, 6, 17\}$. }
    \label{fig:t3}
  \end{center}
\end{figure}

\begin{figure}[t]
  \begin{center}
    \includegraphics[width=\textwidth]{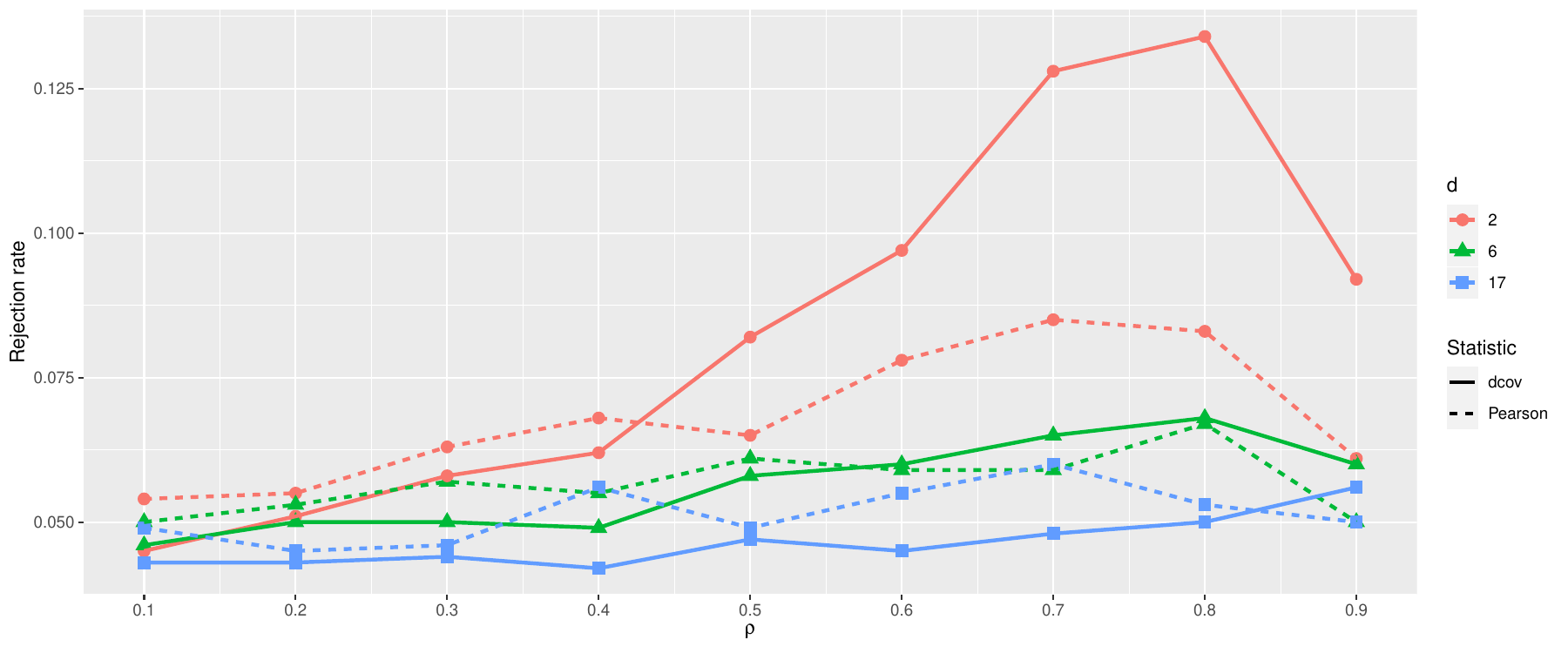}
    \caption{\small Rejection rates of the test based on Pearson's correlation and the distance covariance for two independent AR(1) processes with parameter $\rho$  and length $n = 300$. We consider a quantile transformation that yields a $t$-distribution with $\nu = 3$ degrees of freedom. As block length we choose $d\in \{2, n^{1/3}, \sqrt{n}\} = \{2,6,17\}$.}
    \label{fig:t_ind_1}
  \end{center}
\end{figure}

\begin{figure}[t]
  \begin{center}
    \includegraphics[width=\textwidth]{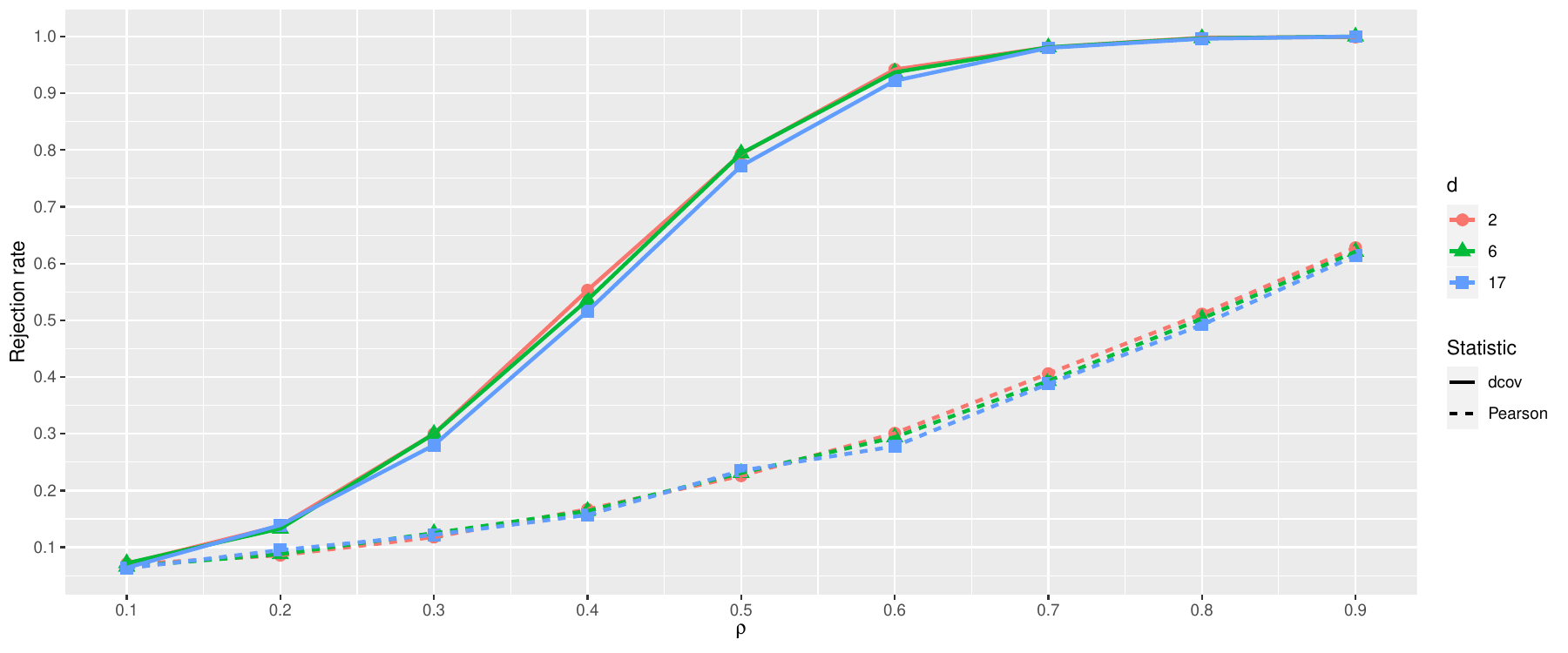}
    \caption{\small Rejection rates of the test based on Pearson's correlation and the distance covariance for stochastic volatility  time series of length $n = 300$, i.e.\@ for $Y_i = X_{i+1}$, $X_j=(\xi_j-\mathbb{E}\xi_j)\exp(\eta_j)$, $j=1, \ldots, n$ where $\eta_j$, $j=1, \ldots, n$, stems from an AR(1) process with parameter $\rho$ and $\xi_j$, $j=1, \ldots, n$, from i.i.d. Pareto distributed random variables with shape parameter $\alpha = 5$. As block length we choose $d\in \{2, n^{1/3}, \sqrt{n}\} = \{2, 6, 17\}$.}
    \label{fig:SV_2}.
  \end{center}
\end{figure}

\begin{figure}[t]
  \begin{center}
    \includegraphics[width=\textwidth]{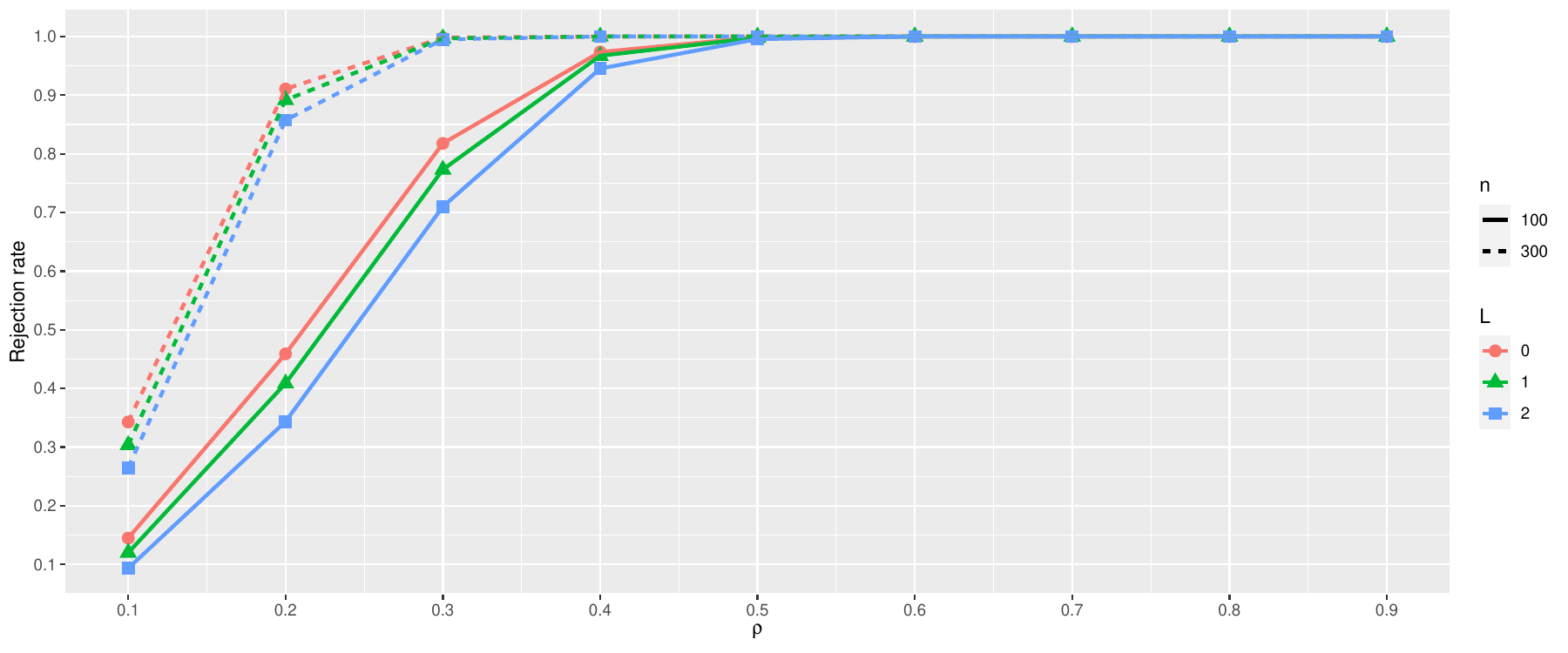}
    \caption{Rejection rates of the $L$-lag-test for data simulated according to Model \ref{model:shao} with parameter $l = 0$. In this model, $X_k$ and $Y_k$ are jointly normal with $\mathrm{Cov}(X_i, Y_j) = \rho$ if $|i-j| = l$ and $0$ otherwise. As block length we choose $d = \lfloor \sqrt{n}\rfloor $.}
    \label{fig:shao}.
  \end{center}
\end{figure}

\begin{figure}[t]
  \begin{center}
    \includegraphics[width=\textwidth]{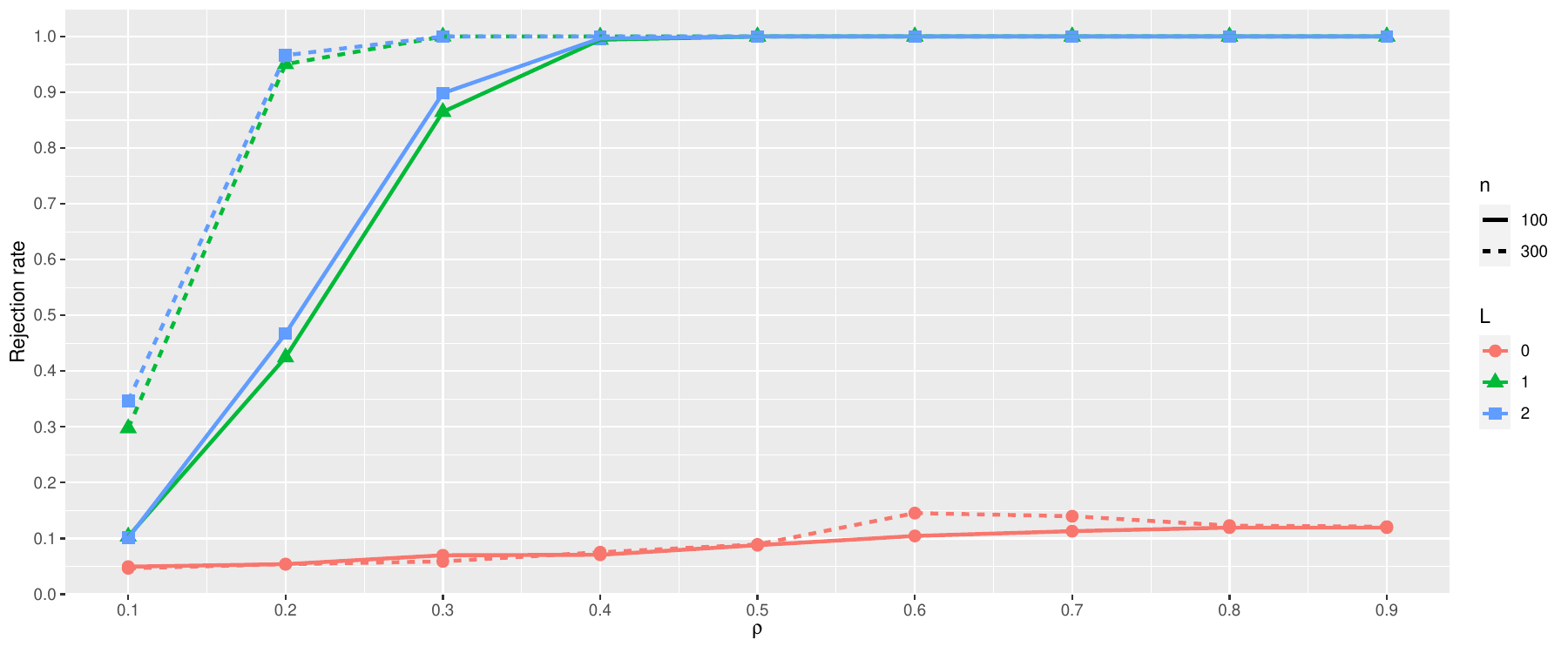}
    \caption{Rejection rates of the $L$-lag-test for data simulated according to Model \ref{model:shao} with parameter $l = 1$. In this model, $X_k$ and $Y_k$ are jointly normal with $\mathrm{Cov}(X_i, Y_j) = \rho$ if $|i-j| = l$ and $0$ otherwise. As block length we choose $d = \lfloor \sqrt{n}\rfloor $.}
    \label{fig:shao_lag2}.
  \end{center}
\end{figure}

In the following, the finite sample performance of the bootstrap procedure established in Section \ref{sec:chapter3_mains} is analysed in the context of testing for dependence between the components  $X_1, \ldots, X_n$ and  $Y_1, \ldots, Y_n$ of a bivariate time series $Z_1, \ldots, Z_n$ with $Z_k=(X_k, Y_k)^{T}$. For this, we perform Algorithm \ref{alg:llt} $k = 5000$ times (with $B = 500$ bootstrap repetitions) and determine the average number of rejections. In the case of $L = 0$, we can adapt Algorithm \ref{alg:llt} to use the Pearson covariance instead of the distance covariance by changing Lines \ref{algline:dcov1} and \ref{algline:dcov2} accordingly. This modified version of Algorithm \ref{alg:llt} serves as a comparison for our $0$-lag-test. Our $L$-lag-test for $L > 0$ has no obvious comparison since the Pearson covariance is not readily generalised to multivariate data.

We have used the following models to generate the data for our simulation study.

\begin{model}
  \label{model:ar_independent}
  We simulate $X_1, \ldots, X_n$ and $Y_1, \ldots, Y_n$ as two independent AR(1) processes, each with parameter $\rho$.
\end{model}
\begin{model}
  \label{model:var}
  We simulate $(X_1, Y_1), \ldots, (X_n, Y_n)$ as a VAR(1) process. A stationary, bivariate Gaussian process $W_i=(X_i,Y_i)$ is called $\mathrm{VAR}(1)$ process if there exists a matrix $A\in \mathbb{R}^{2\times 2}$, and independent, $\mathcal{N}(0,\Sigma)$-distributed Gaussian random variables $\xi_i=(\varepsilon_i,\eta_i)$  such that the equation
  \begin{equation}
    W_{i}=AW_{i-1} +\xi_{i}
    \label{eq:biv-ar1}
  \end{equation}
  is satisfied. Such a process exists if and only if all eigenvalues of $A$ are strictly less than $1$ in absolute value.  In the special example when the $\mathrm{VAR}(1)$ matrix is given by
  \begin{equation}
    A = \begin{pmatrix}
      a & b \\
      b & a
    \end{pmatrix},
    \label{eq:ar1-matrix}
  \end{equation}
  for two parameters $a$ and $b$, the eigenvalues of $A$ are given by $\lambda_1= a+b$ and $\lambda_2=a-b$. Let $\Gamma(h)$ denote the autocovariance matrix of $W_i$ at lag $h$. Then it holds that
  \begin{align*}
    vec \ \Gamma(0)=\left(I-A\otimes A\right)^{-1}vec \ \Sigma,
  \end{align*}
  where $\otimes$ denotes the Kronecker product or direct product of two matrices and the $vec$ operator transforms a matrix $A$ into an $(mn \times 1)$ vector by stacking the columns; cf.\@ \cite{lutkepohl:2005}. We choose $b=0$, $a = 1/2$, and
  \begin{align*}
    \Sigma=\begin{pmatrix}
             1    & \rho \\
             \rho & 1
           \end{pmatrix}.
  \end{align*}
  For this choice, the autocovariance matrix of $(X_i, Y_i)$ at lag $0$ is given by
  \begin{align}
    \Gamma(0)= \begin{pmatrix}
                 (1-0.5^2)^{-1}     & (1-0.5^2)^{-1}\rho \\
                 (1-0.5^2)^{-1}\rho & (1-0.5^2)^{-1}
               \end{pmatrix} = \frac{4}{3} \Sigma.
    \label{eq:corr_matrix}
  \end{align}
\end{model}

\begin{model}
  \label{model:t-verteilung}
  We simulate $(\tilde{X}_1, \tilde{Y}_1), \ldots, (\tilde{X}_n, \tilde{Y}_n)$ according to Model \ref{model:var} and set $X_i=F_{\nu}^{-1}(\Phi(\sqrt{3/4}X_i))$ and $Y_i=F_{\nu}^{-1}(\Phi(\sqrt{3/4}Y_i))$, where $F_{\nu}^{-1}$ denotes the inverse of the distribution function $F_{\nu}$ of a $t$-distribution with $\nu$ degrees of freedom. The resulting time series $(X_k)_{k \in \mathbb{N}}$ and $(Y_k)_{k \in \mathbb{N}}$ are dependent and have $t_\nu$-distributed marginals.
\end{model}

\begin{model}
  \label{model:t-verteilung_indep}
  We simulate $(\tilde{X}_1, \tilde{Y}_1), \ldots, (\tilde{X}_n, \tilde{Y}_n)$ according to Model \ref{model:ar_independent} and then proceed as in Model \ref{model:t-verteilung}, resulting in two independent time series with $t_\nu$-distributed marginals.
\end{model}

\begin{model}
  \label{model:stochastic_volatility}
  We simulate
  \begin{equation}
    \label{eq:vol_process}
    X_i = (\xi_i-\mathbb{E}\xi_i)\exp(\eta_i), \quad Y_i=X_{i+1},
  \end{equation}
  where $\eta_i$ stems from an $\mathrm{AR}(1)$ process with parameter $\rho$, and $\xi_i$ from independent, identically Pareto distributed random variables with shape parameter $\alpha$.

  This process $(X_k)_{k \in \mathbb{N}}$ is an example of a so-called stochastic volatility time series. These are typically defined via
  \begin{align}\label{eq: LMSV}
    X_j=\xi_j\varepsilon_j \ \ \text{with} \ \ \varepsilon_j=\exp\left(\eta_j\right),
  \end{align}
  where $(\xi_j)_{j \in \mathbb{N}}$ is a centred i.i.d. process, and $(\eta_j)_{j \in \mathbb{N}}$, is a  Gaussian process, independent of $(\xi_j)_{j \in \mathbb{N}}$.

  Stochastic volatility models are considered, for example, in \cite{taylor:1986}. Initially, this  model had been introduced by \cite{breidt:crato:delima:1998} and, independently, by \cite{harvey:2002}. Among other things, stochastic volatility time series can be used to model financial time series such as log-returns of stock market indices. (Examples for such time series are considered in Section \ref{subsec:data}.) A characteristic of stochastic volatility time series is that their autocorrelations equal $0$, while there is, nevertheless, (non-linear) dependence; cf.\@ \cite{cont:2005}.

  Another characteristic of financial time series is  the occurrence of heavy tails. In particular, probability distributions of log-returns of stock market indices often exhibit tails which are heavier than those of a normal distribution. This is the reason for our use of the Pareto distribution in this model.
\end{model}

\begin{model}
  \label{model:shao}
  We simulate a sequence $Z_i$, $i=1, \ldots, 2n$, of independent normally distributed observations.
  We define the covariance matrix
  \begin{align*}
    \Sigma=\begin{pmatrix}
             I_n & A   \\
             A   & I_n \\
           \end{pmatrix},
  \end{align*}
  where $I_n$ is the $n\times n$-identity matrix and $A=(a_{ij})_{1\leq i, j \leq n}$ is an $n\times n$-matrix with $a_{ij}=\rho$ if $|i-j|=l$ and $0$ otherwise. We generate $X_1, \ldots, X_n$ and $Y_1, \ldots, Y_n$ as follows:
  \begin{align*}
    (X_1, \ldots, X_n, Y_1, \ldots, Y_n)^{T}=\Sigma^{\frac{1}{2}}\mathbf{Z}, \ \text{where} \ \mathbf{Z}:=(Z_1, \ldots, Z_{2n}).
  \end{align*}
  As a result,
  \begin{align*}
    \mathrm{Cov}(X_i, Y_j)=\begin{cases}
                             \rho & \text{if $|i-j|=l$} \\
                             0    & \text{else}
                           \end{cases}
  \end{align*}
  This model has also been considered in \cite{shao:2009}.
\end{model}

To evaluate the test's performances under the hypothesis of no dependence between the two component series, we simulate $X_1, \ldots, X_n$ and $Y_1, \ldots, Y_n$ according to Model \ref{model:ar_independent}. Corresponding simulation results can be found in \mbox{Figure \ref{fig:ar1}}. Based on these, we observe that for both dependence measures an increase of dependence within each time series results in an increase of (false) rejections. Moreover, the simulation results indicate that an increase of the block length leads to a decrease of rejections. Generally speaking, the rejection rates tend to be lower for the test based on Pearson's correlation, which is expected as the observations follow a normal distribution. \mbox{Figure \ref{fig:ar1}} also illustrates the need for a growing block length, as the fixed choice $d = 2$ results in an oversized test even for moderate values of $\rho$ -- a problem which does not occur for, say, $d = \sqrt{n}$.

To evaluate the tests' performances under the alternative, we simulated data according to Models \ref{model:var} and \ref{model:t-verteilung}. The corresponding simulation results are given in Figures \ref{fig:var1} and \ref{fig:t3}. In both of these cases, the empirical power of the tests increases as the dependence within each time series increases. Additionally, the simulation results indicate consistency of the tests as an increasing sample size goes along with an increase of empirical power. In both cases, a smaller block length yields higher empirical power. For normally distributed margins (the first example) this effect is more pronounced than  for $t$-distributed margins (the second example). For normally distributed margins, the test based on Pearson's correlation yields a slightly higher power than the test based on distance covariance. Again, this may be attributed to the fact that Pearson's correlation is designed for the identification of dependence between Gaussian observations. For $t$-distributed margins the test based on distance covariance yields the higher power.

Figure \ref{fig:t3} should be compared with Figure \ref{fig:t_ind_1}. Here we have generated the data according to Model \ref{model:t-verteilung_indep}, resulting in two independent processes whose marginals are equally distributed as those in Figure \ref{fig:t3}. For block lengths growing with $n$, the rejection rates are reasonably close to the tests' level of $\alpha = 5\%$, i.e.\@ the better performance of the tests based on the distance covariance shown in Figure \mbox{\ref{fig:t3}} does not compromise the tests' behaviour under the hypothesis. In this sense, we consider the tests based on the distance covariance superior to those based on Pearson's correlation. Simulations for $\nu = 5$ are provided in Appendix \ref{app:sims_hypothesentest}.

Simulation results for data generated according to Model \ref{model:stochastic_volatility} can be found in Figure \ref{fig:SV_2}. Again, the empirical power of the test increases as the dependence within each time series increases. Additionally, the simulation results indicate consistency of the tests as an increasing sample size goes along with an increase of empirical power. With respect to the choice of block length there does not seem to be any pronounced effect on the tests' performances. Due to the fact that $\mathrm{Cov}(X_1, X_{j+1})=0$ for every fixed point in time $j$, $X_j$ and $Y_j$ are uncorrelated, but dependent. It is therefore not surprising that a test based on  distance covariance clearly outperforms a test based on Pearson's correlation when testing for dependence between the two components of an accordingly generated time series. Results for $\alpha = 3$ are provided in Appendix \ref{app:sims_hypothesentest}.

Figures \ref{fig:shao} and \ref{fig:shao_lag2} demonstrate the performance of our $L$-lag-test. For $n = 300$, the empirical rejection rates are essentially $1$ for most values of $\rho$. In Figure \ref{fig:shao}, corresponding to Model \ref{model:shao} with parameter $l = 0$, the power drops slightly if we choose $L$ too large compared to the real lag $l$ (which is unknown in practice). Although this effect is not replicated with $l = 1$ (see Figure \ref{fig:shao_lag2}), it may suggest that some caution is warranted when determining the maximum detectable lag $L$. However, the impact on the rejection rate gets smaller as $n$ grows larger, and, in our example, is not very significant even for $n = 300$. Figure \ref{fig:shao_lag2} also illustrates the need for vectorising the observation in the presence of lagged cross-correlations, as the rejection rate of our test is small if $L < l$.

\subsection{Data Analysis}
\label{subsec:data}

\begin{figure}[h]
  \begin{center}
    \includegraphics[width=\textwidth]{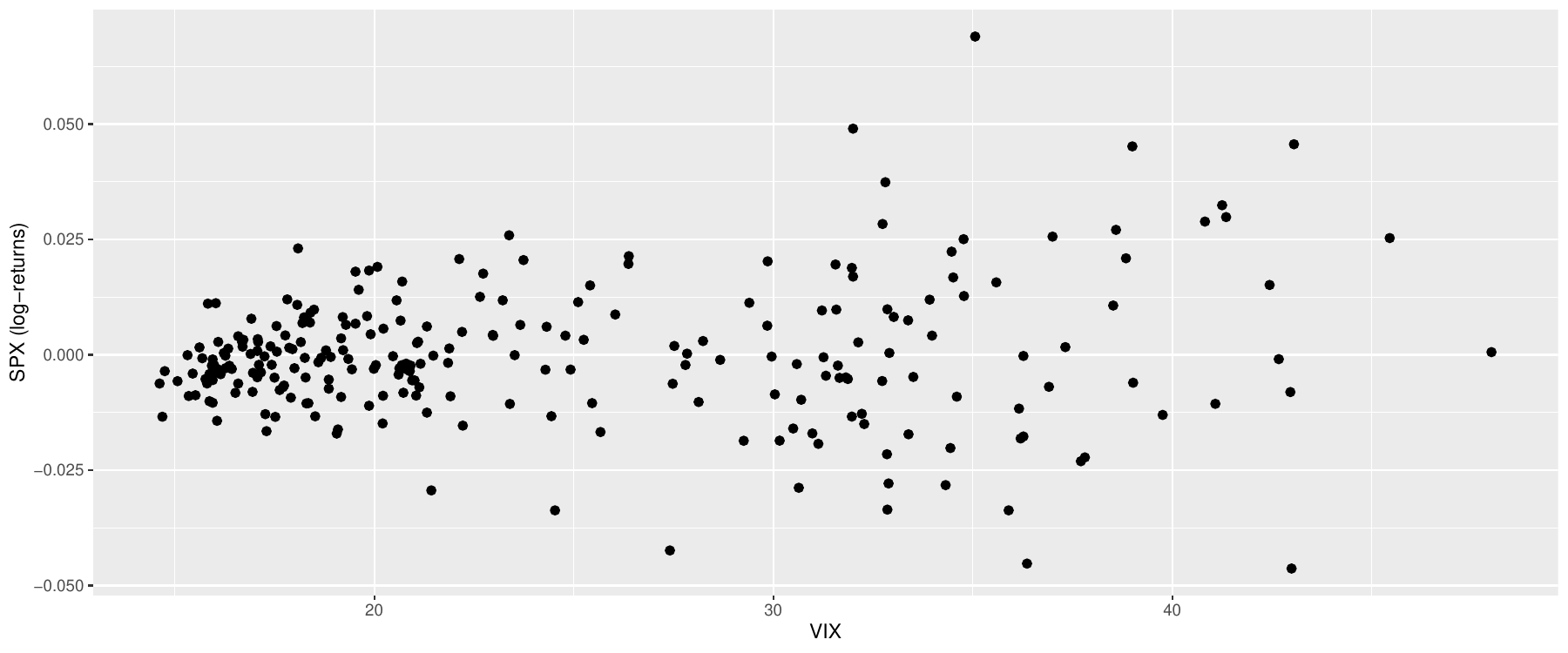}
    \caption{Scatterplot of the VIX and the log-returns of the SPX from January 2011 to January 2012.}
    \label{fig:scatterplot}
  \end{center}
\end{figure}

As an example for data analysis, we consider Standard \& Poor's 500  (S\&P 500 or SPX, in short), a stock market index tracking the performance of 500 of the largest companies listed on stock exchanges in the United States, and the corresponding Chicago Board Options Exchange (CBOE) Volatility Index (VIX, in short), a measure of the stock market's expectation of volatility based on S\&P 500 index options introduced in \cite{whaley:1993}.

The  analysis of the two indices in this section is based on open-source historical data of the closing indices recorded daily over a time period of one year from 26 January 2011 to 23 January 2012; cf. \cite{marketwatchSPX} and \cite{marketwatchVIX}.
As the indices are not recorded every single day and as the recorded days for VIX and SPX do not coincide, we deleted the recordings for which there are no values for both SPX \textit{and} VIX, resulting in time series of 252 observations each.
Since in general stock prices do not follow a stationary process, whereas their log-returns  display features of  stationarity,  we analyse the S\&P 500 index' log-returns. Formally, the log-returns are defined by
\begin{align*}
  L_t:= \log R_t, \ R_t:=\frac{P_t}{P_{t-1}},
\end{align*}
where $P_t$ denotes the value of the index on day $t$.
The VIX, however,  is generally assumed to be stationary, cf. \cite{avellaneda2019}, such that we consider its actual values, instead of its log-returns,  denoted by $V_t$, $t=1, \ldots, 252$. In order to remove the so-called `look-ahead bias', we consider $(V_t, L_{t+1})$, $t=1, \ldots, 250$ for our analysis.

In established literature, the SPX and VIX are typically found to be negatively correlated: if the SPX increases, the VIX decreases and the other way around; see, e.g. \cite{madan:2011} and \cite{schnurr:2014}. As stated by Robert Whaley, who introduced the VIX in 1993 (see \cite{whaley:1993}),  this relation between the rates of change in the VIX and the SPX is asymmetric, i.e. negative returns correlate with  a sharp rise of volatility while positive returns go along with a smaller drop in volatility; see \cite{whaley:2009}. In particular, empirical evidence disproves  linear dependence between the indices, which could be modelled by correlation; see, e.g. \cite{russon:2017}.

An application of the bootstrap procedure (described in Section \ref{sec:chapter3_mains}) with respect to both Pearson's correlation coefficient and the  distance correlation, based on 500 Bootstrap samples and a block length $d = 2$, yields quantiles as given in Table \ref{tab:datenquantile}.
Since the value of  Pearson's correlation coefficient of VIX and   SPX' log-returns from  January 2011 to  January 2012 is given by $0.1135$, the hypothesis that the two time series are independent is not rejected at a significance level of 5\%. This result confirms that the dependence of the two time series may be non-linear and as such cannot be easily  detected by a statistical test based on Pearson's correlation coefficient.

On the other hand, since the value of the distance correlation of VIX and SPX' log-returns  from January 2011 to  January 2012 is given by $0.3184$, the hypothesis that the two time series are independent is clearly rejected at a significance level of 1\%, as per the quantiles given in Table \ref{tab:datenquantile}. In this regard, our data analysis can be considered as proof of the time series' dependence.

Additionally, we applied the hypothesis test to two time series that we expect to be independent. For this, we chose the log-returns based on the SPX recorded daily from 26 January 2011 to 23 January 2012 and the same index recorded after a lag of one year, i.e.\@ recorded daily from 28 January 2013 to 23 January 2014.

For these data, the quantiles as determined by the bootstrap procedure with respect to both Pearson's correlation coefficient and the  distance correlation, based on 500 bootstrap samples and block length $d=2$, are given in Table \ref{tab:datenquantile}.

Since the value of the  distance correlation  of the indices' log-returns is given by $0.1222$, the hypothesis that the two time series are independent is not rejected at a significance level of 10\%. The value of  Pearson's correlation coefficient of the indices' log-returns  is given by $0.0869$, and so the hypothesis that the two time series are independent is not rejected at a significance level of 10\%.

\begin{table}[htb]
  \centering
  \begin{tabular}{cccccccc}
                          & \multicolumn{3}{c}{SPX 11/12 \& VIX 11/12} &        & \multicolumn{3}{c}{SPX 11/12 \& SPX 13/14}                               \\
    \cline{2-8}
                          &                                            &        &                                            &  &        &                 \\
                          & 90 \%                                      & 95\%   & 99\%                                       &  & 90 \%  & 95\%   & 99\%   \\
    \hline
    Pearson's correlation & 0.1014                                     & 0.1171 & 0.1626                                     &  & 0.1125 & 0.1329 & 0.1569 \\
    Distance correlation  & 0.1392                                     & 0.1494 & 0.1870                                     &  & 0.1421 & 0.1565 & 0.1798 \\
    \hline
  \end{tabular}
  \caption{\small Bootstrap quantiles for Pearson's covariance and the distance correlation. As a data example for the dependent case, we have used the VIX and the log-returns of the SPX for the same year (24 January 2011 to 24 January 2012; left side). For the independent case, we have used the log-returns of the SPX for different years (26 January 2011 to 24 January 2012 and 28 January 2013 to 24 January 2014; right side). Each quantile was calculated based on $500$ bootstrap samples with block length $d = 2$.}
  \label{tab:datenquantile}
\end{table}

\section{Outline of Proofs}
\label{sec:chapter3_proof_outlines}
\subsection{Outline for Theorem \ref{thm:hyp_bootstrap}}
\label{subsec:proofoutline_bs}
Consider that any two distinct blocks $B_i^*$ and $B_j^*$ of our bootstrap sequence correspond to concatenations $(B_{X,k_1}, B_{Y,l_1})$ and $(B_{X,k_2}, B_{Y,l_2})$ of blocks from our sample sequences $X_1, \ldots, X_n$ and $Y_1, \ldots, Y_n$. Roughly speaking, for large $n$, these four sample blocks will be far apart from each other with high probability, and because we assume $(Z_k)_{k \in \mathbb{N}}$ to be absolutely regular, this implies that they will be `almost' independent. The sequence of bootstrap samples $Z_1^*, \ldots, Z_n^*$ can therefore be expected to behave similarly to a collection of random variables $\tilde{Z}_{1,n}, \ldots, \tilde{Z}_{n,n}$ consisting of $N$ i.i.d. random vectors of length $d$ with marginal distribution
$$
  \mathcal{L}(X_1, \ldots, X_d) \otimes \mathcal{L}(Y_1, \ldots, Y_d) \overset{H_0}{=} \mathcal{L}(Z_1, \ldots, Z_d).
$$

Denote by $\tilde{V}$ the empirical distance covariance of $\tilde{Z}_{1,n}, \ldots, \tilde{Z}_{n,n}$, then a simple application of the triangle inequality yields
$$
  d_1(\zeta, nV^*) \leq d_1(\zeta, n\tilde{V}) + d_1(n\tilde{V}, nV^*).
$$
We can then employ the following theorem which we prove in Appendix \ref{app:proofs}.

\begin{theorem}
  \label{thm:asymptotik}
  If $X$ and $Y$ are independent and have finite $(2+\varepsilon)$-moments for some $\varepsilon > 0$, and the sequence $(Z_k)_{k \in \mathbb{N}}$ is absolutely regular with mixing coefficients $\beta(n) = \mathcal{O}(n^{-r})$ for some growth rate $r > 6(1+2\varepsilon^{-1})$, it holds that
  $$
    d_1\left(n\,\mathrm{dcov}(\theta_n), \zeta\right) \xrightarrow[n \to \infty]{} 0,
  $$
  where $\zeta := \sum_{k=1}^\infty \lambda_k \zeta_k^2$ with $(\zeta_k)_{k \in \mathbb{N}}$ being a centred Gaussian process whose covariance function is determined by the dependence structure of the sequence $(Z_k)_{k \in \mathbb{N}}$, and the parameters $\lambda_k > 0$ are determined by the underlying distribution $\theta$. If furthermore $d = d(n) \to \infty$ for $n \to \infty$ and $d^3 = o\left(n\right)$, then it also holds that
  $$
    d_1\left(n\tilde{V}, \zeta\right) \xrightarrow[n \to \infty]{} 0.
  $$
\end{theorem}

It remains to control the distance $d_1(n\tilde{V}, nV^*)$. By some standard theory for $V$-statistics, we can instead investigate the simpler object
$$
  d_1\left(n\tilde{V}_n^{(2)}(h; \mu \otimes \nu), nV_n^{*(2)}(h;\mu_n \otimes \nu_n)\right),
$$
where $V_n^{(c)}$ denotes the $c$-th term in the Hoeffding decomposition of $V$; see \cite{vandervaart} Section 12.3.

To specify our intuitive reasoning above, note that $\text{dcov}(\theta_n)$ can be expressed as a $V$-statistic with kernel
\begin{equation}
  \label{eq:definition_H}
  H(B_1, \ldots, B_6) := \frac{1}{d^6} \sum_{1 \leq i_1, \ldots, i_6 \leq d} h(B_{1,i_1}, \ldots, B_{6,i_6}),
\end{equation}
where $B_i := (Z_{(i-1)d + 1}, \ldots, Z_{id })$, $1 \leq i \leq N := n/d$, denotes the blocks of length $d = d(n)$ and $B_{j, i_j}$ denotes the $i_j$-th coordinate of $B_j$. More precisely, this means that
\begin{align}
  \label{eq:emp_dcov_H}
  \text{dcov}(\theta_n)=V_{H}(B_1, \ldots, B_N) ={\frac {1}{N^{6}}}\sum_{1\leq i_1, \ldots, i_6 \leq N}H(B_{i_{1}},\ldots ,B_{i_{6}}).
\end{align}

The bootstrapped empirical distance covariance $V^*$ and the statistic $\tilde{V}$ allow for analogous representations, e.g.\@
\begin{equation}
  \label{eq:definition_V*}
  V^* = V_H\left(B_1^*, \ldots, B_N^*\right) = \frac {1}{N^{6}} \sum_{1 \leq i_1, \ldots, i_6 \leq N}H(B^*_{i_{1}}, \ldots ,B^*_{i_{6}}).
\end{equation}

Furthermore, we can establish a link between the Hoeffding decompositions of the $V$-statistic with kernel $h$ and that of the blockwise `$V$-statistic' with kernel $H$. For this, let
\begin{align*}
  F^{(X)} := F_d^{(X)} & := \mathcal{L}(X_1, \ldots, X_d), \\
  F^{(Y)} := F_d^{(X)} & := \mathcal{L}(Y_1, \ldots, Y_d),
\end{align*}
and let $F_N^{(X)}$ and $F_N^{(Y)}$ be their empirical versions, i.e.\@ the empirical measures of $B_{X,1}, \ldots, B_{X,N}$ and $B_{Y,1}, \ldots, B_{Y,N}$, respectively. We show in Appendix \ref{app:proofs} that
\begin{align}
  \label{eq:vrangzwei}
  \begin{split}
    \tilde{V}_n^{(2)}(h, \mu \otimes \nu) &= \tilde{V}_N^{(2)}\left(H, F^{(X)} \otimes F^{(Y)}\right), \\
    V_n^{*(2)}(h, \mu_n \otimes \nu_n) &= V_N^{*(2)}\left(H, F_N^{(X)} \otimes F_N^{(Y)}\right).
  \end{split}
\end{align}
Conditionally on $Z_1, \ldots, Z_n$, both objects on the right-hand side are $V$-statistics of i.i.d. sample data. To control their Wasserstein distance, we will use the following proposition.

\begin{proposition}
  \label{prop:abstandvstatistik}
  Let $U_i \overset{i.i.d.}{\sim} \eta$, $V_i \overset{i.i.d.}{\sim} \xi$. Then there exists a universal constant $C$ such that
  \begin{align*}
     & d_2 \left(\mathcal{L}_\eta\left(\frac{1}{n} \sum_{1\leq i, j\leq n} H_2(U_i,U_j;\eta) \right), \mathcal{L}_\xi\left( \frac{1}{n} \sum_{1\leq i, j\leq n} H_2(V_i,V_j;\xi) \right)\right) \\
     & \qquad \leq   C \cdot  d_4(\eta,\xi) \left(\|U_1\|_{L_4} + \|V_1\|_{L_4}\right).
  \end{align*}
\end{proposition}

\begin{remark}
  By the notation $\mathcal{L}_\eta$, and $\mathcal{L}_\xi$, we indicate that the random variables $(U_i)_{i \in \mathbb{N}}$, and $(V_i)_{i \in \mathbb{N}}$, have distribution $\eta$, and $\xi$ respectively. Moreover, in the present context both of them are i.i.d. processes. Strictly speaking, we could use the same symbols for the  random variables in both cases, i.e.\@ also $\mathcal{L}_\xi(\frac{1}{n} \sum_{1\leq i\neq j\leq n} g_2(U_i,U_j;\xi))$, since the symbol $\mathcal{L}_\xi$ specifies the distribution of the process.
\end{remark}

This allows us to bound the Wasserstein distance of the $V$-statistics of interest by the Wasserstein distance of the marginal distributions of the underlying i.i.d. processes, i.e.\@ by
$$
  d_4\left(F^{(X)} \otimes F^{(Y)}, F_N^{(X)} \otimes F_N^{(Y)}\right).
$$
Finally, we show in Appendix \ref{app:proofs} that for some universal constant $c$,
$$
  d_4\left(F^{(X)} \otimes F^{(Y)}, F_N^{(X)} \otimes F_N^{(Y)}\right) \leq c\cdot\left\{d_4\left(F^{(X)}, F_N^{(X)}\right) + d_4\left(F^{(Y)}, F_N^{(Y)}\right)\right\},
$$
and thus we have reduced the problem to that of bounding the Wasserstein distance between the empirical measure of an absolutely regular process and its marginal distribution. To achieve this, we use Corollary \ref{cor:boundwasserstein}.

In Appendix \ref{app:proofs}, we prove a more general version of Theorem \ref{thm:asymptotik}. More specifically, we derive the following asymptotic result for degenerate $V$-statistics.
\begin{theorem}
  \label{thm:zschlange}
  Let $\mathcal{U}$ be a metrisable topological space, $(U_k)_{k \in \mathbb{N}}$ a strictly stationary sequence of $\mathcal{U}$-valued random variables with marginal distribution $\mathcal{L}(U_1) = \xi$. Consider a continuous, symmetric, degenerate and positive semidefinite kernel $g : \mathcal{U}^2 \to \mathbb{R}$ with finite $(2+\varepsilon)$-moments with respect to $\xi^2$ and finite $(1+\frac{\varepsilon}{2})$-moments on the diagonal, i.e.\@ $\mathbb{E}|g(U_1,U_1)|^{1+\varepsilon/2} < \infty$. Furthermore, let the sequence $(U_k)_{k \in \mathbb{N}}$ satisfy an $\alpha$-mixing condition such that $\alpha(n) = \mathcal{O}(n^{-r})$ for some $r > 1+2\varepsilon^{-1}$. Then, with $V = V_g(U_1, \ldots, U_n)$,
  $$
    d_1\left(nV, \sum_{k=1}^\infty \lambda_k \zeta_k^2\right) \xrightarrow[n \to \infty]{} 0,
  $$
  where $(\lambda_k, \varphi_k)$ are pairs of the non-negative eigenvalues and matching eigenfunctions of the integral operator
  $$
    f \mapsto \int g(\cdot,u)f(u)~\mathrm{d}\xi(u),
  $$
  and $(\zeta_k)_{k \in \mathbb{N}}$ is a sequence of centred Gaussian random variables whose covariance structure is given by
  \begin{equation}
    \label{eq:kovarianz}
    \mathrm{Cov}(\zeta_i, \zeta_j) = \lim_{n \to \infty} \frac{1}{n} \sum_{s,t=1}^n \mathrm{Cov}(\varphi_i(U_s), \varphi_j(U_t)).
  \end{equation}
\end{theorem}

\subsection{Outline for Theorem \ref{thm:wasserstein_gesamt}}
\label{subsec:proof_outline_wasserstein}
A similar result of this type was previously developed by \cite{dereich} for i.i.d. sample generating processes. In generalising the authors' result to strictly stationary and $\alpha$-mixing processes, we use some of the same general ideas that form the basis for their proof. Let us briefly recall the basic definitions from \cite{dereich}.

Consider the $d$-dimensional hypercube $[0,1)^d$. For any given $l \in \mathbb{N}_0$, we can partition $[0,1)^d$ into $2^{dl}$ translations of $2^{-l}[0,1)^d$. Denote by $\mathcal{P}_l$ the collection of all these translations. On the union $\mathcal{P} := \bigcup_{l=0}^\infty \mathcal{P}_l$, we can define a tree structure as follows: For any given set $C \in \mathcal{P}_l$, $l \in \mathbb{N}$, the father of $C$ is the unique set $F \in \mathcal{P}_{l-1}$ with $C \subseteq F$. We adopt the shorthand notation $C \leftarrow F$ for \textit{$C$ is a child of $F$}.

We also recall Lemma 2 from \cite{dereich}, which forms the basis for their proofs (and that we will use in the same manner). Because it is central to their proofs as well as ours, we include it here for the sake of clarity.
\begin{lemma}[\cite{dereich}, Lemma 2]
  \label{lem:paperlem2}
  Let $\eta$ and $\xi$ be two probability measures on $[0,1)^d$, with the property that $\eta(C) > 0$ if $\xi(C) > 0$ for all $C \in \mathcal{P}$. It then holds that
  $$
    d_p^p(\eta,\xi) \leq \frac{1}{2}d^\frac{p}{2}\sum_{l=0}^\infty 2^{-pl}\sum_{F \in \mathcal{P}_l} \sum_{C \leftarrow F} \left|\xi(C) - \xi(F)\frac{\eta(C)}{\eta(F)}\right|.
  $$
\end{lemma}
We will apply this lemma to the distance $d_p^p(\xi, \xi_n)$. However, because we are not working under i.i.d. assumptions, we cannot use the same methods as \cite{dereich} to further control the bound in Lemma \ref{lem:paperlem2}. Instead, note that
\begin{align*}
  \left|\xi_n(C) - \xi_n(F)\frac{\xi(C)}{\xi(F)}\right| & = \left|\xi_n(C) - \xi(C) + \xi(C) - \xi_n(F)\frac{\xi(C)}{\xi(F)}\right|                   \\
                                                        & \leq \left|\xi_n(C) - \xi(C)\right| + \frac{\xi(C)}{\xi(F)} \left|\xi_n(F) - \xi(F)\right|.
\end{align*}
Ignoring the factor $\xi(C)/\xi(F)$ for a moment, this means that finding a sufficient bound for the distances in the series in Lemma \ref{lem:paperlem2} can be reduced to finding a bound for
$$
  \mathbb{E}|\xi_n(F) - \xi(F)|,
$$
where $F$ is now a generic set in $\mathcal{P}$. Necessarily, such a bound will depend on $\xi(F)$. We are able to derive a bound for sufficiently small sets $F$ (small in the sense of having a small measure), and have to somehow bound the number of sets which have too big a measure for our bound to apply. This explains the necessity for the assumption of bounded Lebesgue densities, which implies that
\begin{equation}
  \label{eq:hypercube_gleichung_1}
  \xi(F) \leq M \cdot \mathrm{vol}(F)
\end{equation}
for any hypercube $F$. Because the sets in $\mathcal{P}_l$ have small Lebesgue measure for large $l$, this allows us to determine the maximum $l$ such that a set that has a too large volume with respect to $\xi$ could still be an element of $\mathcal{P}_l$. We show in Appendix \ref{app:proofs} that the assumption of essentially bounded Lebesgue densities is in fact equivalent to \mbox{Eq.\@ \eqref{eq:hypercube_gleichung_1}} holding for all hypercubes. From this, we can prove Theorem \ref{thm:wasserstein_gesamt}. Corollary \ref{cor:boundwasserstein} then is a direct consequence of Theorem \ref{thm:wasserstein_gesamt} and the remarks thereafter.

At this point, let us briefly mention that one can derive similar bounds for the Wasserstein distance between an empirical measure and its marginal distribution without the assumption of bounded Lebesgue densities if one instead assumes the underlying sample process to be $\phi$-mixing. This is because under $\phi$-mixing one is able to find a smaller bound for $\mathbb{E}|\xi_n(F) - \xi(F)|$. Explicit proofs of this are contained in Appendix \ref{app:proofs}.

\section*{Funding}
This research has been supported by the German Research Foundation DFG through the Collaborative Research Center SFB 823 \textit{Statistical Modeling of Nonlinear Dynamic Processes}, and the Research Training Group RTG 2131 \textit{High-dimensional phenomena in probability -- fluctuations and discontinuity}.

\bibliographystyle{abbrvnat}
\bibliography{bsbib}

\appendix
\section{Proofs and Auxiliary Results}
\label{app:proofs}
\subsection{Proof of Consistency and Related Results}
\label{sec:bootstrap_dcov}
For a given kernel function $g$, let us introduce the metric
$$
  d_{p,g}^p(\eta, \xi) := \inf\left\{ \int \left| g(u_1, \ldots, u_m) - g(v_1, \ldots, v_m)\right|^p ~\mathrm{d}\gamma^m\left((u_1, v_1), \ldots, (u_m, v_m)\right)\right\},
$$
where the infimum is taken over all distributions $\gamma$ with marginals $\eta$ and $\xi$.

We will prove Proposition \ref{prop:abstandvstatistik} in two steps: First, we will bound the Wasserstein distance between the $V$-statistics in terms of $d_{2,H}$; this is achieved by Theorem \ref{thm:heroldstheorem}. Thereupon, in Lemma \ref{lem:lipschitz}, we show that this bound can in turn be bounded in terms of $d_4$.

Theorem \ref{thm:heroldstheorem} is an extension of the results by \cite{dehling:1994}. It has been stated for $U$-statistics of arbitrary degree as Lemma 5.1 by \cite{dehlingprozesse}, but without an explicit proof.
\begin{theorem}
  \label{thm:heroldstheorem}
  Let $U_i \overset{iid}{\sim} \eta$, $V_i \overset{iid}{\sim} \xi$ and $g$ be a square-integrable kernel function of order $6$. Then there exists a universal constant $C$ such that
  \begin{align*}
     & d_2 \left(\mathcal{L}_\eta\left(\frac{1}{n} \sum_{1\leq i, j\leq n} g_2(U_i,U_j;\eta) \right), \mathcal{L}_\xi\left( \frac{1}{n} \sum_{1\leq i, j\leq n} g_2(V_i,V_j;\xi) \right)\right) \\
     & \quad\leq   C \cdot  \left\{d_{2,g}(\eta,\xi) + \left\|g(U_1,U_1, U_3, \ldots,U_6)-g(V_1,V_1, V_3, \ldots,V_6) \right\|_{L_2}\right\}.
  \end{align*}
\end{theorem}
\begin{proof}
  Let $\varepsilon>0$ be arbitrary. By definition of $d_{2,g}(\eta,\xi)$, we can find a distribution $\gamma$ with marginals $\eta$ and $\xi$, such that for the process of iid pairs $(U_i,V_i)_{i \in \mathbb{N}}$ with marginal distribution $\gamma$, we have
  $$
    \mathbb{E}\left[(g(U_1,\ldots,U_6)-g(V_1,\ldots,V_6))^2\right] \leq d_{2,g}^2(\eta,\xi) +\varepsilon.
  $$
  Thus, the Wasserstein distance of the two distributions
  $$
    \mathcal{L}_\eta\left(\frac{1}{n} \sum_{1\leq i, j\leq n} g_2(U_i,U_j;\eta) \right) \qquad \textrm{and} \qquad \mathcal{L}_\xi\left(\frac{1}{n} \sum_{1\leq i, j\leq n} g_2(V_i,V_j;\xi) \right)
  $$
  is bounded by the $L_2$-distance between the random variables
  $$
    \frac{1}{n} \sum_{1\leq i, j\leq n} g_2(U_i,U_j;\eta) \qquad \textrm{and} \qquad \frac{1}{n} \sum_{1\leq i, j\leq n} g_2(V_i,V_j;\xi),
  $$
  where $(U_i,V_i)$ have joint distribution $\gamma$. We thus obtain
  \begin{align}
    \begin{split}
      \label{eq:d2h-ineq_beide}
      & d_2^2 \left(\mathcal{L}_\eta\left(\frac{1}{n} \sum_{1\leq i, j\leq n} g_2(U_i,U_j;\eta) \right), \mathcal{L}_\xi\left( \frac{1}{n} \sum_{1\leq i, j\leq n} g_2(V_i,V_j;\xi) \right)  \right)\\
      & \qquad \leq
      \mathbb{E}\left[\left( \frac{1}{n} \sum_{1\leq i, j \leq n} g_2(U_i,U_j;\eta) -\frac{1}{n} \sum_{1\leq i, j \leq n}
        g_2(V_i,V_j;\xi)  \right)^2 \right] \\
      &\qquad \leq 2\mathbb{E}\left[\left( \frac{1}{n} \sum_{1\leq i \neq j \leq n} g_2(U_i,U_j;\eta) -\frac{1}{n} \sum_{1\leq i \neq j \leq n}
        g_2(V_i,V_j;\xi)  \right)^2 \right] \\
      &\qquad \qquad + 2\mathbb{E}\left[\left( \frac{1}{n} \sum_{i=1}^n g_2(U_i,U_i;\eta) -\frac{1}{n} \sum_{i=1}^n
        g_2(V_i,V_i;\xi)  \right)^2 \right] \\
      &\qquad =: 2I_1 + 2I_2. \\
    \end{split}
  \end{align}

  Let us first consider $I_1$. We have
  \begin{align}
    \begin{split}
      \label{eq:d2h-ineq}
      I_1 &= \mathbb{E}\left[\left( \frac{1}{n} \sum_{1\leq i \neq j \leq n} g_2(U_i,U_j;\eta) -\frac{1}{n} \sum_{1\leq i \neq j \leq n}
        g_2(V_i,V_j;\xi)  \right)^2 \right] \\
      & = \frac{1}{n^2}  \mathbb{E}\left[\left( \sum_{1\leq i\neq j \leq n} \left(g_2(U_i,U_j;\eta) - g_2(V_i,V_j;\xi\right) \right)^2\right]
      \\
      &  =\frac{n(n-1)}{n^2} \mathbb{E} \left[\left( g_2(U_1,U_2;\eta) - g_2(V_1,V_2;\xi)\right)^2\right] \\
      & \leq \mathbb{E} \left[\left( g_2(U_1,U_2;\eta)-g_2(V_1,V_2;\xi)  \right)^2\right],
    \end{split}
  \end{align}
  where in the penultimate step, we have used the fact that by degeneracy of the kernel $g_2$, the summands $g_2(U_i,U_j;\eta)-g_2(V_i,V_j;\xi)$ are uncorrelated. To verify this claim, we consider e.g.\@ the terms $g_2(U_1,U_2;\eta) - g_2(V_1,V_2;\xi)$ and $g_2(U_1,U_3;\eta) - g_2(V_1,V_3;\xi)$, and show that they are uncorrelated. It holds that
  \begin{align*}
     & \mathbb{E}\left[\left(g_2(U_1,U_2;\eta) - g_2(V_1,V_2;\xi)\right)\left(g_2(U_1,U_3;\eta) - g_2(V_1,V_3;\xi)\right) \right]             \\
     & \qquad = \iiint \left(g_2(u_1,u_2;\eta) - g_2(v_1,v_2;\xi)\right) \left(g_2(u_1,u_3;\eta) - g_2(v_1,v_3;\xi) \right)                   \\
     & \qquad \qquad \qquad \mathrm{d}\gamma(u_1,v_1) \mathrm{d}\gamma(u_2,v_2) \mathrm{d}\gamma(u_3,v_3)                                     \\
     & \qquad = \iint \left( \int \left(g_2(u_1,u_2;\eta)-g_2(v_1,v_2;\xi) \right)  ~\mathrm{d}\gamma(u_2,v_2) \right)                        \\
     & \qquad\qquad \qquad \cdot  \left(g_2(u_1,u_3;\eta) - g_2(v_1,v_3;\xi)\right)  ~\mathrm{d}\gamma(u_1,v_1)\mathrm{d}\gamma(u_3,v_3) = 0,
  \end{align*}
  since
  \begin{align*}
     & \int g_2(u_1,u_2;\eta)-g_2(v_1,v_2;\xi) ~\mathrm{d}\gamma(u_2,v_2)                                    \\
     & \quad = \int g_2(u_1,u_2;\eta) ~\mathrm{d}\eta(u_2) - \int g_2(v_1,v_2;\xi) ~\mathrm{d}\xi(v_2) =  0,
  \end{align*}
  due to the degeneracy of $g$. It remains to show that the right-hand side of Eq.\@ \eqref{eq:d2h-ineq} can be bounded by a constant times $d_{2,g}(\eta,\xi)$.

  Recall the definition of $g_2(u,v;\eta)$,
  \begin{align*}
     & g_2(u,v;\eta)                                                                                                                                                         \\
     & = \idotsint g(u,v,u_3,u_4,u_5,u_6) ~\mathrm{d}\eta(u_3) \mathrm{d}\eta(u_4) \mathrm{d}\eta(u_5) \mathrm{d}\eta(u_6)                                                   \\
     & \quad -\idotsint g(u,u_2,u_3,u_4,u_5,u_6) ~\mathrm{d}\eta(u_2) \mathrm{d}\eta(u_3) \mathrm{d}\eta(u_4) \mathrm{d}\eta(u_5) \mathrm{d}\eta(u_6)                        \\
     & \quad -\idotsint g(u_1,v,u_3,u_4,u_5,u_6) ~\mathrm{d}\eta(u_1) \mathrm{d}\eta(u_3) \mathrm{d}\eta(u_4) \mathrm{d}\eta(u_5) \mathrm{d}\eta(u_6)                        \\
     & \quad +\idotsint g(u_1,u_2,u_3,u_4,u_5,u_6) ~\mathrm{d}\eta(u_1) \mathrm{d}\eta(u_2) \mathrm{d}\eta(u_3) \mathrm{d}\eta(u_4) \mathrm{d}\eta(u_5) \mathrm{d}\eta(u_6).
  \end{align*}
  Thus we obtain from Minkowski's inequality
  \begin{align}
    \begin{split}
      \label{eq:d2h-ineq_hilfe_1}
      & \left\{  \mathbb{E} \left[\left(g_2(U_1,U_2;\eta) -g_2(V_1,V_2;\xi)\right)^2\right] \right\}^{1/2} \\
      & \quad \leq  \left\{  \mathbb{E}\left[\left( \idotsint g(U_1,U_2,u_3,u_4,u_5,u_6) ~\mathrm{d}\eta(u_3) \mathrm{d}\eta(u_4) \mathrm{d}\eta(u_5) \mathrm{d}\eta(u_6)\right.\right.\right. \\
        &\qquad\qquad  \left.\left.\left. - \idotsint  g(V_1,V_2,v_3,v_4,v_5,v_6) ~\mathrm{d}\xi(v_3) \mathrm{d}\xi(v_4) \mathrm{d}\xi(v_5) \mathrm{d}\xi(v_6)   \right)^2\right]  \right\}^{1/2}\\
      & \qquad + 2 \left\{ \mathbb{E}\left[\left(  \idotsint g(U_1,u_2,u_3,u_4,u_5,u_6) ~\mathrm{d}\eta(u_2) \mathrm{d}\eta(u_3) \mathrm{d}\eta(u_4) \mathrm{d}\eta(u_5) \mathrm{d}\eta(u_6)\right.\right.\right. \\
        &\qquad\qquad   \left.\left.\left.- \idotsint g(V_1,v_2,v_3,v_4,v_5,v_6) ~\mathrm{d}\xi(v_2) \mathrm{d}\xi(v_3) \mathrm{d}\xi(v_4) \mathrm{d}\xi(v_5) \mathrm{d}\xi(v_6)    \right)^2  \right]\right\}^{1/2} \\
      & \qquad + \left| \mathbb{E}g(U_1,\ldots,U_6)-\mathbb{E}g(V_1,\ldots,V_6)  \right|.
    \end{split}
  \end{align}
  We can now bound each of the terms on the right-hand side using H\"older's inequality. E.g.\@ we obtain for the first term
  \begin{align}
    \begin{split}
      \label{eq:d2h-ineq_hilfe_2}
      &  \mathbb{E}\left[\left( \idotsint g(U_1,U_2,u_3,u_4,u_5,u_6) ~\mathrm{d}\eta(u_3) \mathrm{d}\eta(u_4) \mathrm{d}\eta(u_5) \mathrm{d}\eta(u_6)\right.\right.\\
        &\quad  \left.\left. - \idotsint  g(V_1,V_2,v_3,v_4,v_5,v_6) ~\mathrm{d}\xi(v_3) \mathrm{d}\xi(v_4) \mathrm{d}\xi(v_5) \mathrm{d}\xi(v_6)   \right)^2\right] \\
      &= \iint\left( \idotsint g(u_1,u_2,u_3,u_4,u_5,u_6) ~\mathrm{d}\eta(u_3) \mathrm{d}\eta(u_4) \mathrm{d}\eta(u_5) \mathrm{d}\eta(u_6)\right.\\
      &\quad \left. - \idotsint  g(v_1,v_2,v_3,v_4,v_5,v_6) ~\mathrm{d}\xi(v_3) \mathrm{d}\xi(v_4) \mathrm{d}\xi(v_5) \mathrm{d}\xi(v_6)   \right)^2~\mathrm{d}\gamma(u_1,v_1)\mathrm{d}\gamma(u_2,v_2) \\
      & =  \iint\left( \idotsint g(u_1,u_2,u_3,u_4,u_5,u_6)  -  g(v_1,v_2,v_3,v_4,v_5,v_6) \right. \\
      &\quad \left. \vphantom{\idotsint}~\mathrm{d}\gamma(u_3,v_3) \mathrm{d}\gamma(u_4,v_4) \mathrm{d}\gamma(u_5,v_5) \mathrm{d}\gamma(u_6,v_6)   \right)^2 ~\mathrm{d}\gamma(u_1,v_1)\mathrm{d}\gamma(u_2,v_2) \\
      & \leq \mathbb{E}\left[\left(g(U_1,\ldots,U_6)-g(V_1,\ldots,V_6)  \right)^2\right] \\
      &\leq d_{2,g}^2(\eta,\xi) + \varepsilon,
    \end{split}
  \end{align}
  and $\varepsilon > 0$ is arbitrary.

  Let us now turn to $I_2$. By Jensen's inequality, we immediately get
  \begin{align*}
    I_2 & = \mathbb{E}\left[\left( \frac{1}{n} \sum_{i=1}^n g_2(U_i,U_i;\eta) -\frac{1}{n} \sum_{i=1}^n
    g_2(V_i,V_i;\xi)  \right)^2 \right]                                                                                  \\
        & \leq \mathbb{E}\left[ \frac{1}{n} \sum_{i=1}^n \left\{g_2(U_i,U_i;\eta) - g_2(V_i,V_i;\xi)  \right\}^2 \right] \\
        & = \mathbb{E}\left[\left(g_2(U_1,U_1;\eta) - g_2(V_1,V_1;\xi)\right)^2 \right].
  \end{align*}
  We can now proceed as in Eqs.\@ \eqref{eq:d2h-ineq_hilfe_1} and \eqref{eq:d2h-ineq_hilfe_2} and obtain a bound in terms of
  $$
    \mathbb{E}\left[\left(g(U_1,U_1, U_3, \ldots,U_6)-g(V_1,V_1, V_3, \ldots,V_6)  \right)^2\right].
  $$
  Thus, the statement of the theorem follows from Eq.\@ \eqref{eq:d2h-ineq_beide}.
\end{proof}

We will now prove that the bound from Theorem \ref{thm:heroldstheorem} can in turn be bounded in terms of $d_q(\eta,\xi)$ for an appropriate $q > 0$. This can be understood as a type of Hölder continuity in expectation of the kernel $g$. Because not all kernels $g$ have this property, we only show it for the kernel to which we will eventually apply this result, namely the function $H$ as defined in Eq.\@ \eqref{eq:definition_H}. If one wishes to adapt our bootstrap procedure to general kernel functions $g$, these kernel functions have to fulfil this Hölder continuity assumption.
\begin{lemma}
  \label{lem:lipschitz}
  Let $r,s\geq 2$ be given, such that $\frac{1}{r}+\frac{1}{s}=\frac{1}{2}$, and let $U_i \overset{iid}{\sim} \eta$ and $V_i \overset{iid}{\sim} \xi$. Then it holds that
  \begin{align*}
     & d_{2,H}(\eta,\xi) + \left\|H(U_1, U_1, U_3, \ldots, U_6) - H(V_1, V_1, V_3, \ldots, V_6)\right\|_{L_2} \\
     & \quad \leq 64\cdot  d_r(\eta,\xi) \left(\|U_1\|_{L_s} + \|V_1\|_{L_s}\right),
  \end{align*}
  for all measures $\eta$ and $\xi$ on $\mathbb{R}^{(\ell_1 + \ell_2)d}$. In particular, choosing $r = s = 4$ yields the bound
  $$
    64 \cdot d_4(\eta,\xi) \left(\|U_1\|_{L_4} + \|V_1\|_{L_4}\right).
  $$
\end{lemma}
\begin{remark}
  \begin{enumerate}
    \item Implicitly, we assume that the measures $\eta$ and $\xi$ have finite $r$-th moments.
    \item Lemma~\ref{lem:lipschitz} also holds for $r=2$ and $s=\infty$. In this case, i.e.\@ for measures with bounded support in $[-K,K]^{(\ell_1 + \ell_2)d}$, we obtain
          \[
            d_{s,H}(\eta,\xi)\leq 64\cdot K\cdot d_2(\eta,\xi).
          \]
  \end{enumerate}
\end{remark}
\begin{proof}[Proof of Lemma~\ref{lem:lipschitz}:] We choose $\mathbb{R}^{(\ell_1 + \ell_2)d}$-valued random vectors $\textbf{W},\textbf{W}^\prime$ with distributions $\eta$ and $\xi$, respectively, satisfying
  \[
    \mathbb{E}\|\textbf{W}-\textbf{W}^\prime\|_2^r=d_r^r (\eta,\xi).
  \]
  We denote the coordinates of $\textbf{W}$ by $W_1,\ldots, W_d$, noting that each $W_i$ is $\mathbb{R}^{(\ell_1 + \ell_2)d}$-valued random variables whose projections into $\mathbb{R}^{\ell_1}$ and $\mathbb{R}^{\ell_2}$ we again denote by $U_i$ and $V_i$, respectively. Thus, we have
  \begin{align*}
    \textbf{W}        & =(U_1,V_1,\ldots, U_d,V_d)                              \\
    \textbf{W}^\prime & =(U_1^\prime,V_1^\prime,\ldots, U_d^\prime,V_d^\prime).
  \end{align*}
  Let $(\textbf{W}_1,\textbf{W}^\prime_1),\ldots, (\textbf{W}_6,\textbf{W}_6^\prime)$ be independent copies of $(\textbf{W},\textbf{W}^\prime)$. By definition of $d_{2,H}(\eta,\xi)$, we obtain
  \begin{align*}
     & d_{2,H}(\eta,\xi) \leq \left\| H(\textbf{W}_1,\ldots, \textbf{W}_6)-H(\textbf{W}_1^\prime,\ldots, \textbf{W}_6^\prime)   \right\|_{L_2}                     \\
     & \quad\leq \frac{1}{d^6} \sum_{1\leq i_1,\ldots,i_6\leq d} \left\| h(W_{1,i_1},\ldots,W_{6,i_6}) - h(W_{1,i_1}^\prime,\ldots,W_{6,i_6}^\prime)\right\|_{L_2} \\
     & \quad\leq \frac{1}{d^6\, 6!} \sum_{1\leq i_1,\ldots,i_6\leq d} \sum_{\sigma\in \mathcal{S}_6}
    \left\| h'\left(W_{\sigma(1),i_{\sigma(1)}},\ldots,W_{\sigma(6),i_{\sigma(6)}}\right) - h'\left(W_{\sigma(1),i_{\sigma(1)}}^\prime,\ldots,W_{\sigma(6),i_{\sigma(6)}}^\prime\right)\right\|_{L_2}.
  \end{align*}
  Each of the summands on the right-hand side is of the type
  \begin{equation}\label{eq:lemma1-basic}
    \left\| h'(W_{i_1,j_1},\ldots,W_{i_6,j_6}) - h'(W_{i_1,j_1}^\prime,\ldots,W_{i_6,j_6}^\prime)\right\|_{L_2},
  \end{equation}
  where $(i_1,\ldots,i_6)$ is a permutation of $(1,\ldots,6)$, and where $1\leq j_1,\ldots,j_6\leq d$. We will give upper bounds for each of these summands in the more general case that $1 \leq i_1, \ldots, i_6 \leq 6$. Note that
  \begin{align*}
     & h'(W_{i_1,j_1},\ldots,W_{i_6,j_6}) - h'(W_{i_1,j_1}^\prime,\ldots,W_{i_6,j_6}^\prime)                                                                                                                                                        \\
     & \; = f(U_{i_1,j_1},U_{i_2,j_2},U_{i_3,j_3},U_{i_4,j_4})f(V_{i_1,j_1},V_{i_2,j_2},V_{i_5,j_5}, V_{i_6,j_6})                                                                                                                                   \\
     & \quad - f(U_{i_1,j_1}^\prime,U_{i_2,j_2}^\prime,U_{i_3,j_3}^\prime,U_{i_4,j_4}^\prime)f(V_{i_1,j_1}^\prime,V_{i_2,j_2}^\prime,V_{i_5,j_5}^\prime, V_{i_6,j_6}^\prime)                                                                        \\
     & \;=  f(U_{i_1,j_1},U_{i_2,j_2},U_{i_3,j_3},U_{i_4,j_4})\left(f(V_{i_1,j_1},V_{i_2,j_2},V_{i_5,j_5}, V_{i_6,j_6})
    -f(V_{i_1,j_1}^\prime,V_{i_2,j_2}^\prime,V_{i_5,j_5}^\prime, V_{i_6,j_6}^\prime)\right)                                                                                                                                                         \\
     & \quad + f(V_{i_1,j_1}^\prime,V_{i_2,j_2}^\prime,V_{i_5,j_5}^\prime, V_{i_6,j_6}^\prime)\left(f(U_{i_1,j_1},U_{i_2,j_2},U_{i_3,j_3},U_{i_4,j_4}) -   f(U_{i_1,j_1}^\prime,U_{i_2,j_2}^\prime,U_{i_3,j_3}^\prime,U_{i_4,j_4}^\prime) \right) .
  \end{align*}
  We can now apply H\"older's inequality, and obtain
  \begin{align}
    \begin{split}
      \label{eq:fviereck}
      &\left\|f(U_{i_1,j_1},U_{i_2,j_2},U_{i_3,j_3},U_{i_4,j_4})\right. \\
      &\qquad \cdot\left.\left(f(V_{i_1,j_1},V_{i_2,j_2},V_{i_5,j_5}, V_{i_6,j_6})
      -f(V_{i_1,j_1}^\prime,V_{i_2,j_2}^\prime,V_{i_5,j_5}^\prime, V_{i_6,j_6}^\prime)\right) \right\|_{L_2} \\
      & \leq \left\|f(U_{i_1,j_1},U_{i_2,j_2},U_{i_3,j_3},U_{i_4,j_4})\right\|_{L_s} \\
      &\qquad\cdot   \left\|f(V_{i_1,j_1},V_{i_2,j_2},V_{i_5,j_5}, V_{i_6,j_6})
      -f(V_{i_1,j_1}^\prime,V_{i_2,j_2}^\prime,V_{i_5,j_5}^\prime, V_{i_6,j_6}^\prime) \right\|_{L_r}
    \end{split}
  \end{align}
  Note that by the triangle inequality it holds that
  \begin{align*}
     & |f(V_{i_1, j_1},  V_{i_2, j_2}, V_{i_3, j_3}, V_{i_4, j_4}) - f(V'_{i_1, j_1},  V'_{i_2, j_2}, V'_{i_3, j_3}, V'_{i_4, j_4})|                                              \\
     & \quad \leq 2\left(\|V_{i_1, j_1} - V'_{i_1, j_1}\|_2 + \|V_{i_2, j_2} - V'_{i_2, j_2}\|_2 + \|V_{i_3, j_3} - V'_{i_3, j_3}\|_2 + \|V_{i_4, j_4} - V'_{i_4, j_4}\|_2\right) \\
     & \quad \leq 8 \|\textbf{W} - \textbf{W}'\|_2,
  \end{align*}
  and thus
  $$
    \|f(V_{i_1, j_1},  V_{i_2, j_2}, V_{i_3, j_3}, V_{i_4, j_4}) - f(V'_{i_1, j_1},  V'_{i_2, j_2}, V'_{i_3, j_3}, V'_{i_4, j_4})\|_{L_r} \leq 8\|\textbf{W} - \textbf{W}'\|_{L_r}.
  $$
  Similarly, one shows that
  \begin{align*}
    \|f(U_{i_1, j_1}, U_{i_2, j_2}, U_{i_3, j_3}, U_{i_4, j_4})\|_{L_s} & \leq 4 \|U_{i_1, j_1}\|_{L_s} = 4 \left( \int \|x\|_2^s ~\mathrm{d}\eta(x)\right)^{1/s}.
  \end{align*}

  Combining the above inequalities, we can bound Eq.\@ \eqref{eq:fviereck} by
  $$
    32 \left(\int\|x\|_2^s ~\mathrm{d}\eta(x)\right)^{1/s} d_r(\eta,\xi),
  $$
  and thus
  $$
    d_{2,H}(\eta, \xi) \leq 32 d_r(\eta, \xi) \left\{ \left(\int\|x\|_2^s ~\mathrm{d}\eta(x)\right)^{1/s} + \left(\int\|y\|_2^s ~\mathrm{d}\xi(y)\right)^{1/s}\right\}.
  $$

  To bound $\left\|H(U_1, U_1, U_3, \ldots, U_6) - H(V_1, V_1, V_3, \ldots, V_6)\right\|_{L_2}$, we again need to find a bound for Eq.\@ \eqref{eq:lemma1-basic}, with the difference being that $(i_1, \ldots, i_6)$ is now not a permutation of $(1, \ldots, 6)$, but rather a permutation of $(1, 1, 3, \ldots, 6)$.  But because we have deduced the bounds above not only in the case where $(i_1, \ldots, i_6)$ is a permutation of $(1, \ldots, 6)$ but rather any collection of indices between $1$ and $6$, we arrive at the same bound as before, i.e.\@
  \begin{align*}
     & \left\|H(U_1, U_1, U_3, \ldots, U_6) - H(V_1, V_1, V_3, \ldots, V_6)\right\|_{L_2}                                                                     \\
     & \quad \leq 32 d_r(\eta, \xi) \left\{ \left(\int\|x\|_2^s ~\mathrm{d}\eta(x)\right)^{1/s} + \left(\int\|y\|_2^s ~\mathrm{d}\xi(y)\right)^{1/s}\right\}.
  \end{align*}
  This proves the lemma.
\end{proof}

Proposition \ref{prop:abstandvstatistik} now follows from Theorem \ref{thm:heroldstheorem} and Lemma \ref{lem:lipschitz}.

\begin{lemma}
  \label{lem:metrikprodsumme}
  Let $\eta_i$ and $\xi_i$ be measures with finite $p$-moments on $\mathbb{R}^{\ell_i}$, $i = 1,2$. Then,
  $$
    d_p^p(\eta_1 \otimes \eta_2, \xi_1 \otimes \xi_2) \leq \max\left\{1, 2^{p/2-1}\right\}\left(d_p^p(\eta_1, \xi_1) + d_p^p(\eta_2, \xi_2)\right).
  $$
\end{lemma}
\begin{proof}
  For any $z, z' \in \mathbb{R}^{\ell_1 + \ell_2}$, it holds that
  $$
    \|z - z'\|_2^p \leq \max\left\{1, 2^{p/2-1}\right\}\left(\|x - x'\|_2^p + \|y - y'\|_2^p\right),
  $$
  where $x$ (or $x'$) and $y$ (or $y'$) denote the corresponding projections on $\mathbb{R}^{\ell_1}$ and $\mathbb{R}^{\ell_2}$, respectively. Let $\Gamma$ denote the set of all couplings of $\eta_1 \otimes \eta_2$ and $\xi_1 \otimes \xi_2$, and $\Gamma_i$ the set of all couplings of $\eta_i$ and $\xi_i$, $i = 1,2$.

  For any given measure $\gamma$, we say that two given projections $\pi_1$ and $\pi_2$ are independent in $\gamma$, if the pushforward of $\gamma$ under $(\pi_1, \pi_2)$ is equal to the product measure of the individual pushforwards, i.e.\@
  $$
    \gamma^{(\pi_1, \pi_2)} = \gamma^{\pi_1} \otimes \gamma^{\pi_2}.
  $$
  Any $\gamma \in \Gamma$ operates on $\mathbb{R}^{\ell_1} \times \mathbb{R}^{\ell_2} \times \mathbb{R}^{\ell_1} \times \mathbb{R}^{\ell_2}$, which we will associate with the four projections $\pi_i$ and $\tau_i$, $i = 1,2$, such that $\gamma^{\pi_i} = \eta_i$ and $\gamma^{\tau_i} = \xi_i$ for $i = 1,2$. Because $\Gamma$ is the set of all couplings of $\eta_1 \otimes \eta_2$ and $\xi_1 \otimes \xi_2$, it holds that $\pi_1$ and $\pi_2$ are independent in $\gamma$, for any $\gamma \in \Gamma$. The same is true for $\tau_1$ and $\tau_2$. Finally, let $\Gamma'$ be the subset of all measures $\gamma \in \Gamma$, such that $\pi_1$ and $\tau_2$ are independent in $\gamma$ and $\pi_2$ and $\tau_1$ are independent in $\gamma$. For such $\gamma \in \Gamma'$, the only pairs of projections that are not independent in $\gamma$ are $\pi_1$ and $\tau_1$, as well as $\pi_2$ and $\tau_2$. This means that
  $$
    \Gamma' = \{\gamma_1 \otimes \gamma_2 ~|~ \gamma_i \in \Gamma_i, i = 1,2\}.
  $$
  Therefore it holds that
  \begin{align*}
    d_p^p & (\eta_1 \otimes \eta_2, \xi_1 \otimes \xi_2) = \inf_{\gamma \in \Gamma} \int \|z - z'\|_2^p ~\mathrm{d}\gamma(z,z')                                                                                    \\
          & \leq \max\left\{1, 2^{p/2-1}\right\} \inf_{\gamma \in \Gamma} \left\{ \int \|x - x'\|_2^p ~\mathrm{d}\gamma(z,z') + \int \|y - y'\|_2^p ~\mathrm{d}\gamma(z,z') \right\}                               \\
          & \leq \max\left\{1, 2^{p/2-1}\right\} \inf_{\gamma \in \Gamma'} \left\{ \int \|x - x'\|_2^p ~\mathrm{d}\gamma(z,z') + \int \|y - y'\|_2^p ~\mathrm{d}\gamma(z,z') \right\}                              \\
          & = \max\left\{1, 2^{p/2-1}\right\} \left(\inf_{\gamma_1 \in \Gamma_1} \int \|x - x'\|_2^p ~\mathrm{d}\gamma_1(x,x') + \inf_{\gamma_2 \in \Gamma_2} \int \|y - y'\|_2^p ~\mathrm{d}\gamma_2(y,y')\right) \\
          & = \max\left\{1, 2^{p/2-1}\right\} \left(d_p^p(\eta_1, \xi_1) + d_p^p(\eta_2, \xi_2)\right).
  \end{align*}
\end{proof}

\begin{lemma}
  \label{lem:hoeffdinghH}
  If $(X_k)_{k \in \mathbb{N}}$ and $(Y_k)_{k \in \mathbb{N}}$ are independent, then for any $r \in \{0, 1, \ldots, 6\}$, the following two identities hold:
  \begin{align*}
    H_r\left(B_1, \ldots, B_r;F^{(X)} \otimes F^{(Y)}\right)      & = d^{-r}\sum_{1\leq i_1, \ldots, i_r \leq d} h_r\left(B_{1, i_1}, \ldots, B_{r, i_r} ; \theta\right),             \\
    H_r\left(B_1, \ldots, B_r; F_N^{(X)} \otimes F_N^{(Y)}\right) & = d^{-r}\sum_{1\leq i_1, \ldots, i_r \leq d}  h_r\left(B_{1, i_1}, \ldots, B_{r, i_r} ; \mu_n\otimes\nu_n\right),
  \end{align*}
  where $B_{i,j}$ denotes the $j$-th element from the $i$-th block.
\end{lemma}
\begin{proof}
  We prove the claim for $r = 2$, noting that the other cases can be proven analogously. Let $F := F^{(X)} \otimes F^{(Y)}$. Using the linearity of the integral and the definition of $H$, we get that
  \begin{alignat*}{2}
    H_2(B_1, B_2;F) & = \frac{1}{d^6}\sum_{1 \leq i_1, \ldots, i_6 \leq d} &  & \left\{ \int h(B_{1,i_1}, B_{2,i_2}, B'_{3,i_3} \ldots, B'_{6,i_6}) ~\mathrm{d}F^4(B'_3, \ldots, B'_6) \right. \\
                    &                                                      &  & \quad- \int h(B_{1,i_1}, B'_{2,i_2} \ldots, B'_{6,i_6}) ~\mathrm{d}F^5(B'_2, \ldots, B'_6)                     \\
                    &                                                      &  & \quad-  \int h(B_{2,i_1}, B'_{2,i_2} \ldots, B'_{6,i_6}) ~\mathrm{d}F^5(B'_2, \ldots, B'_6)                    \\
                    &                                                      &  & \quad + \left. \int h(B'_{1,i_1}, \ldots, B'_{6,i_6}) ~\mathrm{d}F^6(B'_1, \ldots, B'_6) \right\}              \\
                    & = \frac{1}{d^6}\sum_{1 \leq i_1, \ldots, i_6 \leq d} &  & \left\{\int h(B_{1,i_1}, B_{2,i_2}, z_3 \ldots, z_6) ~\mathrm{d}\theta^4(z_3, \ldots, z_6)\right.              \\
                    &                                                      &  & \quad - \int h(B_{1,i_1}, z_2 \ldots, z_6) ~\mathrm{d}\theta^5(z_2, \ldots, z_6)                               \\
                    &                                                      &  & \quad - \int h(B_{2,i_1}, z_2 \ldots, z_6) ~\mathrm{d}\theta^5(z_2, \ldots, z_6)                               \\
                    &                                                      &  & \quad + \left.\int h(z_1 \ldots, z_6) ~\mathrm{d}\theta^6(z_1, \ldots, z_6)\right\} .
  \end{alignat*}
  Note that the summands in the last sum are equal to $h_2(B_{1, i_1}, B_{2, i_2};\theta)$, and thus the first identity is proven. What we have used here is the fact that in every summand, each block $B_j'$ has an effect only through a single coordinate $B'_{j, i_j}$, which has marginal distribution $\mu \otimes \nu = \theta$. Because there is exactly one such coordinate per block in every summand, we need not worry about the dependence between the coordinates.

  The second identity can be shown analogously, noting that taking the mean over $i_3, \ldots, i_6$ gives us as marginals $\mu_n \otimes \nu_n$.
\end{proof}

The following lemma is a slight extension of Lemma 3 by \cite{arcones}, where it is stated for $U$-statistics. Since we derive the result for $V$-statistics, we include a full proof for the sake of readability.
\begin{lemma}
  \label{lem:arconeslemma}
  Let $(U_i)_{i \in \mathbb{N}}$ be a strictly stationary and absolutely regular process and $g$ be a symmetric and degenerate kernel of order $m$, i.e.\@ $\mathbb{E}g(U_1, u_2, \ldots, u_m) = 0$ almost surely. Furthermore, assume that for some $p > 2$ it holds that
  $$
    \left\|g(U_{i_1}, \ldots, U_{i_m})\right\|_{L_p} < M
  $$
  for some $M$ uniform in $i_1, \ldots, i_m$. Then it holds that
  $$
    \mathbb{E}\left[\left(\sum_{1 \leq i_1, \ldots, i_m \leq n} g(U_{i_1}, \ldots, U_{i_m})\right)^2\right] \leq 8M^2m(2m)! \cdot n^m \left\{1 + \sum_{d=1}^n d^{m-1} \beta(d)^{(p-2)/p}\right\}.
  $$
\end{lemma}
\begin{proof}
  Lemma 3 in \cite{arcones} gives a proof in the case of $U$-statistics. By their method of proof we consider
  \begin{align*}
     & \mathbb{E}\left[\left(\sum_{1 \leq i_1, \ldots, i_m \leq n} g(U_{i_1}, \ldots, U_{i_m})\right)^2\right]                                                                                              \\
     & =\sum_{1 \leq i_1, \ldots, i_{2m} \leq n} \mathbb{E}[g(U_{i_1}, \ldots, U_{i_m})g(U_{i_{m+1}}, \ldots, U_{i_{2m}})]                                                                                  \\
     & = \sum_{\sigma \in \mathcal{S}_{2m}} \sum_{1 \leq i_1 \leq \ldots \leq i_{2m} \leq n} \mathbb{E}[g(U_{i_{\sigma(1)}}, \ldots, U_{i_{\sigma(m)}})g(U_{i_{\sigma(m+1)}}, \ldots, U_{i_{\sigma(2m)}})].
  \end{align*}
  Using the definition
  $$
    I_l := \left\{ (i_1, \ldots, i_{2m}) ~|~ 1 \leq i_1 \leq \ldots \leq i_{2m} \leq n \land \#\{i_1, \ldots, i_{2m}\} = l\right\},
  $$
  we can write the sum above as
  \begin{equation}
    \label{eq:lsummesymmetrisch}
    \sum_{l=1}^{2m} \sum_{\sigma \in \mathcal{S}_{2m}}\sum_{(i_1, \ldots, i_{2m}) \in I_l} \mathbb{E}[g(U_{\sigma(i_1)}, \ldots, U_{\sigma(i_m)})g(U_{\sigma(i_{m+1})}, \ldots, U_{\sigma(i_{2m})})].
  \end{equation}
  We will bound this object using different methods depending on $l$.

  By using the Cauchy-Schwarz inequality and the monotony of the $L_p$-norm, the expected value appearing in Eq.\@ \eqref{eq:lsummesymmetrisch} can be bounded by $M^2$. Furthermore, if $1 \leq l \leq m$, $\#I_l$ can be bounded by $n^m$. Thus,
  \begin{align}
    \begin{split}
      \label{eq:lsummesymmetrisch2}
      &\sum_{l=1}^{m} \sum_{\sigma \in \mathcal{S}_{2m}}\sum_{(i_1, \ldots, i_{2m}) \in I_l} \mathbb{E}[g(U_{\sigma(i_1)}, \ldots, U_{\sigma(i_m)})g(U_{\sigma(i_{m+1})}, \ldots, U_{\sigma(i_{2m})})] \\
      &\quad \leq M^2 m (2m)! \cdot n^m.
    \end{split}
  \end{align}
  Next, suppose that $(i_1, \ldots, i_{2m}) \in I_l$ for some $m < l \leq 2m$. Define
  $$
    j_k := \begin{cases} i_2 - i_1                                      & \textrm{for } k = 1             \\
              \min\{i_{2k-1} - i_{2k-2}, i_{2k} - i_{2k-1}\} & \textrm{for } 2 \leq k \leq m-1 \\
              i_{2m} - i_{2m-1}                              & \textrm{for } k=m
    \end{cases}
  $$
  We want to make use of Lemma 2 in \cite{arcones}, which allows us to compare the expected value $\mathbb{E}[g(U_{i_{\sigma(1)}}, \ldots, U_{i_{\sigma(m)}})g(U_{i_{\sigma(m+1)}}, \ldots, U_{i_{\sigma(2m)}})]$ with the integral
  \begin{align*}
     & \int g(u_{\sigma(1)}, \ldots, u_{\sigma(m)})g(u_{\sigma(m+1)}, \ldots, u_{\sigma(2m)})                                                                                                          \\
     & \qquad\mathrm{d}\left(\mathcal{L}(U_{i_1}, \ldots, U_{i_s}) \otimes \mathcal{L}(U_{i_{s+1}}, \ldots, U_{i_t}) \otimes \mathcal{L}(U_{i_{t+1}}, \ldots, U_{i_{2m}})\right)(u_1, \ldots, u_{2m}),
  \end{align*}
  where $s$ and $t$ are chosen such that $1 \leq i_s < i_{s+1} \leq i_t < i_{t+1} \leq 2m$, i.e.\@ we separate the original sequence into three independent blocks. Because of the degeneracy of $g$ and the specific choice of our blocks, this will allow us to bound $\mathbb{E}[g(U_{i_{\sigma(1)}}, \ldots, U_{i_{\sigma(m)}})g(U_{i_{\sigma(m+1)}}, \ldots, U_{i_{\sigma(2m)}})]$ itself.

  More precisely, in this situation, we want to determine some number $K$ for which it holds that $j_K = \max\{j_1, \ldots, j_m\}$ and then separate the original sequence $U_{i_1}, \ldots, U_{i_{2m}}$ into three blocks, $\{U_{i_1}, \ldots, U_{i_{2K-2}}\}$, $\{U_{i_{2K-1}}\}$ and $\{U_{i_{2K}}, \ldots, U_{i_{2m}}\}$. If the maximum is achieved at $j_1$ or $j_m$, we instead separate the sequence into two blocks, $\{U_{i_1}\}$ and $\{U_{i_2}, \ldots, U_{i_{2m}}\}$, or $\{U_{i_1}, \ldots, U_{i_{2m-1}}\}$ and $\{U_{i_{2m}}\}$, respectively.

  However, for this method to work, we need to ensure that $i_{2K-1}$ is distinct from both $i_{2K-2}$ and $i_{2K}$ if $2 \leq K \leq m-1$; that $i_1$ is distinct from $i_2$ if $K = 1$; or that $i_{2m-1}$ is distinct from $i_{2m}$ if $K = m$. Equivalently, one needs to show that $j_K > 0$.

  Because $j_K = \max\{j_1, \ldots, j_m\}$, it suffices to show that under our assumptions, there is some $1 \leq k \leq m$ such that $j_k > 0$. Suppose that this is not true, i.e.\@ $j_k = 0$ for all $1 \leq k \leq m$. This implies
  \begin{align*}
     & i_1 = i_2,                                                                          \\
     & i_{2k-1} = i_{2k-2} ~\lor~ i_{2k} = i_{2k-1} \quad \textrm{for } 2 \leq k \leq m-1, \\
     & i_{2m} = i_{2m - 1}.
  \end{align*}
  Under this set of constraints, $\#\{i_1, \ldots, i_{2m}\}$ is maximised if for any $2 \leq k \leq m-1$ either $i_{2k - 1} = i_{2k-2}$ or $i_{2k} = i_{2k-1}$ hold, but not both, and $i_{2m-1} \neq i_{2m-2}$. In this case, $\#\{i_1, \ldots, i_{2m}\} = m$. This can be seen via induction over $m$. Therefore, $j_K = 0$ implies $\#\{i_1, \ldots, i_{2m}\} \leq m$; but we have assumed that $(i_1, \ldots, i_{2m}) \in I_l$ for some $m < l \leq 2m$, and so $j_K > 0$.

  Let us now formalise the `comparison' of integrals described above. Suppose that we have the equality $j_K = \max\{j_1, \ldots, j_m\}$, then by Lemma 2 in \cite{arcones}, the degeneracy of $g$ and the observations above, it holds that
  $$
    |\mathbb{E}[g(U_{i_{\sigma(1)}}, \ldots, U_{i_{\sigma(m)}})g(U_{i_{\sigma(m+1)}}, \ldots, U_{i_{\sigma(2m)}})]| \leq 8M^2 \beta(j_K)^{(p-2)/p}.
  $$
  Now if $K = 1$ and $j_K = j_1 = d$, then there are $n-d \leq n$ possible values for $i_1$ which also determine $i_2$. Next, because $j_1 \geq j_2, \ldots, j_m$, it holds that
  $$
    j_2 = \min\{i_3 - i_2, i_4 - i_3\} \leq d.
  $$
  If $i_4 - i_3 \leq i_3 - i_2$, then $i_4 - i_3 \leq d$ and thus there are at most $d$ possible values for $i_4$ and at most $n$ possible values for $i_3$ (this is a somewhat trivial bound, as there are at most $n$ possible values for any index. The important part is that we can bound the number of possible values for one of the two indices, in this case $i_4$, by something other than $n$). Conversely, if $i_3 - i_2 \leq i_4 - i_3$, then there are at most $d$ possible values for $i_3$ and at most $n$ possible values for $i_4$. In general, it holds for any $2 \leq k \leq m$ that there are at most $nd$ possible values for $(i_{2k-1}, i_{2k})$ and thus the number of all ordered collections of indices $(i_1, \ldots, i_{2m})$ such that $j_1 \geq j_2, \ldots, j_m$ is bounded by $n^m d^{m-1}$. The same method can be employed for any $1 \leq K \leq m$, and thus we obtain that
  \begin{align}
    \begin{split}
      \label{eq:lsummesymmetrisch3}
      &\sum_{l = m+1}^{2m} \sum_{\sigma \in \mathcal{S}_{2m}}\sum_{(i_1, \ldots, i_{2m}) \in I_l} \mathbb{E}[g(U_{\sigma(i_1)}, \ldots, U_{\sigma(i_m)})g(U_{\sigma(i_{m+1})}, \ldots, U_{\sigma(i_{2m})})] \\
      &\quad \leq 8M^2 m(2m)! \cdot n^m \sum_{d=0}^n d^{m-1} \beta(d)^{(p-2)/p}.
    \end{split}
  \end{align}
  Combining Eqs.\@ \eqref{eq:lsummesymmetrisch2} and \eqref{eq:lsummesymmetrisch3}, we can bound Eq.\@ \eqref{eq:lsummesymmetrisch} by
  $$
    8M^2m(2m)!\cdot n^m \left\{1 + \sum_{d=1}^n d^{m-1} \beta(d)^{(p-2)/p}\right\},
  $$
  which proves the lemma.
\end{proof}

\begin{corollary}
  \label{cor:arconescor}
  Under the conditions of Lemma \ref{lem:arconeslemma}, if $\beta(n) = \mathcal{O}\left(n^{-r}\right)$ for some growth rate $r > mp/(p-2)$, it holds that
  $$
    \mathbb{E}\left[\left(\sum_{1 \leq i_1, \ldots, i_m \leq n} g(U_{i_1}, \ldots, U_{i_m})\right)^2\right] \leq 8M^2m(2m)! \cdot n^m \left\{1 + \sum_{d=1}^\infty d^{m-1} \beta(d)^{(p-2)/p}\right\},
  $$
  and the series on the right-hand side converges. Note that the bound on the right-hand side may be written as $\mathrm{const} \cdot n^m$, where $\mathrm{const}$ only depends on the growth rate of the mixing coefficients, the degree of the kernel $g$ and the uniform bound $M$.
\end{corollary}
\begin{proof}
  By assumption on the mixing coefficients, there is some constant $c > 0$ such that $\beta(n) \leq c n^{-r}$ for all $n \in \mathbb{N}$. Therefore,
  $$
    \sum_{d=1}^n d^{m-1} \beta(d)^{(p-2)/p} \leq c \sum_{d=1}^n d^{m-1} d^{-r(p-2)/p} =  c \sum_{d=1}^n d^{m-1 - r(p-2)/p},
  $$
  and because $r > p/(p-2)$ by assumption, it holds that $m-1 - r(p-2)/p < -1$. Thus,
  $$
    \sum_{d=1}^n d^{m-1} \beta(d)^{(p-2)/p} \leq \sum_{d=1}^\infty d^{m-1} \beta(d)^{(p-2)/p} < \infty,
  $$
  and so Lemma \ref{lem:arconeslemma} yields the desired result.
\end{proof}

\begin{lemma}
  \label{lem:hoeffdingdominiert}
  Let $Z_1$ have finite $(2+\delta)$-moments for some $\delta > 0$. If $(X_k)_{k \in \mathbb{N}}$ and $(Y_k)_{k \in \mathbb{N}}$ are independent, it holds that
  $$
    \mathbb{E}\left[\left(nV - 15nV_n^{(2)}(h;\theta)\right)^2\right] \xrightarrow[n \to \infty]{} 0.
  $$
  If we furthermore assume that $d^3 = o(n)$, it holds that
  $$
    \mathbb{E}\left[\left(n\tilde{V} - 15n\tilde{V}_n^{(2)}(h;\theta)\right)^2\right] \xrightarrow[n \to \infty]{} 0,
  $$
  and
  $$
    \mathbb{E}\left[\left(nV^* - 15nV_n^{*(2)}(h;\mu_n \otimes \nu_n)\right)^2\right] \xrightarrow[n \to \infty]{} 0.
  $$
\end{lemma}
\begin{proof}
  We first consider
  $$
    \tilde{V}_n^{(r)}(h;\theta) = \tilde{V}_N^{(r)}\left(H;F^{(X)} \otimes F^{(Y)}\right),
  $$
  where the equality holds due to Lemma \ref{lem:hoeffdinghH}. The $n$-th row of the triangular array $\tilde{Z}$ can be taken as the first $N$ observations of an iid process $(\tilde{B}_k)_{k \in \mathbb{N}}$ with marginal distribution $\mathcal{L}(Z_1, \ldots, Z_d)$. By Corollary \ref{cor:arconescor}, we only need to find a uniform bound for
  $$
    \mathbb{E}\left[H\left(\tilde{B}_{i_1}, \ldots, \tilde{B}_{i_6}\right)^{2+\delta}\right].
  $$
  By definition of $H$, this is possible if we can instead find a uniform bound for
  \begin{align*}
    \mathbb{E}\left[h\left(\tilde{Z}_{i_1}, \ldots, \tilde{Z}_{i_6}\right)^{2+\delta}\right],
  \end{align*}
  which, because $\theta = \mu \otimes \nu$, in turn follows if we can find a finite bound for
  $$
    \mathbb{E}\left[f_{\ell_1}\left(\tilde{X}_{i_1}, \ldots, \tilde{X}_{i_4}\right)^{2+\delta}\right]
  $$
  for any $1 \leq i_1, \ldots, i_4 \leq n$; analogous arguments can then be used for $f_{\ell_2}$. Recall the definition of $f_{\ell_1}$ from Section \ref{sec:test_statistic_and_bs_technique}. By the triangle inequality, it holds that
  \begin{equation}
    \label{eq:fdreieck}
    f_{\ell_1}(x_1, \ldots, x_4) \leq  2\|x_2 - x_3\|_2 \land 2\|x_1 - x_4\|_2
  \end{equation}
  for all $x_1, \ldots, x_4$. The uniform bound then follows because $Z_1$ has finite $(2+\delta)$-moments by assumption.

  Thus, by Corollary \ref{cor:arconescor}, it holds for all $2 \leq r \leq 6$ that
  $$
    \mathbb{E}\left[\tilde{V}_n^{(r)}(h;\theta)^2\right] = \mathcal{O}\left(N^{-r}\right),
  $$
  with the constant involved being depending only on $r$ and the uniform bound just derived. Therefore,
  \begin{align*}
    \mathbb{E}\left[\left(n\tilde{V} - 15n\tilde{V}_n^{(2)}(h;\theta)\right)^2\right] & = n^2\mathbb{E}\left[\left(\sum_{i=3}^6 {6 \choose i} V_n^{(i)}(h;\theta)\right)^2\right] \\
                                                                                      & = n^2\mathcal{O}\left(N^{-3}\right) = \mathcal{O}\left(\frac{d^3}{n}\right),
  \end{align*}
  and the bound on the right-hand side is $o(1)$ by assumption.

  One can show that
  $$
    \mathbb{E}\left[\left(nV - 15nV_n^{(2)}(h;\theta)\right)^2\right] \xrightarrow[n \to \infty]{} 0
  $$
  in much the same way, even without the assumption that $d^3 = o(n)$.

  Now let us consider
  $$
    V_n^{*(r)}\left(h;\mu_n \otimes \nu_n\right) = V_N^{*(r)}\left(H;F_N^{(X)} \otimes F_N^{(Y)}\right).
  $$
  After conditioning on $Z_1, \ldots, Z_n$, we can once more use Corollary \ref{cor:arconescor}. However, we now need to consider the expected values conditional on $Z_1, \ldots, Z_n$. Let $j_1, \ldots, j_6$ be any collection of indices. Then,
  \begin{align}
    \begin{split}
      \label{eq:xir_bootstrap}
      &\mathbb{E}\left[H(B_{j_1}^*, \ldots, B_{j_6}^*)^{2+\delta} ~\Big|~ Z_1, \ldots, Z_n\right] \\
      &\quad= \mathbb{E}\left[\left(d^{-6}\sum_{1\leq i_1, \ldots, i_6 \leq d} h\left(B^*_{j_1, i_1}, \ldots, B^*_{j_6, i_6}\right)\right)^{2+\delta} ~\Big|~ Z_1, \ldots, Z_n\right] \\
      &\quad\leq d^{-6} \sum_{1 \leq i_1, \ldots, i_6 \leq d} \mathbb{E}\left[h\left(B^*_{j_1, i_1}, \ldots, B^*_{j_6, i_6} \right)^{2+\delta} ~\Big|~ Z_1, \ldots, Z_n\right] \\
    \end{split}
  \end{align}
  Applying the Jensen inequality once more shows that the $(2+\delta)$-th power of $h$, which by definition is the symmetrisation of $h'$, is bounded by the symmetrisation of the $(2+\delta)$-th powers of $h'$. This only changes the order of the arguments of $h'$, so we will have proven our claim if we can show that
  \begin{equation}
    \label{eq:xir_bootstrap2}
    d^{-6}\sum_{1 \leq i_1, \ldots, i_6 \leq d} \mathbb{E}\left[ h'\left(B_{j_1, i_1}^*, \ldots, B_{j_6, i_6}^*\right)^{2+\delta} ~\Big|~ Z_1, \ldots, Z_n\right]
  \end{equation}
  is uniformly bounded for any collection of indices $1 \leq j_1, \ldots, j_6 \leq N$. Using Eq.\@ \eqref{eq:fdreieck}, it holds for any $z_1, \ldots, z_6$ that
  $$
    |h'(z_1, \ldots, z_6)| \leq 4\left(\sum_{k=1}^6 \|x_k\|_2\right)\left(\sum_{k=1}^6 \|y_k\|_2\right),
  $$
  Now, because the bootstrapped $X$-blocks and $Y$-blocks, denoted for the moment by $B_k^{*(X)}$ and $B_k^{*(Y)}$, form two independent iid sequences, Eq.\@ \eqref{eq:xir_bootstrap2} can be bounded by
  \begin{align*}
     & 4^{2+\delta}d^{-6}\sum_{1 \leq i_1, \ldots, i_6 \leq d} \left\{\mathbb{E}\left[ \left(\sum_{k=1}^6 \left\|B_{j_k, i_k}^{*(X)}\right\|_2\right)^{2+\delta} ~\Big|~ X_1, \ldots, X_n\right]\right.                                                                        \\
     & \phantom{4^{2+\delta}d^{-6}\sum_{1 \leq i_1, \ldots, i_6 \leq d}}\quad\left.\cdot \mathbb{E}\left[ \left(\sum_{k=1}^6 \left\|B_{j_k, i_k}^{*(Y)}\right\|_2\right)^{2+\delta} ~\Big|~ Y_1, \ldots, Y_n\right] \right\}                                                   \\
     & \quad \leq 4^{2+\delta} 6^{2(1+\delta)} d^{-6}\sum_{1 \leq i_1, \ldots, i_6 \leq d}\left\{ \mathbb{E}\left[ \sum_{k=1}^6 \left\|B_{1, i_k}^{*(X)}\right\|_2^{2+\delta} ~\Big|~ X_1, \ldots, X_n\right]\right.                                                           \\
     & \phantom{\quad \leq 4^{2+\delta} 6^{2(1+\delta)} d^{-6}\sum_{1 \leq i_1, \ldots, i_6 \leq d}}\quad \left.\cdot\mathbb{E}\left[ \sum_{k=1}^6 \left\|B_{1, i_k}^{*(Y)}\right\|_2^{2+\delta} ~\Big|~ Y_1, \ldots, Y_n\right]\right\}                                       \\
     & \quad = 4^{2+\delta} 6^{2(1+\delta)} \sum_{k,l=1}^6 d^{-6} \sum_{1 \leq i_1, \ldots, i_6 \leq d}\left\{ \mathbb{E}\left[ \left\|B_{1, i_k}^{*(X)}\right\|_2^{2+\delta} ~\Big|~ X_1, \ldots, X_n\right] \right.                                                          \\
     & \phantom{\quad = 4^{2+\delta} 6^{2(1+\delta)} \sum_{k,l=1}^6 d^{-6} \sum_{1 \leq i_1, \ldots, i_6 \leq d}}\quad\left.\cdot\mathbb{E}\left[ \left\|B_{1, i_l}^{*(Y)}\right\|_2^{2+\delta} ~\Big|~ Y_1, \ldots, Y_n\right]\right\}                                        \\
     & \quad = 4^{2+\delta} 6^{2(1+\delta)} d^{-2} \sum_{1 \leq i_1, i_2 \leq d} \mathbb{E}\left[ \left\|B_{1, i_1}^{*(X)}\right\|_2^{2+\delta} ~\Big|~ X_1, \ldots, X_n\right] \mathbb{E}\left[ \left\|B_{1, i_2}^{*(Y)}\right\|_2^{2+\delta} ~\Big|~ Y_1, \ldots, Y_n\right] \\
     & \quad = 4^{2+\delta} 6^{2(1+\delta)} \left((Nd)^{-1} \sum_{k=1}^N \sum_{i=1}^d \left\|X_{(k-1)d+i}\right\|_2^{2+\delta}\right)\left((Nd)^{-1} \sum_{k=1}^N \sum_{i=1}^d \left\|Y_{(k-1)d+i}\right\|_2^{2+\delta}\right)                                                 \\
     & \quad = 4^{2+\delta} 6^{2(1+\delta)} \int \left\|x\right\|_2^{2+\delta} ~\mathrm{d}\mu_n(x) \int \left\|y\right\|_2^{2+\delta} ~\mathrm{d}\nu_n(y)                                                                                                                      \\
     & \quad \xrightarrow[n \to \infty]{a.s.} 4^{2+\delta} 6^{2(1+\delta)} \int \left\|x\right\|_2^{2+\delta} ~\mathrm{d}\mu(x) \int \left\|y\right\|_2^{2+\delta} ~\mathrm{d}\nu(y)
  \end{align*}
  by Birkhoff's Ergodic Theorem. The integrals in the last line are finite because $X_1$ and $Y_1$ have finite $(2+\delta)$-moments by assumption. From this it follows that we can almost surely bound Eq.\@ \eqref{eq:xir_bootstrap2} and therefore Eq.\@ \eqref{eq:xir_bootstrap} uniformly in $j_1, \ldots, j_6$ and for almost all $n$, and with a sufficiently large constant even for all $n$. Furthermore, because of the almost sure convergence, this uniform bound only depends on the $(2+\delta)$-moments of $\mu$ and $\nu$. Therefore, by Corollary \ref{cor:arconescor},
  $$
    \mathbb{E}\left[V_n^{*(r)}\left(h;\mu_n \otimes \nu_n\right)^2 ~\Big|~ Z_1, \ldots, Z_n\right] = \mathcal{O}\left(N^{-r}\right),
  $$
  where the constant involved only depends on $r$ and the $(2+\delta)$-moments of $\mu$ and $\nu$. Thus,
  \begin{align*}
    \mathbb{E}\left[V_n^{*(r)}\left(h;\mu_n \otimes \nu_n\right)^2 \right] & = \mathbb{E}\left[\mathbb{E}\left[V_n^{*(r)}\left(h;\mu_n \otimes \nu_n\right)^2 ~\Big|~ Z_1, \ldots, Z_n\right]\right] \\
                                                                           & = \mathcal{O}\left(N^{-r}\right)
  \end{align*}
  for all $2 \leq r \leq 6$, with the constant involved depending only on $r$ and the $(2+\delta)$-moments of $\mu$ and $\nu$. From here we can proceed as before, concluding the proof.
\end{proof}

The previous results have given us enough tools to derive a bound in terms of Wasserstein distances of the block distributions. We will need to show that these converge in some sense to $0$. This is achieved by the next lemma, which is simply an application of the results of Section \ref{sec:wasserstein} to the situation at hand.

\begin{lemma}
  \label{lem:opoly}
  Suppose that the conditions of Corollary \ref{cor:boundwasserstein} are fulfilled with the bound $M = M(d)$ growing at most exponentially in $d$. If $d = \mathcal{O}\left(\log(n)^\gamma\right)$ for some $0 < \gamma < 1/2$ and $d \to \infty$ as $n \to \infty$, and $Z_1$ has finite $q$-moments for some $q > p > 2$, then, for any $s > 0$, it holds that
  $$
    d_p^p\left(F_N^{(X)}, F^{(X)}\right) = o_\mathbb{P}\left(d^{-s}\right),
  $$
  and
  $$
    d_p^p\left(F_N^{(Y)}, F^{(Y)}\right) = o_\mathbb{P}\left(d^{-s}\right).
  $$
\end{lemma}
\begin{proof}
  We prove the first claim. The second claim can be proven in the same way, substituting $\ell_2$ for $\ell_1$.

  For any $\varepsilon > 0$, the Markov inequality and Corollary \ref{cor:boundwasserstein} give us
  \begin{align}
    \begin{split}
      \label{eq:gesamtwsk}
      \mathbb{P}&\left(d_p^p\left(F_N^{(X)}, F^{(X)}\right) \geq \varepsilon d^{-s}\right) \leq \frac{d^s}{\varepsilon}\mathbb{E}d_p^p\left(F_N^{(X)}, F^{(X)}\right) \\
      &\leq \frac{d^s}{\varepsilon}6^p c_0 2^{3(\ell_1 d)/2 - p} d^\frac{p}{2} N^{-\frac{p-2}{4(\ell_1 d)}} \left(\frac{1 + M^\frac{(\ell_1 d)/2 - p}{(\ell_1 d)}}{1 - 2^{p-(\ell_1 d)/2}} + \frac{1}{1 - 2^{-p}} + 4M^\frac{1}{(\ell_1 d)}\right) \\
      &\quad + \frac{d^s}{\varepsilon}6^p 2c' \cdot (\ell_1 d)^{1+q/2} N^{(p-2)(p-q)/(2(\ell_1 d)^2)}
    \end{split}
  \end{align}

  We first consider the term
  \begin{equation}
    \label{eq:term1}
    \frac{d^s}{\varepsilon}6^p c_0 2^{3(\ell_1 d)/2 - p} d^\frac{p}{2} N^{-\frac{p-2}{4(\ell_1 d)}} \left(\frac{1 + M^\frac{(\ell_1 d)/2 - p}{(\ell_1 d)}}{1 - 2^{p-(\ell_1 d)/2}} + \frac{1}{1 - 2^{-p}} + 4M^\frac{1}{(\ell_1 d)}\right).
  \end{equation}

  Because $M$ grows at most exponentially in $d$, we can bound $M^d$ by $2^{\rho d}$ for some $\rho > 0$. Without loss of generality, assume that $\rho > \ell_1$. Using this in Eq.\@ \eqref{eq:term1} results in a bound whose asymptotic behaviour is determined by
  \begin{equation}
    \label{eq:grenzeoben}
    N^{-\frac{p-2}{4(\ell_1 d)}} 2^\frac{(3\ell_1 + \rho) d}{2} d^s d^\frac{p}{2} \leq N^{-\frac{p-2}{4(\ell_1 d)}} 2^{3\rho d},
  \end{equation}
  where the inequality holds for sufficiently large $n$, because $d^s d^\frac{p}{2} = d^\frac{2s + p}{2}$ is polynomial in growth (and because $\rho > \ell_1$). Observe that, due do our choice of $d$,
  \begin{align*}
    N^{\frac{1}{d}} = \exp\left(\log(N)\right)^\frac{1}{d} = \exp\left(\frac{\log(N)}{d}\right) = \exp\left(\frac{\log(n) - \log(d)}{d}\right) \geq \exp\left(k_c \log(n)^{1-\gamma} - \frac{\log(d)}{d}\right),
  \end{align*}
  for some constant $k_c > 0$, and $2^d = \exp(\log(2)\cdot\log(n)^\gamma)$. Thus, the right-hand side of Eq.\@ \eqref{eq:grenzeoben} is bounded by
  \begin{equation}
    \label{eq:grenzeoben2}
    \exp\left(3\rho \log(2)\cdot\log(n)^\gamma - \frac{p-2}{4\ell_1}\left\{k_c \log(n)^{1-\gamma} - \frac{\log(d)}{\log(n)^\gamma}\right\}\right).
  \end{equation}
  Because $0 < \gamma < 1/2$, it holds that
  $$
    3\rho \log(2)\cdot\log(n)^\gamma - \frac{(p-2)k_c}{4\ell_1 }\log(n)^{1-\gamma} \xrightarrow[n \to \infty]{} -\infty,
  $$
  and
  $$
    \frac{\log(d)}{d} \xrightarrow[n \to \infty]{} 0,
  $$
  since $d \to \infty$ for $n \to \infty$. This implies that Eq.\@ \eqref{eq:grenzeoben2} converges to $0$ as $n \to \infty$, which in turn implies that the right-hand side of Eq.\@ \eqref{eq:grenzeoben} and thus Eq.\@ \eqref{eq:term1} converge to $0$.

  Let us now consider the term
  \begin{equation}
    \label{eq:grenzeunten}
    \frac{d^s}{\varepsilon}6^p 2c' \cdot (\ell_1 d)^{1+q/2} N^{(p-2)(p-q)/(2(\ell_1 d)^2)},
  \end{equation}
  which, writing $\alpha := (p-2)(p-q)/(2\ell_1^2)$, is bounded by
  \begin{align*}
    \frac{6^p2c'\ell_1^{1+q/2}}{\varepsilon}\exp\left(\left(1 + \frac{q}{2} + s\right) \log(d) + \alpha k_c \log(n)^{1 - 2\gamma} - \alpha\frac{\log(d)}{d^2}\right).
  \end{align*}
  The asymptotic behaviour of this expression is determined by
  \begin{align*}
     & \exp\left(\left(1 + \frac{q}{2} + s\right)\log(d) + \alpha k_c \log(n)^{1 - 2\gamma}\right)                                                                                  \\
     & \quad \leq \exp\left(\left(1 + \frac{q}{2} + s\right)\left(\gamma \log(\log(n)) - \log(k_c)\right) + \alpha k_c \log(n)^{1 - 2\gamma}\right) \xrightarrow[n \to \infty]{} 0,
  \end{align*}
  where the inequality and the convergence holds because $0 < \gamma < 1/2$ and $\alpha < 0$. Thus, Eq.\@ \eqref{eq:grenzeunten} converges to $0$, which together with the convergence of Eq.\@ \eqref{eq:term1} implies that the entire bound Eq.\@ \eqref{eq:gesamtwsk} converges to $0$, proving the lemma.
\end{proof}

The next lemma is a slight adaptation of Lemma \ref{lem:opoly} to the case where the blocks constituting the empirical measures $F_N^{(X)}$ and $F_N^{(Y)}$ overlap by some fixed amount. We will need this version of the lemma to prove consistency of the $L$-lag-test.

\begin{lemma}
  \label{lem:opoly_overlapping}
  Let $L \in \mathbb{N}$ be fixed and define
  $$
    \hat{X}_{k,d} := (X_{(k-1)d+1}, \ldots, X_{kd+L}),
  $$
  for any $d = d(n)$ and $1 \leq k \leq N$, and $\hat{Y}_{k,d}$ analogously. Let $\hat{F}_N^{(X)}$ and $\hat{F}_N^{(X)}$ denote the empirical measures of $\hat{X}_{1,d}, \ldots, \hat{X}_{N,d}$ and $\hat{Y}_{1,d}, \ldots, \hat{Y}_{N,d}$, respectively. Let $F_{d+L}^{(X)} := \mathcal{L}(X_1, \ldots, X_{d+L})$ and $F_{d+L}^{(Y)} := \mathcal{L}(Y_1, \ldots, Y_{d+L})$. Then, under the assumptions of Lemma \ref{lem:opoly}, it holds for any $s > 0$ that
  $$
    d_p^p\left(\hat{F}_N^{(X)}, F_{d+L}^{(X)}\right) = o_\mathbb{P}\left(d^{-s}\right),
  $$
  and
  $$
    d_p^p\left(\hat{F}_N^{(Y)}, F_{d+L}^{(Y)}\right) = o_\mathbb{P}\left(d^{-s}\right).
  $$
\end{lemma}
\begin{proof}
  The proof is analogous to that of Lemma \ref{lem:opoly}. We only need to check that, for any $d$, the transformed processes $\left(\hat{X}_{k,d}\right)_{k \in \mathbb{N}}$ and $\left(\hat{Y}_{k,d}\right)_{k \in \mathbb{N}}$ and their marginal distributions $F_{d+L}^{(x)}$ and $F_{d+L}^{(Y)}$ fulfil the assumptions of Corollary \ref{cor:boundwasserstein} with the bound $M = M(d)$ growing at most exponentially in $d$.

  It is obvious that the processes $\left(\hat{X}_{k,d}\right)_{k \in \mathbb{N}}$ and $\left(\hat{Y}_{k,d}\right)_{k \in \mathbb{N}}$ inherit the stationarity and mixing properties from the original processes $(X_k)_{k \in \mathbb{N}}$ and $(Y_k)_{k \in \mathbb{N}}$. Furthermore, the mixing rates can be bounded by some multiple of the original mixing rates (i.e.\@ those of the original processes $(X_k)_{k \in \mathbb{N}}$ and $(Y_k)_{k \in \mathbb{N}}$). Likewise, the moment conditions are satisfied by the newly constructed processes, and by assumption, the distributions $\mathcal{L}\left(\hat{X}_{1,d}\right) = \mathcal{L}\left(\hat{X}_{1,d}\right)$ have Lebesgue densities whose essential suprema are of order $\mathcal{O}\left(M^{d+L}\right) = \mathcal{O}\left(M^d\right)$. The same holds for $\mathcal{L}\left(\hat{Y}_{1,d}\right)$. We can thus proceed as in the proof of Lemma \ref{lem:opoly}.
\end{proof}

We are now in a position to prove the validity of our bootstrap method.

\begin{proof}[Proof of Theorem \ref{thm:hyp_bootstrap}]
  The triangle inequality gives us
  $$
    d_1(\zeta, nV^*) \leq d_1\left(\zeta, n\tilde{V}\right) + d_1\left(n\tilde{V}, nV^*\right),
  $$
  By Theorem \ref{thm:asymptotik}, the first summand converges to $0$ as $n \to \infty$. It remains to examine the second summand.

  By Lemma \ref{lem:hoeffdingdominiert}, it is sufficient to consider the $V$-statistics with kernel functions $h_2(z,z' ; \theta)$ and $h_2\left(z,z';\mu_n \otimes \nu_n\right)$, which, due to Lemma \ref{lem:hoeffdinghH}, are equal to the $V$-statistics with kernel functions $H_2\left(B,B';F^{(X)} \otimes F^{(Y)}\right)$ and $H_2\left(B,B';F_N^{(X)} \otimes F_N^{(Y)}\right)$. Thus, we consider the distance
  \begin{align}
    \begin{split}
      \label{eq:u-abstand}
      &d_1\left(\frac{Nd}{N^2}\sum_{1 \leq i_1, i_2 \leq N} H_2\left(\tilde{B}_{i_1},\tilde{B}_{i_2} ; F^{(X)} \otimes F^{(Y)}\right),\frac{Nd}{N^2}\sum_{1 \leq i_1, i_2 \leq N} H_2\left(B_{i_1}^*, B_{i_2}^* ; F_N^{(X)} \otimes F_N^{(Y)}\right)\right) \\
      &\leq d \cdot d_2\left(\frac{1}{N}\sum_{1 \leq i_1, i_2 \leq N} H_2\left(\tilde{B}_{i_1},\tilde{B}_{i_2} ; F^{(X)} \otimes F^{(Y)}\right), \frac{1}{N}\sum_{1 \leq i_1, i_2 \leq N} H_2\left(B_{i_1}^*, B_{i_2}^* ; F_N^{(X)} \otimes F_N^{(Y)}\right)\right) \\
      &\leq C \cdot \left\{ \left\|\tilde{B}_1\right\|_{L_4} + \left\|B_1^*\right\|_{L_4}\right\} \cdot d \cdot d_4\left(F_N^{(X)} \otimes F_N^{(Y)}, F^{(X)} \otimes F^{(Y)}\right),
    \end{split}
  \end{align}
  for some uniform constant $C$, by Proposition \ref{prop:abstandvstatistik}. In the remainder of the proof, we will allow $C$ to change in value from line to line, but it will always remain uniform in $n$ and $d$.

  For any two vectors $u \in \mathbb{R}^{m_1}$ and $v \in \mathbb{R}^{m_2}$, it holds that $\|(u,v)\|_2 \leq \|u\|_2 + \|v\|_2$, and so
  \begin{align*}
    \|\tilde{B}_1\|_{L_4}^4 & = \|(Z_1, \ldots, Z_d)\|_{L_4}^4 = \mathbb{E}\left[\|(Z_1, \ldots, Z_d)\|_2^4\right]                                                              \\
                            & \leq \mathbb{E}\left[\left(\sum_{k=1}^d \|Z_k\|_2\right)^4\right] \leq \mathbb{E}\left[d^3 \sum_{k=1}^d \|Z_k\|_2^4\right] = d^4 \|Z_1\|_{L_4}^4,
  \end{align*}
  or, equivalently,
  \begin{equation}
    \label{eq:bschlangel4bound}
    \|\tilde{B}_1\|_{L_4} \leq d \cdot \|Z_1\|_{L_4}.
  \end{equation}

  To bound $\|B_1^*\|_{L_4}$ note that $\|(u_1, \ldots, u_m)\|_2 = \|(u_{\sigma(1)}, \ldots, u_{\sigma(m)})\|_2$ for any $u \in \mathbb{R}^m$ and any permutation $\sigma \in \mathcal{S}_m$ (the symmetric group of order $m$). Thus,
  \begin{align*}
    \|B_1^*\|_{L_4}^4 & = \mathbb{E}\left[\|B_1^*\|_2^4 ~\big|~ Z_1, \ldots, Z_n\right] = \mathbb{E}\left[\left\|\left(B_{X,1}^{*}, B_{Y,1}^{*}\right)\right\|_2^4 ~\big|~ Z_1, \ldots, Z_n\right] \\
                      & \leq 2^3 \left\{\mathbb{E}\left[\|B_{X,1}^*\|_2^4 ~\big|~ Z_1, \ldots, Z_n\right] + \mathbb{E}\left[\|B_{Y,1}^*\|_2^4 ~\big|~ Z_1, \ldots, Z_n\right]\right\}.
  \end{align*}
  Note that
  \begin{align*}
    \mathbb{E}\left[\|B_{X,1}^*\|_2^4 ~\big|~ Z_1, \ldots, Z_n\right] & = \mathbb{E}\left[\left(\sum_{k=1}^d (X_k^*)^2\right)^2 ~\Big|~ Z_1, \ldots, Z_n\right]      \\
                                                                      & \leq d^2 \cdot \mathbb{E}\left[d^{-1} \sum_{k=1}^d (X_k^*)^4 ~\Big|~ Z_1, \ldots, Z_n\right] \\
                                                                      & = d^2 \cdot (Nd)^{-1} \sum_{k=1}^{Nd} X_k^4 = d^2 \cdot \frac{1}{n} \sum_{k=1}^n X_k^4.
  \end{align*}
  Analogously, one can show that
  $$
    \mathbb{E}\left[\|B_{Y,1}^*\|_2^4 ~\big|~ Z_1, \ldots, Z_n\right] \leq d^2 \cdot \frac{1}{n} \sum_{k=1}^n Y_k^4,
  $$
  and so
  \begin{equation}
    \label{eq:bsternl4bound}
    \|B_1^*\|_{L_4} \leq 2^\frac{3}{4} d^\frac{1}{2} \left(\frac{1}{n} \sum_{k=1}^n X_k^4 + \frac{1}{n} \sum_{k=1}^n Y_k^4\right)^\frac{1}{4}.
  \end{equation}

  By Lemma \ref{lem:metrikprodsumme}, it holds that
  $$
    d_4\left(F_N^{(X)} \otimes F_N^{(Y)}, F^{(X)} \otimes F^{(Y)}\right) \leq d_4\left(F_N^{(X)}, F^{(X)}\right) + d_4\left(F_N^{(Y)}, F^{(Y)}\right),
  $$
  and so combining \eqref{eq:bschlangel4bound} and \eqref{eq:bsternl4bound} yields that \eqref{eq:u-abstand} is bounded by
  \begin{equation}
    \label{eq:u-abstand2}
    C \cdot \left\{\|Z_1\|_{L_4} +  \left(\frac{1}{n} \sum_{k=1}^n X_k^4 + \frac{1}{n} \sum_{k=1}^n Y_k^4\right)^\frac{1}{4}\right\} \cdot d^2 \cdot \left( d_4\left(F_N^{(X)}, F^{(X)}\right) + d_4\left(F_N^{(Y)}, F^{(Y)}\right)\right).
  \end{equation}

  The distances $d_4\left(F_N^{(X)}, F^{(X)}\right)$ and $d_4\left(F_N^{(Y)}, F^{(Y)}\right)$ are $o_\mathbb{P}\left(d^{-2}\right)$ due to Lemma \ref{lem:opoly}, and by Birkhoff's Ergodic Theorem it holds that
  $$
    \frac{1}{n} \sum_{k=1}^n X_k^4 + \frac{1}{n} \sum_{k=1}^n Y_k^4 \xrightarrow[n \to \infty]{a.s.} \|X_1\|_{L_4}^4 + \|Y_1\|_{L_4}^4 < \infty,
  $$
  so the entire bound \eqref{eq:u-abstand2} converges to $0$ in probability.
\end{proof}

\begin{lemma}
  \label{lem:schwache_konvergenz_quantile}
  Let $\xi$ and $\xi_n$, $n \in \mathbb{N}$, be finite measures on $(\mathbb{R}, \mathcal{B}(\mathbb{R}))$ that satisfy the following two assumptions:
  \begin{enumerate}
    \item $\xi(\{x\}) = 0$ for all $x \in \mathbb{R}$,
    \item $\xi_n \Rightarrow \xi$ as $n \to \infty$.
  \end{enumerate}
  For an arbitrary but fixed $\alpha \in (0,1)$, let $I(\alpha) := \{x \in \mathbb{R} ~|~ \xi([x, \infty)) = \alpha\}$ be the set of all upper $\alpha$-quantiles of $\xi$. Then the following holds: For any sequence $(c_{n,\alpha})_{n \in \mathbb{N}}$ in $\mathbb{R}$ with the property that for each $n \in \mathbb{N}$, $c_{n, \alpha}$ is an upper $\alpha$-quantile of $\xi_n$, it holds that
  $$
    \Delta_n := \inf_{x \in I(\alpha)} |x - c_{n,\alpha}| \xrightarrow[n \to \infty]{} 0.
  $$
\end{lemma}
\begin{proof}
  First, note that $I(\alpha)$ is an interval, for if $\xi([x_1, \infty)) = \xi([x_2, \infty)) = \alpha$ for two points $x_1 \leq x_2$, then it follows that $\xi([x_1, x_2)) = 0$, which in turn implies that $\xi([x, \infty)) = \xi([x_2, \infty)) = \alpha$ for all $x \in [x_1, x_2]$. Now define
  $$
    I_n := \left[\inf I(\alpha) +\frac{1}{n}, \sup I(\alpha) - \frac{1}{n}\right).
  $$
  Then, by our observations above, $\xi(I_n) = 0$ for all $n \in \mathbb{N}$. By the continuity from below, this implies
  $$
    \xi\left(I(\alpha)^\circ\right) = \xi\left(\bigcup_{n \in \mathbb{N}} I_n\right) = \lim_{n \to \infty} \xi(I_n) = 0,
  $$
  where $A^\circ$ denotes the interior of a set $A$. Because $\xi$ has no measure in any single point, this also implies that
  $$
    \xi\left(\overline{I(\alpha)}\right) = \xi\left(I(\alpha)^\circ\right)  + \xi\left(\partial I(\alpha)\right) = 0,
  $$
  where $\overline{A}$ denotes the closure of a set $A$ and $\partial A$ its boundary.

  By the same reasoning, if $x \in I(\alpha)$ is an upper $\alpha$-quantile of $\xi$ and $x_0$ is another point such that $\xi([x_0, x)) = 0$, then $x_0 \in I(\alpha)$ as well (and analogously for $[x, x_0)$ if the other point is larger than $x$). As a consequence of this, any interval $I'$ for which it holds that $I' \cap I(\alpha) \neq \emptyset$ and $\xi(I') = 0$ is a subset of $I(\alpha)$. Since $\xi$ has no mass in any single point by assumption, this fact is true regardless of whether $I'$ contains one, both or neither of its endpoints.

  From this it follows that $I(\alpha)$ is a closed interval, since
  $$
    \overline{I(\alpha)} \cap I(\alpha) \neq \emptyset
  $$
  and so
  $$
    \overline{I(\alpha)} \subseteq I(\alpha).
  $$

  Now, for any $n \in \mathbb{N}$ such that $\Delta_n > 0$, define the interval
  $$
    R_n(\alpha) := \begin{cases} [c_{n,\alpha}, \min\{I(\alpha)\}), \quad & \textrm{if } c_{n,\alpha} < \min\{I(\alpha)\}, \\
              [\max\{I(\alpha)\}, c_{n,\alpha}), \quad & \textrm{if } c_{n,\alpha} > \max\{I(\alpha)\}.\end{cases}
  $$
  The fact that $\Delta_n > 0$ ensures that this object is well-defined, since in that case, $c_{n,\alpha}$ cannot be an element of $I(\alpha)$ itself, and so it must either be smaller than the minimum or larger than the maximum of $I(\alpha)$.

  Suppose that there is a subsequence $(n_k)_{k \in \mathbb{N}}$ and some $\varepsilon > 0$ such that $\Delta_{n_k} > \varepsilon$ for all $k \in \mathbb{N}$. We may assume without loss of generality that all $c_{n_k,\alpha}$ lie on one side of the interval $I(\alpha)$, since if there are infinitely many points on both sides, we may choose a further subsequence to achieve this. For the remainder of the proof, we assume that $c_{n_k, \alpha} < \min\{I(\alpha)\}$ for all $k \in \mathbb{N}$, but the other case (i.e.\@ $c_{n_k, \alpha} > \max\{I(\alpha)\}$) is completely analogous.

  By assumption, the length of each $R_{n_k}(\alpha)$ is strictly larger than $\varepsilon$, and so, defining
  $$
    R^{(\varepsilon)}(\alpha) := [\min\{I(\alpha)\} - \varepsilon, \min\{I(\alpha)\}),
  $$
  we have that $R^{(\varepsilon)}(\alpha) \subseteq R_{n_k}(\alpha)$ for all $k \in \mathbb{N}$. Furthermore, there exists some $\beta > 0$ such that $\xi\left(R_{n_k}^{(\varepsilon)}(\alpha)\right) > \beta$ for all $k \in \mathbb{N}$. To see this, suppose that
  $$
    \xi\left(R^{(\varepsilon)}\right) = \xi\left([\min\left\{I(\alpha)\right\} - \varepsilon, \min\left\{I(\alpha)\right\})\right) = 0.
  $$
  Because $\xi$ has no mass in any single point, this implies that the closed interval
  $$
    \overline{R^{(\varepsilon)}(\alpha)} = [\min\{I(\alpha)\} - \varepsilon, \min\{I(\alpha)\}]
  $$
  has probability $0$ as well, but because this closed interval intersects with the interval $I(\alpha)$, this implies that it is a subset of $I(\alpha)$ by the observations above. This however is a contradiction to the assumption that $\varepsilon > 0$, since $\min\{I(\alpha)\} - \varepsilon$ cannot be an element of the interval $I(\alpha)$.

  Finally, it holds that
  \begin{align*}
    \xi_{n_k}\left(R_{n_k}(\alpha)\right) & = \xi_{n_k}\left([c_{n_k, \alpha}, \min\{I(\alpha)\})\right) = \xi_{n_k}\left([c_{n_k, \alpha}, \infty)\right) - \xi_{n_k}\left([\min\{I(\alpha)\}, \infty)\right) \\
                                          & = \alpha - \xi_{n_k}\left([\min\{I(\alpha)\}, \infty)\right)                                                                                                       \\
                                          & \xrightarrow[k \to \infty]{} \alpha - \xi\left([\min\{I(\alpha)\}, \infty)\right) = \alpha - \alpha = 0.
  \end{align*}
  The equality in the second line holds because $c_{n_k, \alpha}$ is an upper $\alpha$-quantile of $\xi_{n_k}$ by definition, and the convergence in the last line holds by the Portmanteau theorem. Because we have $R^{(\varepsilon)}(\alpha) \subseteq R_{n_k}(\alpha)$ for all $k \in \mathbb{N}$, this implies
  $$
    \xi_{n_k}\left(R^{(\varepsilon)}(\alpha)\right) \leq \xi_{n_k}\left(R_{n_k}(\alpha)\right) \xrightarrow[k \to \infty]{} 0,
  $$
  but this is a contradiction since we have shown that $\xi\left(R^{(\varepsilon)}\right) > \beta > 0$.
\end{proof}

\begin{lemma}
  \label{lem:schwache_konvergenz_exceeding_prob}
  Let $(U_n)_{n \in \mathbb{N}}$ and $\left(U_n^*\right)_{n \in \mathbb{N}}$ be two sequences of real random variables that both converge in distribution to a random variable $U$. Suppose that $\mathbb{P}(U = x) = 0$ for all $x \in \mathbb{R}$, and for $\alpha \in (0,1)$ let $\left(c_{n,\alpha}^*\right)_{n \in \mathbb{N}}$ denote a sequence of real numbers such that $c_{n, \alpha}^*$ is an upper $\alpha$-quantile of $\mathcal{L}\left(X_n^*\right)$ for all $n \in \mathbb{N}$. Then it holds that
  $$
    \lim_{n \to \infty} \mathbb{P}\left(U_n > c_{n, \alpha}^*\right) = \lim_{n \to \infty} \mathbb{P}\left(U_n \geq c_{n, \alpha}^*\right)  = \alpha.
  $$
\end{lemma}
\begin{proof}
  Let us begin by proving the claim for the strict inequality, i.e.\@
  \begin{equation}
    \label{eq:strikte_gleichung_alpha}
    \lim_{n \to \infty} \mathbb{P}\left(U_n > c_{n, \alpha}^*\right)  = \alpha.
  \end{equation}
  Let $I(\alpha)$ denote the set of all upper $\alpha$-quantiles of $\mathcal{L}(U)$. As in the proof of Lemma \ref{lem:schwache_konvergenz_quantile}, one can show that $I(\alpha)$ is a closed interval. Let $\Delta_n := \inf_{x \in I(\alpha)} \left|x - c_{n, \alpha}^*\right|$. For any $\delta > 0$, define the sets
  \begin{align*}
    A(\delta)   & := \left\{\Delta_n  < \delta\right\},                                      \\
    A_+(\delta) & := \left\{0 < \Delta_n = c_{n,\alpha}^* - \max I(\alpha) < \delta\right\}, \\
    A_-(\delta) & := \left\{0 < \Delta_n = \min I(\alpha) - \Delta_n < \delta\right\},       \\
    A_0(\delta) & := \left\{c_{n,\alpha}^* \in I(\alpha)\right\}.
  \end{align*}
  The disjoint sets $A_+(\delta)$, $A_-(\delta)$ and $A_0(\delta)$ form a partition of $A(\delta)$, as $c_{n, \alpha}^*$ can either lie to the right of, to the left of or within $I(\alpha)$.

  Now, for an arbitrary but fixed $\varepsilon > 0$, choose some $\delta > 0$ such that
  \begin{equation}
    \label{eq:quantil_epsilon_delta}
    \mathbb{P}\left(U > \min I(\alpha) - \delta\right) \leq \alpha + \varepsilon.
  \end{equation}
  This is possible because $\mathbb{P}(U = x) = 0$ for all $x \in \mathbb{R}$. Now,
  \begin{align}
    \begin{split}
      \label{eq:wsk_quantil_A_partition}
      \mathbb{P}\left(U_n > c_{n, \alpha}^*\right) &= \mathbb{P}\left(\left\{U_n > c_{n,\alpha}^*\right\} \cap A_+(\delta)\right) + \mathbb{P}\left(\left\{U_n > c_{n,\alpha}^*\right\} \cap A_-(\delta)\right) \\
      &\qquad + \mathbb{P}\left(\left\{U_n > c_{n,\alpha}^*\right\} \cap A_0(\delta)\right) + \mathbb{P}\left(\left\{U_n > c_{n,\alpha}^*\right\} \cap A(\delta)^C\right).
    \end{split}
  \end{align}
  Observe that
  \begin{align*}
    \mathbb{P}\left(\left\{U_n > c_{n,\alpha}^*\right\} \cap A_+(\delta)\right) & \leq \mathbb{P}\left(\left\{U_n > \max I(\alpha) - \left| c_{n,\alpha}^* - \max I(\alpha)\right|\right\} \cap A_+(\delta)\right) \\
                                                                                & \leq \mathbb{P}\left(\left\{U_n > \max I(\alpha) - \delta\right\} \cap A_+(\delta)\right)                                        \\
                                                                                & \leq \mathbb{P}\left(\left\{U_n > \min I(\alpha) - \delta\right\} \cap A_+(\delta)\right).
  \end{align*}
  In the second inequality, we have used the fact that for any three sets $A, B, C$, $A \cap B \subseteq C$ implies that $A \cap B = (A \cap B) \cap B \subseteq C \cap B$. Similarly, one shows that
  $$
    \mathbb{P}\left(\left\{U_n > c_{n,\alpha}^*\right\} \cap A_-(\delta)\right) \leq \mathbb{P}\left(\left\{U_n > \min I(\alpha) - \delta\right\} \cap A_-(\delta)\right),
  $$
  and
  $$
    \mathbb{P}\left(\left\{U_n > c_{n,\alpha}^*\right\} \cap A_0(\delta)\right)  \leq \mathbb{P}\left(\left\{U_n > \min I(\alpha) - \delta\right\} \cap A_0(\delta)\right).
  $$
  Trivially, we also have
  $$
    \mathbb{P}\left(\left\{U_n > c_{n,\alpha}^*\right\} \cap A(\delta)^C\right) \leq \mathbb{P}\left(A(\delta)^C\right).
  $$
  Using these inequalities in Eq.\@ \eqref{eq:wsk_quantil_A_partition}, we get
  \begin{align*}
    \mathbb{P}\left(U_n > c_{n, \alpha}^*\right) & \leq \mathbb{P}\left(\left\{U_n > \min I(\alpha) - \delta\right\} \cap A_+(\delta)\right) + \mathbb{P}\left(\left\{U_n > \min I(\alpha) - \delta\right\} \cap A_-(\delta)\right) \\
                                                 & \qquad + \mathbb{P}\left(\left\{U_n > \min I(\alpha) - \delta\right\} \cap A_0(\delta)\right) + \mathbb{P}\left(A(\delta)^C\right)                                               \\
                                                 & = \mathbb{P}\left(\left\{U_n > \min I(\alpha) - \delta\right\} \cap A(\delta)\right) + \mathbb{P}\left(A(\delta)^C\right)                                                        \\
                                                 & \leq  \mathbb{P}\left(U_n > \min I(\alpha) - \delta\right) + \mathbb{P}\left(A(\delta)^C\right)
  \end{align*}
  as the subsets $A_+(\delta), A_-(\delta), A_0(\delta)$ form a partition of $A(\delta)$. Therefore,
  \begin{align*}
    \lim_{n \to \infty} \mathbb{P}\left(U_n > c_{n, \alpha}^*\right) & \leq \lim_{n \to \infty} \mathbb{P}\left(\left\{U_n > \min I(\alpha) - \delta\right\} \right) + \lim_{n \to \infty} \mathbb{P}\left(A(\delta)^C\right) \\
                                                                     & = \mathbb{P}\left(\left\{U > \min I(\alpha) - \delta\right\} \right) +  \lim_{n \to \infty} \mathbb{P}\left(A(\delta)^C\right)                         \\
                                                                     & \leq \alpha + \varepsilon,
  \end{align*}
  where in the second line we have used the Portmanteau theorem, and the last inequality follows from Eq.\@ \eqref{eq:quantil_epsilon_delta} and Lemma \ref{lem:schwache_konvergenz_quantile}. $\varepsilon > 0$ was arbitrary, and thus
  \begin{equation}
    \label{eq:quantil_kleiner_gleich}
    \lim_{n \to \infty} \mathbb{P}\left(U_n > c_{n, \alpha}^*\right)  \leq \alpha.
  \end{equation}

  To show the reverse inequality, we again choose an arbitrary but fixed $\varepsilon > 0$ and choose some $\delta > 0$ such that
  \begin{equation}
    \label{eq:quantil_epsilon_delta_2}
    \mathbb{P}(U > \max I(\alpha) + \delta) \geq \alpha - \varepsilon.
  \end{equation}
  We will now employ our methods above to the event $\{U_n > c_{n,\alpha}^*\}^C = \{U_n \leq c_{n,\alpha}^*\}$. More precisely, it holds that
  \begin{align*}
    \mathbb{P}\left(\{U_n > c_{n,\alpha}^*\}^C \cap A_+(\delta)\right) & = \mathbb{P}\left(\{U_n \leq c_{n,\alpha}^*\} \cap A_+(\delta)\right)                                                   \\
                                                                       & \leq \mathbb{P}\left(\left\{U_n \leq \max I(\alpha) + |c_{n,\alpha}^* - \max I(\alpha)|\right\} \cap A_+(\delta)\right) \\
                                                                       & \leq \mathbb{P}\left(\left\{U_n \leq \max I(\alpha) + \delta\right\} \cap A_+(\delta)\right)                            \\
                                                                       & = \mathbb{P}\left(\left\{U_n > \max I(\alpha) + \delta\right\}^C \cap A_+(\delta)\right).
  \end{align*}
  Similarly, one shows that
  $$
    \mathbb{P}\left(\{U_n > c_{n,\alpha}^*\}^C \cap A_-(\delta)\right) \leq \mathbb{P}\left(\{U_n > \max I(\alpha) + \delta\}^C \cap A_-(\delta)\right)
  $$
  and
  $$
    \mathbb{P}\left(\{U_n > c_{n,\alpha}^*\}^C \cap A_0(\delta)\right) \leq \mathbb{P}\left(\{U_n > \max I(\alpha) + \delta\}^C \cap A_0(\delta)\right).
  $$
  Because $A_+(\delta)$, $A_-(\delta)$ and $A_0(\delta)$ form a partition of $A(\delta)$, these three inequalities again imply
  \begin{align*}
    \mathbb{P}\left(\{U_n > c_{n,\alpha}^*\}^C\right) & = \mathbb{P}\left(\{U_n > c_{n,\alpha}^*\}^C \cap A(\delta)\right) + \mathbb{P}\left(\{U_n > c_{n,\alpha}^*\}^C \cap A(\delta)^C\right) \\
                                                      & \leq \mathbb{P}\left(\{U_n > \max I(\alpha) + \delta\}^C \cap A(\delta)\right) + \mathbb{P}\left(A(\delta)^C\right)                     \\
                                                      & \leq \mathbb{P}\left(\{U_n > \max I(\alpha) + \delta\}^C\right) + \mathbb{P}\left(A(\delta)^C\right),
  \end{align*}
  and so
  $$
    \mathbb{P}\left(U_n > c_{n,\alpha}^*\right) \geq 1 - \mathbb{P}\left(\{U_n > \max I(\alpha) + \delta\}^C\right) - \mathbb{P}\left(A(\delta)^C\right).
  $$
  This implies
  \begin{align*}
    \lim_{n \to \infty} \mathbb{P}\left(U_n > c_{n,\alpha}^*\right) & \geq 1 - \lim_{n \to \infty} \mathbb{P}\left(\left\{U_n > \max I(\alpha) + \delta\right\}^C\right) - \lim_{n \to \infty} \mathbb{P}\left(A(\delta)^C\right) \\
                                                                    & \geq 1 - \mathbb{P}\left(\left\{U > \max I(\alpha) + \delta\right\}^C\right)                                                                                \\
                                                                    & = \mathbb{P}\left(U > \max I(\alpha) + \delta\right)                                                                                                        \\
                                                                    & \geq \alpha - \varepsilon,
  \end{align*}
  where we have again used the Portmanteau Theorem and Lemma \ref{lem:schwache_konvergenz_quantile} in the second line, and the last line follows from Eq.\@ \eqref{eq:quantil_epsilon_delta_2}. Since $\varepsilon > 0$ was again arbitrary, this implies
  $$
    \lim_{n \to \infty} \mathbb{P}\left(U_n > c_{n,\alpha}^*\right) \geq \alpha,
  $$
  which together with Eq.\@ \eqref{eq:quantil_kleiner_gleich} proves Eq.\@ \eqref{eq:strikte_gleichung_alpha}.

  To prove our claim for the non-strict inequality, we will show that
  $$
    \lim_{n \to \infty} \mathbb{P}\left(U_n = c_{n, \alpha}^*\right) = 0,
  $$
  which will imply the desired result because
  $$
    \mathbb{P}\left(U_n \geq c_{n, \alpha}^*\right) = \mathbb{P}\left(U_n > c_{n, \alpha}^*\right) + \mathbb{P}\left(U_n = c_{n, \alpha}^*\right).
  $$
  For any $\delta > 0$ and any interval $I$, define $I_{\pm \delta} := [\min I - \delta, \max I + \delta]$. Then it holds that
  $$
    \left\{U_n = c_{n, \alpha}^*\right\} \cap A(\delta) \subseteq \{U_n \in I(\alpha)_{\pm \delta}\},
  $$
  and so, for any $\delta > 0$,
  \begin{align*}
    \mathbb{P}\left(U_n = c_{n, \alpha}^*\right) & =  \mathbb{P}\left(\left\{U_n = c_{n, \alpha}^*\right\} \cap A(\delta)\right) + \mathbb{P}\left(\left\{U_n = c_{n, \alpha}^*\right\} \cap A(\delta)^C\right) \\
                                                 & \leq \mathbb{P}\left(U_n \in I(\alpha)_{\pm \delta}\right) + \mathbb{P}\left(A(\delta)^C\right).
  \end{align*}
  By the Portmanteau theorem and Lemma \ref{lem:schwache_konvergenz_quantile}, it therefore follows that
  $$
    \lim_{n \to \infty} \mathbb{P}\left(U_n = c_{n, \alpha}^*\right) \leq \mathbb{P}\left(U \in I(\alpha)_{\pm \delta}\right)
  $$
  for any $\delta > 0$. Choose a sequence $(\delta_k)_{k \in \mathbb{N}}$ with $\delta_k \searrow 0$ for $k \to \infty$, and recall from the proof of Lemma \ref{lem:schwache_konvergenz_quantile} that $\mathbb{P}(U \in I(\alpha)) = 0$. It then follows that
  \begin{align*}
    \lim_{n \to \infty} \mathbb{P}\left(U_n = c_{n, \alpha}^*\right) & \leq \liminf_{k \to \infty} \mathbb{P}\left(U \in I(\alpha)_{\pm \delta_k}\right) = \mathbb{P}\left(U \in \bigcap_{k \in \mathbb{N}} I(\alpha)_{\pm \delta_k}\right) = \mathbb{P}(U \in I(\alpha)) = 0
  \end{align*}
  by the continuity from above of the pushforward measure $\mathbb{P}^U$.
\end{proof}

\begin{proof}[Proof of Corollary \ref{cor:asymptoticlevelalpha}]
  The fact that our test has asymptotic level $\alpha$ follows from Theorem \ref{thm:hyp_bootstrap} and Lemma \ref{lem:schwache_konvergenz_exceeding_prob}. Now, if $X$ and $Y$ are not independent, it follows that
  $$
    \mathrm{dcov}(\theta_n) \xrightarrow[n \to \infty]{a.s.} \mathrm{dcov}(\theta) > 0,
  $$
  by Theorem 1 in \cite{kroll}, and thus $n\,\mathrm{dcov}(\theta_n) > c_\alpha^*$ for large $n$ almost surely.
\end{proof}

\begin{proof}[Proof of Corollary \ref{cor:L-lag-test}]
  Let $\theta_{n-L}'$ denote the empirical measure of the vectorised sample $Z'_1, \ldots, Z'_{n-L}$.  We begin by showing that
  \begin{equation}
    \label{eq:thm1_analog_vektorisiert}
    d_1\left(\zeta', (n-L)V'^*\right) \xrightarrow[n \to \infty]{\mathbb{P}} 0,
  \end{equation}
  where $\zeta'$ is the weak limit of $(n-L) \, \mathrm{dcov}(\theta'_{n-L})$, and $V'^*$ is the bootstrapped version of $\mathrm{dcov}(\theta'_{n-L})$. Eq.\@ \eqref{eq:thm1_analog_vektorisiert} is the analogue of Theorem \ref{thm:hyp_bootstrap} for the vectorised sequence.

  Note that the assumptions of Theorem \ref{thm:hyp_bootstrap} carry over to the vectorised sequence $(Z_k')_{k \in  \mathbb{N}} = (X_k', Y_k')_{k \in  \mathbb{N}}$, with the exception of Assumption \ref{ass:folgecubeassumption}. The proof of Eq.\@ \eqref{eq:thm1_analog_vektorisiert} is therefore largely analogous to that of Theorem \ref{thm:hyp_bootstrap}. The bound in Eq.\@ \eqref{eq:u-abstand} is replaced by
  \begin{equation}
    \label{eq:u-abstand-vectorised}
    C \cdot \left\{ \left\|\tilde{B}'_1\right\|_{L_4} + \left\|B_1^{\prime *}\right\|_{L_4}\right\} \cdot d \cdot d_4\left(F_N^{\prime (X)} \otimes F_N^{\prime (Y)}, F^{\prime (X)} \otimes F^{\prime (Y)}\right),
  \end{equation}
  where the primed symbols denote the vectorised analogues of the objects in the proof of Theorem \ref{thm:hyp_bootstrap}; for instance, $B_1^{\prime *}$ is the first bootstrap block of the bootstrap procedure performed on the vectorised sample $(X'_1, Y'_1), \ldots, (X'_{n-L}, Y'_{n-L})$, and $F_N^{\prime (X)}$ is the empirical distribution of the observed blocks $B'_1, \ldots, B'_N$, again based on the vectorised sample.

  The fact that Assumption \ref{ass:folgecubeassumption} does not hold for the vectorised process $(Z'_k)_{k \in \mathbb{N}}$ means that we cannot bound the Wasserstein distance occuring in Eq.\@ \eqref{eq:u-abstand-vectorised} in the same manner as we did in the proof of Theorem \ref{thm:hyp_bootstrap}. We will now show that this Wasserstein distance converges to $0$ in probability using a different approach.

  Let us define the transformation
  \begin{align*}
    T_{d,L} : \mathbb{R}^{\ell_1 (d+L)} & \to \mathbb{R}^{(L+1) \ell_1 (d+L)},                                                \\
    (x_1, \ldots, x_{d+L})              & \mapsto (x_1, \ldots, x_{1+L}, x_2, \ldots, x_{2+L}, \ldots, x_d, \ldots, x_{d+L}).
  \end{align*}
  Note that each index $1 \leq k \leq d+L$ appears at most $L$ times in the transformed vector on the right-hand side, and so
  \begin{align*}
    \left\|T_{d,L}(x_1, \ldots, x_{d+L}) - T_{d,L}(x'_1, \ldots, x'_{d+L})\right\|_2^2 & \leq L \sum_{i=1}^{d+L} (x_i - x'_i)^2                             \\
                                                                                       & = L \, \| (x_1, \ldots, x_{d+L}) - (x'_1, \ldots, x'_{d+L})\|_2^2,
  \end{align*}
  for any two $(x_1, \ldots, x_{d+L}), (x'_1, \ldots, x'_{d+L}) \in \mathbb{R}^{\ell_1 (d+L)}$, i.e.\@ $T_{d,L}$ is Lipschitz-continuous with Lipschitz constant $\sqrt{L}$.

  We can describe the vectorised sequence using this mapping. More precisely, since we have by definition that $X'_k = (X_k, \ldots, X_{k+L})$ for all $k$, it holds that
  $$
    \left(X'_1, \ldots, X'_d\right) = T_{d,L}(X_1, \ldots, X_{d+L}).
  $$
  Thus, writing $F_{d+L}^{(X)} := \mathcal{L}(X_1, \ldots, X_{d+L})$, it holds that
  $$
    F^{\prime (X)} = \mathcal{L}\left(X'_1, \ldots, X'_d\right) = \mathcal{L}\left(T_{d,L}(X_1, \ldots, X_{d+L})\right) = \left(F_{d+L}^{(X)}\right)^{T_{d,L}},
  $$
  i.e.\@ $F^{\prime (X)}$ is the pushforward measure of $F^{(X)}_{d+L}$ under $T_{d,L}$.

  Now let $A$ be some measurable set. Then it holds that
  \begin{align*}
    F_N^{\prime (X)}(A) & = \frac{1}{N} \,\# \left\{1 \leq k \leq N ~|~ \left(X'_{(k-1)d + 1}, \ldots, X'_{kd}\right) \in A\right\}                 \\
                        & = \frac{1}{N} \,\# \left\{1 \leq k \leq N ~|~ T_{d,L}\left(X_{(k-1)d + 1}, \ldots, X'_{kd+L}\right) \in A\right\}         \\
                        & = \frac{1}{N} \,\# \left\{1 \leq k \leq N ~|~ \left(X_{(k-1)d + 1}, \ldots, X'_{kd+L}\right) \in T_{d,L}^{-1}(A)\right\}.
  \end{align*}
  Define $\hat{X}_{k,d} := \left(X_{(k-1)d + 1}, \ldots, X'_{kd+L}\right)$, then the equality above implies that
  $$
    F_N^{\prime (X)}(A) = \left(\hat{F}_N^{(X)}\right)^{T_{d,L}},
  $$
  where $\hat{F}_N^{(X)}$ denotes the empirical measure of $\hat{X}_{1,d}, \ldots, \hat{X}_{N,d}$.

  Finally, recall that if $g$ is any Lipschitz-continuous function with Lipschitz constant $c_g$, it holds that
  $$
    d_p\left(\eta^g, \xi^g\right) \leq c_g \, d_p(\eta, \xi)
  $$
  for any two measures $\eta$ and $\xi$. This follows directly from the definition of the Wasserstein distance.

  We therefore have
  $$
    d_4\left(F_N^{\prime (X)}, F^{\prime (X)}\right) = d_4\left(\left(\hat{F}_N^{(X)}\right)^{T_{d,L}}, \left(F_{d+L}^{X)}\right)^{T_{d,L}}\right) \leq \sqrt{L} \, d_4\left(\hat{F}_N^{(X)}, F_{d+L}^{X)}\right),
  $$
  and an analogous inequality holds for the Wasserstein distance between $F_N^{\prime (Y)}$ and $F^{\prime (Y)}$. Using Lemma \ref{lem:metrikprodsumme} yields
  $$
    d_4\left(F_N^{\prime (X)} \otimes F_N^{\prime (Y)}, F^{\prime (X)} \otimes F^{\prime (Y)}\right) \leq \sqrt{L} \, \left\{ d_4\left(\hat{F}_N^{(X)}, F_{d+L}^{X)}\right) + d_4\left(\hat{F}_N^{(X)}, F_{d+L}^{X)}\right)\right\},
  $$
  and the term on the right-hand side is $o_\mathbb{P}\left(d^{-s}\right)$ for any $s > 0$ by Lemma \ref{lem:opoly_overlapping}. It is important to point out that, in contrast to the vectorised processes $(X'_k)_{k \in \mathbb{N}}$ and $(Y'_k)_{k \in \mathbb{N}}$, the processes $\left(\hat{X}_{k,d}\right)_{k \in \mathbb{N}}$ and $\left(\hat{Y}_{k,d}\right)_{k \in \mathbb{N}}$ do satisfy all the assumptions of Lemma \ref{lem:opoly_overlapping}, including Assumption \ref{ass:folgecubeassumption}.

  We can now proceed as in the proof of Theorem \ref{thm:hyp_bootstrap} to obtain Eq.\@ \eqref{eq:thm1_analog_vektorisiert}. Some of the constants -- for instance those involved in moment bounds -- may now depend on $L$, but since this is a fixed parameter, this does not pose a problem.

  Now, by Eq.\@ \eqref{eq:thm1_analog_vektorisiert} and Lemma \ref{lem:schwache_konvergenz_exceeding_prob}, the $L$-lag test has asymptotic level $\alpha$. Furthermore, if $X_s$ and $Y_t$ are dependent for some $s,t \in \mathbb{N}$ with $|s-t| \leq L$, this implies that $X'_1$ and $Y'_1$ are dependent due to the stationarity of the sample generating processes. Therefore, again by Theorem 1 in \cite{kroll}, $\mathrm{dcov}\left(\theta'_{n-L}\right)$ converges almost surely to some strictly positive number, and so $(n - L) \, \mathrm{dcov}\left(\theta'_{n-L}\right)$ tends to infinity almost surely.
\end{proof}

\subsection[Bounding the Wasserstein Distance]{Bounding the Wasserstein Distance Between the Empirical Measure of a Strongly Mixing Process and its Marginal Distribution}
\label{sec:wasserstein}
In the proof of Theorem \ref{thm:hyp_bootstrap}, we have used the auxiliary results developed in Section \ref{sec:bootstrap_dcov} to bound the Wasserstein distance between the empirical distance covariance and the bootstrap statistic in terms of the Wasserstein distances between the block distributions and their empirical analogues. In this section, we will develop general results to bound the expected Wasserstein distance between some measure $\xi$ and its empirical analogue $\xi_n$, provided the underlying sample generating process is strictly stationary and $\alpha$-mixing. The central results of this section are Propositions \ref{prop:paperprop1} and \ref{prop:paperprop1phi}, Theorem \ref{thm:wasserstein_gesamt} and Corollary \ref{cor:boundwasserstein}.

We begin by deriving an equivalent characterisation of Assumption \ref{ass:folgecubeassumption}. For this, we need the following auxiliary result, which is a well-known alternative characterisation of sets with Lebesgue measure $0$.

\begin{lemma}
  \label{lem:lebesgue_null}
  Let $N \in \mathcal{B}\left(\mathbb{R}^d\right)$ be a Borel set, and denote by $\mathrm{vol}(A)$ the Lebesgue measure of a set $A$. Define the semiring
  $$
    \mathcal{I}_\mathbb{Q} := \left\{ \prod_{i=1}^d (a_i, b_i] ~\Big|~ \forall 1 \leq i \leq n : a_i, b_i \in \mathbb{Q}, ~a_i \leq b_i \right\}.
  $$
  Then the following two statements are equivalent:
  \begin{enumerate}
    \item $\mathrm{vol}(N) = 0,$
    \item for all $\varepsilon > 0$ there exists a covering $F_i \in \mathcal{I}_\mathbb{Q}$, $i \in \mathbb{N}$, of $N$ with total Lebesgue measure of less than $\varepsilon$, i.e.\@
          $$
            N \subseteq \bigcup_{i \in \mathbb{N}} F_i ~\land~ \sum_{i \in \mathbb{N}} \mathrm{vol}(F_i) < \varepsilon.
          $$
  \end{enumerate}
\end{lemma}
\begin{proof}
  Suppose that (ii) holds, then for any $\varepsilon > 0$, we can find a covering $F_i$, $i \in \mathbb{N}$, such that
  $$
    \mathrm{vol}(N) \leq \mathrm{vol}\left(\bigcup_{i \in \mathbb{N}} F_i\right) \leq \sum_{i \in \mathbb{N}} \mathrm{vol}(F_i) < \varepsilon,
  $$
  which implies (i) since $\varepsilon > 0$ is arbitrary.

  Now assume that (i) holds. Denote by $\mathrm{vol}^*$ the outer Lebesgue measure, i.e.\@
  $$
    \mathrm{vol}^*(A) := \inf \left\{ \sum_{i \in \mathbb{N}} \mathrm{vol}(F_i) ~\Big|~ A \subseteq \bigcup_{i \in \mathbb{N}} F_i, ~F_i \in \mathcal{I}_\mathbb{Q}\right\}.
  $$
  By Carathéodory's extension theorem (cf.\@ \cite{elstrodt:mass_und_integrationstheorie}, Theorems 4.5 and 5.6), we have that $\mathrm{vol}^*_{|\sigma(\mathcal{I}_\mathbb{Q})} = \mathrm{vol}_{|\sigma(\mathcal{I}_\mathbb{Q})}$. By Theorem 4.3 in \cite{elstrodt:mass_und_integrationstheorie}, $\mathcal{I}_\mathbb{Q}$ generates $\mathcal{B}\left(\mathbb{R}^d\right)$, and so we have $0 = \mathrm{vol}(N) = \mathrm{vol}^*(N)$. By definition of the outer Lebesgue measure $\mathrm{vol}^*$, this implies (ii).
\end{proof}

\begin{lemma}
  \label{lem:ass1_equivalent}
  Let $\xi$  be a Borel measure on some $\mathbb{R}^d$. Then the following two statements are equivalent:
  \begin{enumerate}
    \item There is some $M > 0$ such that $\xi(F) \leq M \cdot \mathrm{vol}(F)$ for all hypercubes $F \subseteq \mathbb{R}^d$,
    \item $\xi$ has an essentially bounded Lebesgue density $g$.
  \end{enumerate}
\end{lemma}
\begin{proof}
  If $\xi$ has a Lebesgue density $g$ which is essentially bounded by $M$, then
  $$
    \xi(F) = \int_F g(x) ~\mathrm{d}x \leq \int_F M ~\mathrm{d}x = M \cdot \mathrm{vol}(F)
  $$
  for any hypercube $F$.

  Let us now prove the reverse. Suppose that $\xi$ fulfils the first assumption. Let $\mathcal{I}_\mathbb{Q}$ be the semiring from Lemma \ref{lem:lebesgue_null}, and let $F \in \mathcal{I}_\mathbb{Q}$. Because the side lengths of $F$ are all rational, Theorem 5 in \cite{feshchenko_et_al:2010} implies that it can be partitioned into a finite number of hypercubes $F_1, \ldots, F_k$. The assumption (i) then implies that
  $$
    \xi(F) = \xi\left(\bigcup_{j=1}^k F_j\right) = \sum_{j=1}^k \xi(F_j) \leq M \sum_{j=1}^k \mathrm{vol}(F_j) = M\mathrm{vol}(F).
  $$
  Thus, the inequality from (i) holds for any $F \in \mathcal{I}_\mathbb{Q}$, not just hypercubes.

  Now, let $N \in \mathcal{B}\left(\mathbb{R}^d\right)$ such that $\mathrm{vol}(N) = 0$, and let $\varepsilon > 0$. Lemma \ref{lem:lebesgue_null} implies that we can find a countable covering $F_i \in \mathcal{I}_\mathbb{Q}$ of $N$ with total Lebesgue measure less than $\varepsilon$, and so
  $$
    \xi(N) \leq \xi\left(\bigcup_{i \in \mathbb{N}} F_i\right) \leq \sum_{i \in \mathbb{N}} \xi(F_i) \leq M \sum_{i \in \mathbb{N}} \mathrm{vol}(F_i) \leq M \varepsilon.
  $$
  Our choice of $\varepsilon > 0$ was arbitrary, and so $\xi(N) = 0$. Thus, by Radon-Nikodým's theorem (cf.\@ \cite{klenke:wahrscheinlichkeitstheorie}, Corollary 7.34), $\xi$ has an essentially bounded Lebesgue density.
\end{proof}
Lemma \ref{lem:ass1_equivalent} shows that the assumption of essentially bounded Lebesgue densities is not too strong, even though we only require the more technical characterisation concerning the hypercubes. Trivially, one can choose $M$ to be the essential supremum of the Lebesgue density $g$ of $\xi$, since
$$
  \xi(F) = \int_F g ~\mathrm{d}x \leq \esssup g \cdot \mathrm{vol}(F).
$$

\begin{lemma}
  \label{lem:varianzsumme}
  Let $(U_i)_{i \in \mathbb{N}}$ be a strictly stationary and $\alpha$-mixing sequence of random variables with values in some separable metric space $\mathcal{S}$ and marginal distribution $\xi$, and suppose that $\alpha(n) \leq f(n) = \mathcal{O}(n^{-r_0})$ for some $r_0 > 1$. Furthermore, let $F \subseteq \mathcal{S}$ be a subset with $\xi(F) \geq t_0$ for some $t_0 > 0$. Then there exists a constant $c_0 > 2 $, dependent on $r$ and $f$, but not on $n$, $F$ or $t_0$, such that
  $$
    \mathrm{Var}\left(\sum_{i=1}^n \textbf{1}_F(U_i)\right) \leq c_0 n t_0^{-1}\xi(F).
  $$
  If the sequence is $\phi$-mixing, we can drop the assumption $\xi(F) \geq t_0$. The assumptions concerning the growth rate of the $\alpha$-mixing coefficients then apply to the $\phi$-mixing coefficients, and we have the inequality
  $$
    \mathrm{Var}\left(\sum_{i=1}^n \textbf{1}_F(U_i)\right) \leq c_0 n \xi(F).
  $$
\end{lemma}
\begin{proof}
  We have
  \begin{align*}
    \mathrm{Var}\left(\sum_{i=1}^n \textbf{1}_F(U_i)\right) & = \left|\sum_{i=1}^n \mathrm{Var}(\textbf{1}_F(U_i)) + 2\sum_{d=1}^{n-1} (n-d) \mathrm{Cov}(\textbf{1}_F(U_1), \textbf{1}_F(U_{1+d}))\right| \\
                                                            & \leq \sum_{i=1}^n \mathrm{Var}(\textbf{1}_F(U_i)) + 4n\sum_{d=1}^{n-1} \left|\mathrm{Cov}(\textbf{1}_F(U_1), \textbf{1}_F(U_{1+d}))\right|,
  \end{align*}
  and $|\mathrm{Cov}(\textbf{1}_F(U_1), \textbf{1}_F(U_{1+d}))| \leq \alpha(d) \leq t_0^{-1}\xi(F)\alpha(d)$. By assumption there is some constant $c = c(f)> 0$ such that $\alpha(n) \leq f(n) \leq cn^{-r_0}$ for almost all $n$, and without loss of generality for all $n$. Therefore
  \begin{align*}
    \sum_{d=1}^{n-1}\left|\mathrm{Cov}(\textbf{1}_F(U_1), \textbf{1}_F(U_{1+d}))\right| & \leq c t_0^{-1}\xi(F)\sum_{d=1}^{n-1} d^{-r_0} \leq c\zeta(r_0)t_0^{-1}\xi(F),
  \end{align*}
  where $\zeta$ denotes the $\zeta$-function. Furthermore,
  $$
    \mathrm{Var}(\textbf{1}_F(U_1)) = \xi(F)\xi(F^C) \leq \xi(F) \leq t_0^{-1}\xi(F).
  $$
  Putting all of this together, we obtain
  $$
    \mathrm{Var}\left(\sum_{i=1}^n \textbf{1}_F(U_i)\right) \leq t_0^{-1}n\xi(F) + 4nc\zeta(r_0)t_0^{-1}\xi(F) = (1 + 4\zeta(r_0))\cdot nt_0^{-1}\xi(F).
  $$

  For $\phi$-mixing random variables, we can bound the covariance directly using Theorem 3.9 in \cite{bradley}, obtaining
  $$
    |\mathrm{Cov}(\textbf{1}_F(U_1), \textbf{1}_F(U_{1+d}))| \leq 2\phi(d) \|\textbf{1}_F(U_1)\|_{L_1} \|\textbf{1}_F(U_1)\|_{L_\infty} = 2\phi(d) \xi(F),
  $$
  and the rest of the proof proceeds analogously.
\end{proof}

\begin{remark}
  We assume that $\alpha$ is bounded by some function $f$ to exclude the possibility of any implicit influence the dimension might have on the mixing coefficients. Obviously, $f$ must not depend on the dimension.
\end{remark}

\begin{lemma}
  \label{lem:zetalemma}
  Let $(U_i)_{i \in \mathbb{N}}$ be a stochastic process that fulfils the assumptions of Lemma \ref{lem:varianzsumme}, $C$ a Borel set and $r > 1$. Then, setting $\zeta(t) := \sqrt{t} \land t$ for $t \geq 0$, the following inequality holds:
  $$
    \mathbb{E}|n\xi_n(C) - n \xi(C)| \leq \zeta_{n,r}(\xi(C)),
  $$
  where
  $$
    \zeta_{n,r}(\xi(C)) := \begin{cases}c_0\zeta\left(n\xi(C)\right) \quad                          & \textrm{for } \xi(C) \leq n^{-1},            \\
             c_0 n^{1/2 - 1/(2r)}\left(n \xi(C)\right)^\frac{1}{r} \quad & \textrm{for } n^{-1} < \xi(C) \leq n^{-1/2}, \\
             c_0 n^\frac{1}{4}\zeta(n \xi(C)) \quad                      & \textrm{for } \xi(C) > n^{-1/2}.\end{cases}
  $$
\end{lemma}
\begin{proof}
  We use the triangle inequality to obtain
  $$
    \mathbb{E}|n\xi_n(C) - n \xi(C)| \leq 2n\xi(C) \leq c_0 n \xi(C).
  $$
  If $\xi(C) \leq n^{-1}$, the right-hand side is equal to $c_0\zeta(n\xi(C))$.

  If $n^{-1} < \xi(C) \leq n^{-1/2}$, then for any $r > 1$,
  \begin{align*}
    \mathbb{E}|n\xi_n(C) - n\xi(C)| & \leq c_0 n \xi(C) = c_0 \left(n^r \xi(C)^r\right)^\frac{1}{r} \\
                                    & \leq c_0 (n^{r - (r-1)/2} \xi(C))^\frac{1}{r}                 \\
                                    & = c_0 (n \xi(C))^\frac{1}{r} n^{(r - (r-1)/2 - 1)/r}          \\
                                    & = c_0 (n\xi(C))^\frac{1}{r} n^{1/2 - 1/(2r)}.
  \end{align*}

  For $\xi(C) > n^{-1/2}$ we employ Lemma \ref{lem:varianzsumme} and the Cauchy-Schwarz inequality to obtain
  \begin{align*}
    \mathbb{E}|n\xi_n(C) - n \xi(C)| & \leq \sqrt{\mathrm{Var}(n\xi_n(C))} \leq \sqrt{c_0 n^{3/2} \xi(C)} = n^\frac{1}{4}\zeta(c_0 n \xi(C)),
  \end{align*}
  where the last equality holds because of $\xi(C) > n^{-1/2} \geq n^{-1}$.
\end{proof}

\begin{lemma}
  \label{lem:beschrphi}
  Suppose that $(U_i)_{i \in \mathbb{N}}$ is a stochastic process that fulfils the assumptions of Lemma \ref{lem:varianzsumme} in the $\phi$-mixing version, and let $C$ be any Borel set. Then, with the same function $\zeta$ as in Lemma \ref{lem:zetalemma}, the following inequality holds:
  $$
    \mathbb{E}|n\xi_n(C) - n\xi(C)| \leq \zeta(2c_0 n \xi(C)).
  $$
\end{lemma}
\begin{proof}
  If $\xi(C) \geq n^{-1}$, the Jensen inequality yields
  $$
    \mathbb{E}|n\xi_n(C) - n\xi(C)| \leq \sqrt{\mathrm{Var}(n\xi_n(C))} \leq \sqrt{c_0 n \xi(C)} = \zeta(c_0 n \xi(C)),
  $$
  where we have made use of Lemma \ref{lem:varianzsumme} in the second inequality.

  Suppose now that $\xi(C) < n^{-1}$. The function $\zeta$ is bijective on $\mathbb{R}_{\geq 0}$. We therefore have the existence of an inverse function $\zeta^{-1}$, given by $\zeta^{-1}(t) = t^2 \lor t$. Applying the Jensen inequality gives us
  \begin{align*}
    \zeta^{-1}(\mathbb{E}|n\xi_n(C) - n\xi(C)|) & \leq \mathbb{E}\left[\zeta^{-1}(|n\xi_n(C) - n\xi(C)|)\right]                              \\
                                                & \leq \mathbb{E}\left[\textbf{1}_{(0,1)}(|n\xi_n(C) - n\xi(C)|)|n\xi_n(C) - n\xi(C)|\right] \\
                                                & \quad + \mathrm{Var}(n\xi_n(C)),
  \end{align*}
  where in the last inequality we first partition the domain of integration into $(0,1)$ and $[1,\infty)$ and then bound the integral on $[1,\infty)$ by the integral on $\mathbb{R}_{\geq 0}$.

  We now want to examine the first expected value. For this we define the sets
  \begin{align*}
    M   & := \{|n\xi_n(C) - n\xi(C)| < 1\}, \\
    M_0 & := \{n\xi_n(C) = 0\},             \\
    M_1 & := \{n\xi_n(C) = 1\}.
  \end{align*}
  Due to $n\xi(C) < 1$, the sets $M_0$ and $M_1$ form a partition of $M$. Therefore,
  \begin{align*}
    \mathbb{E} & \left[\textbf{1}_{(0,1)}(|n\xi_n(C) - n\xi(C)|)|n\xi_n(C) - n\xi(C)|\right] \\
               & = n\xi(C)\mathbb{P}(M_0) + (1-n\xi(C))\mathbb{P}(M_1)                       \\
               & \leq n\xi(C) + \mathbb{P}(M_1).
  \end{align*}
  We can furthermore bound $\mathbb{P}(M_1)$ by
  $$
    \mathbb{P}(M_1) = \sum_{i=1}^n \mathbb{P}(U_i \in C \land U_j \notin C \quad \forall j \neq i) \leq \sum_{i=1}^n \mathbb{P}(U_i \in C) = n\xi(C),
  $$
  and so we get
  $$
    \mathbb{E}\left[\textbf{1}_{(0,1)}(|n\xi_n(C) - n\xi(C)|)|n\xi_n(C) - n\xi(C)|\right] \leq 2 n \xi(C) \leq c_0 n \xi(C).
  $$

  With the upper bounds obtained before, it follows that
  \begin{align*}
    \zeta^{-1}(\mathbb{E}|n\xi_n(C) - n\xi(C)|) & \leq c_0 n \xi(C) + \mathrm{Var}(n\xi_n(C)) \leq 2c_0 n \xi(C).
  \end{align*}
  Together with the identity
  $$
    \mathbb{E}|n\xi_n(C) - n\xi(C)| = \zeta\left(\zeta^{-1}(\mathbb{E}|n\xi_n(C) - n\xi(C)|)\right)
  $$
  this proves the lemma because $\zeta$ is an isotone function.
\end{proof}
Recall the construction of the sets $\mathcal{P}_l$ from Subsection \ref{subsec:proof_outline_wasserstein}.
\begin{lemma}
  \label{lem:entropie}
  Let $\xi$ be a probability measure on $[0,1)^d$, and $M > 0$ such that for all $F \in \mathcal{P}_l$ it holds that $\xi(F) \leq M 2^{-dl}$. This is the case if $\xi$ has a bounded density function with respect to the Lebesgue measure. Then, for any $r > 0$, the following inequality holds:
  \begin{align*}
     & L_r = L_r(n) := \max\left\{l \in \mathbb{N} ~|~ \exists F \in \mathcal{P}_l : \xi(F) > n^{-r}\right\} \leq \log_2\left(2M^\frac{1}{d} n^\frac{r}{d}\right)
  \end{align*}
\end{lemma}
\begin{proof}
  Let $F \in \mathcal{P}_l$, then $\xi(F) \leq M 2^{-dl}$. This gives us
  $$
    L_1\left(M^{-1} 2^{dl}\right) \leq l = \log_2\left(M^\frac{1}{d}(M^{-1}2^{dl})^\frac{1}{d}\right)
  $$
  for all $l \in \mathbb{N}$. Now fix some $n \in \mathbb{N}$. We set
  $$l^* := \min\{l \in \mathbb{N} ~|~ M^{-1}2^{d(l - 1)} \leq n \leq M^{-1} 2^{dl}\}.$$
  Obviously it holds that $2^d n \geq M^{-1}2^{dl^*}$. The mappings $t \mapsto \log_2\left((Mt)^\frac{1}{d}\right)$ and $L_1$ are isotone, and therefore
  $$
    L_1(n) \leq L_1\left(M^{-1}2^{dl^*}\right) \leq \log_2\left(M^\frac{1}{d} (M^{-1}2^{dl^*})^\frac{1}{d}\right) \leq \log_2\left(M^\frac{1}{d}(2^dn)^\frac{1}{d}\right).
  $$
  This proves the claim for $r=1$. For arbitrary $r >0$, notice that $L_r(n) = L_1(n^r)$.
\end{proof}
\begin{remark}
  \label{rem:entropieremark}
  Note that if we are considering a hypercube with side length $K$ instead of the unit cube, then we can apply Lemma \ref{lem:entropie} after $\log_2(K)$ steps, and therefore the resulting bound will be
  $$
    \log_2\left(2M^\frac{1}{d} n^\frac{r}{d}\right) + \log_2(K) = \log_2\left(2KM^\frac{1}{d} n^\frac{r}{d}\right).
  $$
\end{remark}

The following proposition is a generalisation of Proposition 1 in \cite{dereich}.

\begin{proposition}
  \label{prop:paperprop1}
  Let $d \in \mathbb{N}$ and $1 \leq p < d/2$ be fixed. Let $\xi$ be a probability measure on $[0,1)^d$ that fulfils the assumptions of Lemma \ref{lem:entropie}. Then, for any $n \in \mathbb{N}$, it holds that
  $$
    \left(\mathbb{E}d_p^{p}(\xi_n, \xi)\right)^\frac{1}{p} \leq n^{-\frac{p-2}{2pd}} \left\{c_0 2^{3d/2-p} d^\frac{p}{2}\left(\frac{1 + M^\frac{d/2 - p}{d}}{1-2^{p - d/2}} + \frac{1}{1 - 2^{-p}}  +  4M^\frac{1}{d}\right)\right\}^\frac{1}{p},
  $$
  where $\xi_n$ is the empirical measure of a strictly stationary and $\alpha$-mixing process $(U_i)_{i \in \mathbb{N}}$ with marginal distribution $\xi$ and $\alpha(n) \leq f(n) = \mathcal{O}(n^{-r_0})$ for some function $f$ and some constant $r_0 > 1$. The constant $c_0$ only depends on $f$ and $r_0$.
\end{proposition}
\begin{proof}
  Lemma \ref{lem:paperlem2} gives us
  \begin{equation}
    \label{eq:wassersteinschranke}
    d_p^{p}(\xi_n, \xi) \leq \frac{1}{2}d^\frac{p}{2}\sum_{l=0}^\infty 2^{-pl} \sum_{F \in \mathcal{P}_l} \sum_{C\leftarrow F} \left| \xi_n(C) - \xi_n(F)\frac{\xi(C)}{\xi(F)}\right|.
  \end{equation}

  We consider the object
  \begin{equation}
    \label{eq:erwartungohnebedingung}
    \mathbb{E}\left| \xi_n(C) - \xi_n(F)\frac{\xi(C)}{\xi(F)}\right|
  \end{equation}
  An application of the triangle inequality lets us bound this expectation by
  $$
    \mathbb{E}\left| \xi_n(C) - \xi(C) + \xi(C) - \xi_n(F)\frac{\xi(C)}{\xi(F)}\right| = \mathbb{E}|\xi_n(C) - \xi(C)| + \xi(C)\mathbb{E}\left|1-\frac{\xi_n(F)}{\xi(F)}\right|.
  $$
  Using Lemma \ref{lem:zetalemma}, this gives us
  \begin{align*}
    \mathbb{E}|\xi_n(C) - \xi(C)| & \leq n^{-1}\zeta_{n,r}\left(\xi(C)\right),
  \end{align*}
  for any $r > 1$. In a similar way, we obtain
  \begin{align*}
    \xi(C)\mathbb{E} & \left|1-\frac{\xi_n(F)}{\xi(F)}\right| = \frac{\xi(C)}{\xi(F)}\mathbb{E}|\xi(F) - \xi_n(F)| \\
                     & \leq n^{-1}\frac{\xi(C)}{\xi(F)}\zeta_{n,r}\left(\xi(F)\right).
  \end{align*}
  Thus, Eq.\@ \eqref{eq:erwartungohnebedingung} can be bounded by
  \begin{align}
    \begin{split}
      \label{eq:erw1}
      \mathbb{E}\left| \xi_n(C) - \xi_n(F)\frac{\xi(C)}{\xi(F)}\right| &\leq n^{-1}\left\{\zeta_{n,r}\left(\xi(C)\right) + \frac{\xi(C)}{\xi(F)}\zeta_{n,r}\left(\xi(F)\right)\right\} \\
      &=: n^{-1}(R_1(C) + R_2(C,F)).
    \end{split}
  \end{align}
  We now wish to further bound the sums occurring in Eq.\@ \eqref{eq:wassersteinschranke}. For this, we will make use of the fact that only finitely many sets in $\mathcal{P} = \bigcup_{l = 0}^\infty \mathcal{P}_l$ have measure greater than $n^{-1/2}$. Note that $\sum_{C \leftarrow F} R_1(C)$ is equal to
  \begin{align}
    \begin{split}
      \label{eq:erw2}
      &\sum_{C \leftarrow F: \xi(C) \leq n^{-1}} \zeta_{n,r}\left(\xi(C)\right) + \sum_{C \leftarrow F: n^{-1} < \xi(C) \leq n^{-1/2}} \zeta_{n,r}\left(\xi(C)\right) + \sum_{C \leftarrow F: \xi(C) > n^{-1/2}} \zeta_{n,r}\left(\xi(C)\right) \\
      &= c_0\left(\sum_{C \leftarrow F: \xi(C) \leq n^{-1}} \zeta\left(n \xi(C)\right)+ n^{1/2 - 1/(2r)}\sum_{C \leftarrow F: n^{-1} < \xi(C) \leq n^{-1/2}} \left(n \xi(C)\right)^\frac{1}{r} \right. \\
      &\qquad\qquad+ \left.n^\frac{1}{4}\sum_{C \leftarrow F: \xi(C) > n^{-1/2}} \zeta\left(n \xi(C)\right)\right) \\
      &\leq c_0\left(\sum_{C \leftarrow F} \zeta\left(n \xi(C)\right)+ n^{1/2 - 1/(2r)}\sum_{C \leftarrow F: \xi(C) > n^{-1}} \left(n \xi(C)\right)^\frac{1}{r} + n^\frac{1}{4}\sum_{C \leftarrow F: \xi(C) > n^{-1/2}} \zeta\left( n \xi(C)\right)\right).
    \end{split}
  \end{align}
  The Jensen inequality implies that
  \begin{align}
    \begin{split}
      \label{eq:erw3}
      \sum_{C\leftarrow F} \zeta(n \xi(C)) &= 2^d 2^{-d} \sum_{C \leftarrow F} \zeta\left( n \xi(C)\right) \\
      &\leq 2^d \zeta\left(2^{-d} \sum_{C \leftarrow F} n \xi(C)\right) \\
      &= 2^d \zeta\left(2^{-d} n \xi(F)\right).
    \end{split}
  \end{align}
  Recall that $\#\mathcal{P}_l = 2^{dl}$, and that every $\mathcal{P}_l$ forms a partition of $[0,1)^d$. This gives us
  \begin{align}
    \begin{split}
      \label{eq:erw4}
      \sum_{F \in \mathcal{P}_l} \zeta(2^{-d} n\xi(F)) \leq 2^{dl} \zeta\left(2^{-dl} \sum_{F \in \mathcal{P}_l} 2^{-d} n\xi(F)\right) \leq 2^{dl} \zeta\left(2^{-dl}  n\right).
    \end{split}
  \end{align}
  We now choose $l^* := \left\lfloor\log_2\left(n^\frac{1}{d}\right)\right\rfloor$, with which we have
  \begin{align}
    \begin{split}
      \label{eq:erw4strich}
      \sum_{l=0}^\infty& 2^{(d-p)l}\zeta\left(2^{-dl} n\right) = \sum_{l=0}^{l^*} 2^{(d/2 - p)l}\sqrt{n} + \sum_{l > l^*} 2^{-pl}n \\
      &\leq \sum_{k=0}^\infty 2^{(d/2 - p)(l^* - k)} \sqrt{n} + 2^{-p(l^* + 1)}\sum_{l=0}^\infty 2^{-pl} n \\
      &= n^\frac{d/2 - p}{d} \sqrt{n}\frac{1}{1 - 2^{p-d/2}} + n^{-\frac{p}{d}} n \frac{1}{1-2^{-p}} \\
      &=  n^{1-\frac{p}{d}}\left\{\frac{1}{1 - 2^{p-d/2}} + \frac{1}{1-2^{-p}}\right\}.
    \end{split}
  \end{align}

  Notice that, in Eq.\@ \eqref{eq:erw2}, the sum over all sets $C \leftarrow F$ with measure greater than $n^{-1/2}$ is empty if  $\xi(F) \leq n^{-1/2}$. Thus,
  \begin{align}
    \begin{split}
      \label{eq:erw5}
      \sum_{F \in \mathcal{P}_l} &n^\frac{1}{4}\sum_{C \leftarrow F: \xi(C) > n^{-1/2}} \zeta\left(n \xi(C)\right) \\
      &\leq n^\frac{1}{4}\sum_{F \in \mathcal{P}_l : \xi(F) > n^{-1/2}} \sum_{C \leftarrow F} \zeta\left( n \xi(C)\right) \\
      &\leq n^\frac{1}{4}\sum_{F \in \mathcal{P}_l : \xi(F) > n^{-1/2}} 2^d \zeta\left(2^{-d}  n \xi(F)\right),
    \end{split}
  \end{align}
  where we have made use of Eq.\@ \eqref{eq:erw3} in the last step. Lemma \ref{lem:entropie} therefore implies
  \begin{align}
    \begin{split}
      \label{eq:erw6}
      n^\frac{1}{4}&\sum_{l=0}^\infty 2^{-pl}\sum_{F \in \mathcal{P}_l : \xi(F) > n^{-1/2}} \zeta\left(2^{-d}  n \xi(F)\right) \\
      &\leq n^\frac{1}{4}\sum_{l=0}^{L_{1/2}(n)} 2^{-pl} \sum_{F \in \mathcal{P}_l} \zeta\left(2^{-d}  n \xi(F)\right) \\
      &\leq n^\frac{1}{4}\sum_{l=0}^{L_{1/2}(n)} 2^{(d-p)l} \zeta\left(2^{-dl}  n\right) \\
      &\leq n^\frac{1}{4}\sum_{l=0}^{L_{1/2}(n)} 2^{(d-p)l}\sqrt{2^{-dl}n} \\
      &= n^\frac{3}{4}\sum_{l=0}^{L_{1/2}(n)} 2^{(d/2 - p)l} \\
      &\leq n^\frac{3}{4} \sum_{k=0}^\infty 2^{(d/2-p)(L_{1/2}(n)-k)} \\
      &= n^\frac{3}{4}\left(2M^\frac{1}{d}n^\frac{1}{2d}\right)^{d/2 - p} \sum_{k=0}^\infty 2^{-k(d/2-p)} \\
      &= 2^{d/2 - p} M^\frac{d/2 - p}{d} n^{\frac{3}{4} + \frac{d}{4d} - \frac{p}{2d}} \frac{1}{1 - 2^{p-d/2}} \\
      &= 2^{d/2 - p} M^\frac{d/2 - p}{d} n^{1 - \frac{p}{2d}}\frac{1}{1 - 2^{p-d/2}}.
    \end{split}
  \end{align}

  We will now bound the sum over all sets $C \leftarrow F$ with measure greater than $n^{-1}$. Once again, it suffices to consider only those sets whose parents $F$ have measure greater than $n^{-1}$, and because the function $t \mapsto t^\frac{1}{r}$ is concave, we obtain
  \begin{align}
    \begin{split}
      \label{eq:erw7}
      \sum_{F \in \mathcal{P}_l} &n^{1/2 - 1/(2r)}\sum_{C \leftarrow F: \xi(C) > n^{-1}} \left(n \xi(C)\right)^\frac{1}{r} \\
      &\leq n^{1/2 - 1/(2r)}\sum_{F \in \mathcal{P}_l : \xi(F) > n^{-1}} \sum_{C \leftarrow F} (n\xi(C)^\frac{1}{r} \\
      &\leq n^{1/2 - 1/(2r)}\sum_{F \in \mathcal{P}_l : \xi(F) > n^{-1}} 2^d \left(2^{-d} n \xi(F)\right)^\frac{1}{r},
    \end{split}
  \end{align}
  and thus
  \begin{align}
    \begin{split}
      \label{eq:erw8}
      n^{1/2 - 1/(2r)}&\sum_{l=0}^\infty 2^{-pl}\sum_{F \in \mathcal{P}_l : \xi(F) > n^{-1}} \left(2^{-d}  n \xi(F)\right)^\frac{1}{r} \\
      &\leq n^{1/2 - 1/(2r)}\sum_{l=0}^{L_1(n)} 2^{-pl} \sum_{F \in \mathcal{P}_l} \left(2^{-d} n\xi(F)\right)^\frac{1}{r} \\
      &\leq n^{1/2 - 1/(2r)}\sum_{l=0}^{L_1(n)} 2^{(d-p)l} \left(2^{-dl} n\right)^\frac{1}{r} \\
      &= n^{1/2 + 1/(2r)}\sum_{l=0}^{L_1(n)} 2^{(d(1 - 1/r) - p)l}.
    \end{split}
  \end{align}
  With the special choice $r = d/(d - p)$ it holds that $1/2 + 1/(2r) \leq 1 - p/(2d)$ and $1-1/r = p/d$. This implies the bound
  \begin{equation}
    \label{eq:erw9}
    n^{1 - \frac{p}{2d}} (L_1(n) + 1) \leq n^{1 - \frac{p}{2d}} \log_2\left(4M^\frac{1}{d}n^\frac{1}{d}\right) \leq 4M^\frac{1}{d} n^{1 - \frac{p-2}{2d}}.
  \end{equation}
  The above Eqs.\@ \eqref{eq:erw2} through \eqref{eq:erw9} give us
  \begin{align}
    \begin{split}
      \label{eq:erw10}
      &\sum_{l=0}^\infty 2^{-pl}\sum_{F \in \mathcal{P}_l} \sum_{C \leftarrow F} n^{-1}R_1(C) \\
      &\leq 2^d c_0 \left\{ n^{-\frac{p}{d}}\left(\frac{1}{1-2^{p - d/2}} + \frac{1}{1 - 2^{-p}}\right) + n^{-\frac{p}{2d}}2^{d/2 - p}M^\frac{d/2 - p}{d}\frac{1}{1-2^{p-d/2}} + n^{-\frac{p-2}{2d}} 4M^\frac{1}{d}\right\} \\
      &\leq 2^{3d/2 - p} c_0 n^{-\frac{p-2}{2d}} \left\{\frac{1 + M^\frac{d/2 - p}{d}}{1-2^{p - d/2}} + \frac{1}{1 - 2^{-p}}  +  4M^\frac{1}{d}\right\}.
    \end{split}
  \end{align}

  It now remains to bound the sums with the summands $R_2(C,F)$. Obviously, it holds that $\sum_{C \leftarrow F} \xi(C)/\xi(F) = 1$, and thus
  \begin{align}
    \begin{split}
      \label{eq:erw11}
      \sum_{C \leftarrow F} R_2(C,F) = \sum_{C \leftarrow F} \frac{\xi(C)}{\xi(F)}\zeta_{n,r}\left(\xi(F)\right) = \zeta_{n,r}\left(\xi(F)\right),
    \end{split}
  \end{align}
  and
  \begin{align}
    \begin{split}
      \label{eq:erw12}
      &\sum_{F \in \mathcal{P}_l} \zeta_{n,r}\left(\xi(F)\right) \\
      &\leq c_0\left\{ \sum_{F \in \mathcal{P}_l} \zeta\left( n \xi(F)\right) + n^{1/2 - 1/(2r)}\sum_{F \in \mathcal{P}_l : \xi(F) > n^{-1}} (n\xi(F))^\frac{1}{r} + n^\frac{1}{4}\sum_{F \in \mathcal{P}_l : \xi(F) > n^{-1/2}} \zeta\left(n\xi(F)\right)\right\}.
    \end{split}
  \end{align}
  Notice that these sums are bounded by the last terms of Eqs.\@ \eqref{eq:erw3}, \eqref{eq:erw5} and \eqref{eq:erw7}. Thus we again obtain the upper bound given in Eq.\@ \eqref{eq:erw10} for the sum
  $$
    \sum_{l=0}^\infty \sum_{F \in \mathcal{P}_l} \sum_{C \leftarrow F} n^{-1} R_2(C,F).
  $$
  Therefore, using Eqs.\@ \eqref{eq:wassersteinschranke}, \eqref{eq:erw1} and \eqref{eq:erw10}, it follows that
  \begin{align}
    \begin{split}
      \label{eq:wassersteinschranke2}
      \mathbb{E}d_p^p(\xi_n, \xi) &\leq c_0 2^{3d/2 - p} d^\frac{p}{2} n^{-\frac{p-2}{2d}} \left\{\frac{1 + M^\frac{d/2 - p}{d}}{1-2^{p - d/2}} + \frac{1}{1 - 2^{-p}}  +  4M^\frac{1}{d}\right\}.
    \end{split}
  \end{align}
\end{proof}

\begin{corollary}
  \label{cor:kkugel}
  Let $d \in \mathbb{N}$ and $1 \leq p < d/2$ be fixed. Let $\xi$ be a probability measure on the open ball $U_K(0)$ for some $K > 0$ that has an essentially bounded Lebesgue density. Then, for any $n \in \mathbb{N}$ it holds that
  $$
    \left(\mathbb{E}d_p^p(\xi_n, \xi)\right)^\frac{1}{p} \leq K^{d/(2p)} \mathfrak{M},
  $$
  where $\xi_n$ is the empirical measure of a strictly stationary and $\alpha$-mixing process $(U_i)_{i \in \mathbb{N}}$ with marginal distribution $\xi$ and $\alpha(n) \leq f(n) = \mathcal{O}\left(n^{-r_0}\right)$ for some function $f$ and some constant $r_0 > 1$, and $\mathfrak{M}$ is the bound from Proposition \ref{prop:paperprop1}.
\end{corollary}
\begin{proof}
  We define a new stationary and $\alpha$-mixing process $(\tilde{U}_i)_{i \in \mathbb{N}}$ by
  $$
    \tilde{U}_i := \frac{U_i}{2K} + \left(\frac{1}{2}, \ldots, \frac{1}{2}\right).
  $$
  Then $\tilde{\xi} := \mathcal{L}(\tilde{U}_1)$ is a measure on $[0,1)^d$, and $\tilde{\xi}_n$, defined as the empirical measure of $\tilde{U}_1, \ldots, \tilde{U}_n$, fulfils the assumptions of Proposition \ref{prop:paperprop1}. Furthermore, for any $F \in \mathcal{P}_l$, it holds that
  \begin{align*}
    \tilde{\xi}(F) & = \mathbb{P}(\tilde{U}_1 \in F) = \mathbb{P}(U_1 \in 2K \cdot F + (K, \ldots, K)) \\
                   & \leq M \cdot \mathrm{vol}(2K\cdot F + (K, \ldots, K))                             \\
                   & = (2K)^dM \cdot \mathrm{vol}(F)                                                   \\
                   & =: \tilde{M} \cdot \mathrm{vol}(F),
  \end{align*}
  by Lemma \ref{lem:ass1_equivalent}. Therefore, by Proposition \ref{prop:paperprop1},
  \begin{align*}
    \mathbb{E}d_p^p(\tilde{\xi}_n, \tilde{\xi}) & \leq c_0 2^{3d/2 - p} d^\frac{p}{2} n^{-\frac{p-2}{2d}} \left\{\frac{1 + \tilde{M}^\frac{d/2 - p}{d}}{1-2^{p - d/2}} + \frac{1}{1 - 2^{-p}}  +  4\tilde{M}^\frac{1}{d}\right\} \\
                                                & \leq (2K)^{d/2 - p} \mathfrak{M}^p.
  \end{align*}
  Thus,
  $$
    \mathbb{E}d_p^p(\xi_n, \xi) = (2K)^p \mathbb{E}d_p^p(\tilde{\xi}_n, \tilde{\xi}) \leq (2K)^{d/2} \mathfrak{M}^p.
  $$
\end{proof}

\begin{proposition}[Version for $\phi$-Mixing]
  \label{prop:paperprop1phi}
  Let $d \in \mathbb{N}$ and $1 \leq p < d/2$ be fixed. Let $\xi$ be a probability measure on $[0,1)^d$. Then, for any $n \in \mathbb{N}$, it holds that
  $$
    \left(\mathbb{E}d_p^{p}(\xi_n, \xi)\right)^\frac{1}{p} \leq n^{-\frac{p}{d}} \left\{c_0 2^{d+1} d^\frac{p}{2}\left(\frac{1}{1-2^{p - d/2}} + \frac{1}{1 - 2^{-p}}   \right)\right\}^\frac{1}{p},
  $$
  where $\xi_n$ is the empirical measure of a strictly stationary and $\phi$-mixing process $(U_i)_{i \in \mathbb{N}}$ with marginal distribution $\xi$ and $\phi(n) \leq f(n) = \mathcal{O}(n^{-r_0})$ for some function $f$ and some constant $r_0 > 1$. The constant $c_0$ only depends on $f$ and $r_0$.
\end{proposition}
\begin{proof}
  The proof will follow that of Proposition \ref{prop:paperprop1}. The bound given in Eq.\@ \eqref{eq:wassersteinschranke} remains valid, and similarly to Eq.\@ \eqref{eq:erw1} we can use Lemma \ref{lem:beschrphi} to obtain the bound
  \begin{align}
    \begin{split}
      \label{eq:phierw1}
      &\mathbb{E}\left| \xi_n(C) - \xi_n(F)\frac{\xi(C)}{\xi(F)}\right| \\
      &\leq n^{-1}\left\{\zeta\left(2c_0 n\xi(C)\right) + \frac{\xi(C)}{\xi(F)}\zeta\left(2c_0 n\xi(F)\right)\right\} \\
      &=: n^{-1}(R_1(C) + R_2(C,F)).
    \end{split}
  \end{align}
  As in Eq.\@ \eqref{eq:erw3}, it follows that
  \begin{equation}
    \label{eq:phierw2}
    \sum_{C \leftarrow F} R_1(C) = \sum_{C \leftarrow F} \zeta(2c_0 n \xi(C)) \leq 2^{d} \zeta\left(2c_0 2^{-d} n \xi(F)\right),
  \end{equation}
  and as in Eq.\@ \eqref{eq:erw4} that
  \begin{equation}
    \label{eq:phierw3}
    \sum_{F \in \mathcal{P}_l} \zeta(2c_0 2^{-d} n \xi(F)) \leq 2^{dl} \zeta\left(2c_0 2^{-dl} n\right).
  \end{equation}
  We now choose $l^* := \left\lfloor \log_2\left((2c_0 n)^\frac{1}{d}\right)\right\rfloor$, so that, as in Eq.\@ \eqref{eq:erw4strich},
  \begin{equation}
    \label{eq:phierw4}
    \sum_{l=0}^\infty 2^{(d-p)l}\zeta\left(2c_0 2^{-dl} n\right) \leq 2c_0 n^{1 - \frac{p}{d}} \left\{\frac{1}{1 - 2^{p - d/2}} + \frac{1}{1 - 2^{-p}}\right\}.
  \end{equation}
  This implies
  \begin{equation}
    \label{eq:phierw5}
    \sum_{C \leftarrow F} R_1(C) \leq 2^{d+1}c_0 n^{1 - \frac{p}{d}} \left\{\frac{1}{1 - 2^{p - d/2}} + \frac{1}{1 - 2^{-p}}\right\},
  \end{equation}
  and because the sum
  $$
    \sum_{C \leftarrow F} \frac{\xi(C)}{\xi(F)} R_2(C,F) = \zeta(2c_0 n \xi(F))
  $$
  is bounded by the last term in Eq.\@ \eqref{eq:phierw2}, we get with Eqs.\@ \eqref{eq:wassersteinschranke} and \eqref{eq:phierw1} through \eqref{eq:phierw5} that
  $$
    \mathbb{E}d_p^{p}(\xi_n, \xi) \leq n^{-\frac{p}{d}} \left\{c_0 2^{d+1} d^\frac{p}{2}\left(\frac{1}{1-2^{p - d/2}} + \frac{1}{1 - 2^{-p}}   \right)\right\}.
  $$
\end{proof}
\begin{remark}
  In the $\phi$-mixing version, extending the support of $\xi$ to $U_K(0)$ only impacts the bound by a factor of $2K$ instead of $(2K)^{d/(2p)}$. The proof is analogous to that of Corollary \ref{cor:kkugel}.
\end{remark}

\begin{proof}[Proof of Theorem \ref{thm:wasserstein_gesamt}]
  Let $n$ and $K$ be arbitrary but fixed. Define the transformation
  $$
    \varphi_K := \textbf{1}_{U_K(0)}\cdot\mathrm{id} + \textbf{1}_{U_K(0)^C}\cdot z_K
  $$
  for some arbitrary but fixed $z_K \in \partial U_K(0)$. $\varphi_K$ maps all points outside of $U_K(0)$ to the single point $z_K$ and leaves the interior of $U_K(0)$ unchanged. Now define the probability measures $\xi^{(K)}$ and $\xi_n^{(K)}$ by
  \begin{align*}
    \xi^{(K)}(A)   & := \xi\left(A \cap U_K(0)\right) + \textbf{1}_A(z_K)\cdot\xi\left(U_K(0)^C\right),     \\
    \xi_n^{(K)}(A) & := \xi_n\left(A \cap U_K(0)\right) + \textbf{1}_A(z_K)\cdot\xi_n\left(U_K(0)^C\right).
  \end{align*}
  These measures are the pushforwards of $\xi$ and $\xi_n$, respectively, under $\varphi_K$. $\xi^{(K)}$ is the marginal distribution of the process $\left(\varphi_K(U_i)\right)_{i \in \mathbb{N}}$, and $\xi_n^{(K)}$ is the empirical measure of its first $n$ observations. Because $\varphi_K$ is measurable, $\left(\varphi_K(U_i)\right)_{i \in \mathbb{N}}$ inherits the $\alpha$-mixing property from $(X_i)_{i \in \mathbb{N}}$.

  Heuristically, in transforming $\xi$ to $\xi^{(K)}$, we have concentrated the probability of $U_K(0)^C$ to one single point, namely $z_K$. For large $K$, this probability of the complement will be small, so we expect $\xi^{(K)}$ to be close to $\xi$.

  Let us now quantify this notion of closeness. Let $\gamma$ be the pushforward of $\xi$ under $x \mapsto (x, \varphi_K(x))$, and $U := U_K(0)$. Observe that
  \begin{align}
    \begin{split}
      \label{eq:dreieck1}
      d_p^p\left(\xi, \xi^{(K)}\right) &\leq \int \|x-y\|^p ~\mathrm{d}\gamma(x,y) = \int \|x-\varphi_K(x)\|^p ~\mathrm{d}\xi(x)\\
      &= \int_U \|x-\varphi_K(x)\|^p ~\mathrm{d}\xi(x) + \int_{U^C} \|x-\varphi_K(x)\|^p ~\mathrm{d}\xi(x) \\
      &\leq 2^{p-1}\left(\int_{U^C} \|x\|^p ~\mathrm{d}\xi(x) + \int_{U^C} \|z_K\|^p ~\mathrm{d}\xi(x)\right) \\
      &\leq 2^{p-1}\left(\xi\left(U^C\right)^\frac{q-p}{q}m_q^{p/q} + \xi\left(U^C\right)K^p\right),
    \end{split}
  \end{align}
  where $m_q$ is the $q$-th moment of $\xi$. In the last line we have used the Hölder inequality and the fact that $z_K \in \partial U$.

  The same argument shows that
  $$
    d_p^p\left(\xi_n, \xi_n^{(K)}\right) \leq 2^{p-1}\left(\xi_n\left(U^C\right)^\frac{q-p}{q}m_{n,q}^{p/q} + \xi_n\left(U^C\right)K^p\right),
  $$
  where $m_{n,q}$ is the $q$-th moment of $\xi_n$. Observe that, by the Hölder inequality,
  \begin{align*}
    \mathbb{E}\left[\xi_n\left(U^C\right)^\frac{q-p}{q}m_{n,q}^{p/q}\right] & \leq \left\|\xi_n\left(U^C\right)^\frac{q-p}{q}\right\|_{L_{q/(q-p)}} \left\|m_{n,q}^{p/q}\right\|_{L_{q/p}} \\
                                                                            & = \xi\left(U^C\right)^\frac{q-p}{q} m_q^{p/q},
  \end{align*}
  and thus
  \begin{equation}
    \label{eq:dreieck2}
    \mathbb{E}d_p^p\left(\xi_n, \xi_n^{(K)}\right) \leq 2^{p-1}\left(\xi\left(U^C\right)^\frac{q-p}{q}m_q^{p/q} + \xi\left(U^C\right)K^p\right),
  \end{equation}
  since $\xi_n$ is the empirical measure of a stationary process with marginal distribution $\xi$.

  Furthermore, due to Corollary \ref{cor:kkugel},
  \begin{equation}
    \label{eq:dreieck3}
    \mathbb{E}d_p^p\left(\xi_n^{(K)}, \xi^{(K)}\right) \leq K^{d/2} \cdot \mathfrak{M}^p.
  \end{equation}

  We now use the triangle inequality and Jensen's inequality to obtain
  $$
    \mathbb{E}d_p^p(\xi, \xi_n) \leq 3^{p-1}\left\{d_p^p\left(\xi, \xi^{(K)}\right) + \mathbb{E}d_p^p\left(\xi_n^{(K)}, \xi^{(K)}\right) + \mathbb{E}d_p^p\left(\xi_n, \xi_n^{(K)}\right)\right\},
  $$
  and thus, by Eqs.\@ \eqref{eq:dreieck1} through \eqref{eq:dreieck3},
  $$
    \mathbb{E}d_p^p(\xi, \xi_n) \leq 3^{p-1}\left\{2^{p}\left(\xi\left(U^C\right)^\frac{q-p}{q}m_q^{p/q} + \xi\left(U^C\right)K^p\right) + K^{d/2}\cdot\mathfrak{M}^p\right\}.
  $$
\end{proof}

\begin{lemma}
  \label{lem:antiton_o_1}
  Let $g : \mathbb{R}_{> 0} \to \mathbb{R}_{\geq 0}$ be an antitone function such that
  $$
    \int_0^\infty g(x) ~\mathrm{d}x < \infty.
  $$
  Then $g = o\left(x^{-1}\right)$ for $x \to \infty$.
\end{lemma}
\begin{proof}
  Suppose that $g \neq o\left(x^{-1}\right)$. Then, by definition, there exists some constant $c > 0$ such that $\limsup_{x \to \infty} x g(x) > c$. We can therefore choose a sequence $(x_n)_{n \in \mathbb{N}}$ such that
  \begin{enumerate}
    \item $x_n \nearrow \infty$ for $n \to \infty$,
    \item $g(x_n)x_n > c$ for all $n \in \mathbb{N}$.
  \end{enumerate}
  Because of (i) we can assume without loss of generality that
  \begin{equation}
    \label{eq:xn_exponentiell}
    x_{n+1} > 2x_n
  \end{equation}
  for all $n \in \mathbb{N}$, otherwise we may consider a subsequence that has this property. Furthermore, because $g$ is antitone by assumption, (ii) implies that
  \begin{equation}
    \label{eq:g_antition_x_n}
    g(x) > \frac{c}{x_n}
  \end{equation}
  for all $x \leq x_n$ and $n \in \mathbb{N}$. Thus, on any interval $(x_n, x_{n+1}]$, it holds that $g(x) > c/x_{n+1}$, and after setting $x_0 := 0$, we obtain a partition $\{(x_n, x_{n+1}] ~|~ n \in \mathbb{N}_0\}$ of $\mathbb{R}_{> 0}$. Therefore,
  \begin{align*}
    \int_0^\infty g(x) ~\mathrm{d}x & = \sum_{n=0}^\infty \int_{x_n}^{x_{n+1}} g(x) ~\mathrm{d}x > \sum_{n=0}^\infty \int_{x_n}^{x_{n+1}} \frac{c}{x_{n+1}} ~\mathrm{d}x = c\sum_{n=0}^\infty \frac{x_{n+1} - x_n}{x_{n+1}} \\
                                    & = c\sum_{n=0}^\infty \left\{1 - \frac{x_n}{x_{n+1}}\right\} > c \sum_{n=0}^\infty \left\{1 - \frac{1}{2}\right\} = c \sum_{n=0}^\infty \frac{1}{2},
  \end{align*}
  which is a contradiction to the assumed integrability of $g$. In the above set of (in-)equalities, the first equality is valid due to the monotone convergence theorem and the two inequalities hold due to Eqs.\@ \eqref{eq:g_antition_x_n} and \eqref{eq:xn_exponentiell}, respectively.
\end{proof}

\begin{proof}[Proof of Corollary \ref{cor:boundwasserstein}]
  By Theorem \ref{thm:wasserstein_gesamt}, we get
  \begin{align}
    \begin{split}
      \label{eq:grenze0}
      \mathbb{E}d_p^p(\xi, \xi_n) &\leq 3^{p-1}\left\{2^{p}\left(\xi\left(U_K(0)^C\right)^\frac{q-p}{q}m_q^{p/q} + \xi\left(U_K(0)^C\right)K^p\right) + K^{d/2}\cdot\mathfrak{M}^p\right\} \\
      &\leq 6^p\left\{\xi\left(U_K(0)^C\right)^\frac{q-p}{q}m_q^{p/q} + \xi\left(U_K(0)^C\right)K^p + K^{d/2}\cdot\mathfrak{M}^p\right\}.
    \end{split}
  \end{align}
  Choose $K = n^{(p-2)/(2d^2)}$, then
  \begin{equation}
    \label{eq:grenze1}
    K^{d/2}\cdot\mathfrak{M}^p = c_0 2^{3d/2 - p} d^\frac{p}{2} n^{-\frac{p-2}{4d}} \left\{\frac{1 + M^\frac{d/2 - p}{d}}{1-2^{p - d/2}} + \frac{1}{1 - 2^{-p}}  +  4M^\frac{1}{d}\right\}.
  \end{equation}

  Writing $m_q'$ for the $q$-th moment of $U_1'$, we can bound $m_q$ by
  $$
    m_q = \mathbb{E}\left[\|U_1\|_2^q\right] = \mathbb{E}\left[\left(\sum_{i=1}^{d'} \|U_i'\|_2^2\right)^{q/2}\right] \leq (d')^{q/2} m_q',
  $$
  where we have used the Jensen inequality and the stationarity of the process $(U_i')_{i \in \mathbb{N}}$.

  Now observe that
  $$
    \infty > m'_q = \int_0^\infty \mathbb{P}\left(\|U_1'\|_2^q > K\right) ~\mathrm{d}K,
  $$
  which by Lemma \ref{lem:antiton_o_1} implies that $\mathbb{P}(\|U_1'\|_2^q > K) = o\left(K^{-1}\right)$, with the constant involved depending only on $\mathcal{L}(U_1')$. Furthermore,
  \begin{align*}
    \xi\left(U_K(0)^C\right) & = \mathbb{P}\left( \sum_{i=1}^{d'} \|U_i'\|_2^2 > K^2\right)                            \\
                             & \leq \mathbb{P}\left(\exists 1 \leq i \leq d' : \|U_i'\|_2 > \frac{K}{\sqrt{d'}}\right) \\
                             & \leq d' \cdot \mathbb{P}\left(\|U_1'\|_2 > \frac{K}{\sqrt{d'}}\right)                   \\
                             & = d' \cdot \mathbb{P}\left(\|U_1'\|_2^q > \frac{K^q}{(d')^{q/2}}\right),
  \end{align*}
  and thus
  $$
    \xi\left(U_K(0)^C\right) = d' \cdot o\left(\frac{K^{-q}}{(d')^{-q/2}}\right) = (d')^{1+q/2}\cdot o\left(K^{-q}\right),
  $$
  with the constant involved depending only on $\mathcal{L}(U_1')$. In other words, there is a constant $c'$ and some $n_0 \in \mathbb{N}$, both depending only on $\mathcal{L}(U_1')$, such that
  $$
    \xi\left(U_K(0)^C\right) \leq c' \cdot (d')^{1+q/2} K^{-q}
  $$
  for all $n \geq n_0$ (recall that our special choice of $K = K(n)$ depends on $n$). Thus, for $n \geq n_0$,
  \begin{align}
    \begin{split}
      \label{eq:grenze2}
      \xi\left(U_K(0)^C\right)^\frac{q-p}{q}m_p^{p/q} &\leq (c')^\frac{q-p}{q} \cdot (d')^\frac{(q-p)(2+q) + pq}{2q} K^{p-q} \\
      &= (c')^\frac{q-p}{q} \cdot (d')^{1 + q/2 - p/q} K^{p-q} \\
      &\leq c'\cdot (d')^{1+q/2} K^{p-q} \\
      &\leq c' \cdot d^{1+q/2} n^{(p-2)(p-q)/(2d^2)},
    \end{split}
  \end{align}
  because $d' \leq d$ (with equality if and only if each $U_k'$ is one-dimensional), and
  \begin{equation}
    \label{eq:grenze3}
    \xi\left(U_K(0)^C\right)K^p \leq c' \cdot (d')^{1+q/2} K ^{p-q} \leq c' \cdot d^{1+q/2} n^{(p-2)(p-q)/(2d^2)}.
  \end{equation}
  Thus, by Eqs.\@ \eqref{eq:grenze0} and \eqref{eq:grenze1} through \eqref{eq:grenze3}, we get
  \begin{align*}
    \mathbb{E}d_p^p(\xi, \xi_n) & \leq 6^p c_0 2^{3d/2 - p} d^\frac{p}{2} n^{-\frac{p-2}{4d}} \left(\frac{1 + M^\frac{d/2 - p}{d}}{1 - 2^{p-d/2}} + \frac{1}{1 - 2^{-p}} + 4M^\frac{1}{d}\right) \\
                                & \quad + 6^p 2c' \cdot d^{1+q/2} n^{(p-2)(p-q)/(2d^2)}.
  \end{align*}
\end{proof}

\subsection{Asymptotic Behaviour of $V$-Statistics of Independent Blocks Derived from a Stationary Process}
\label{sec:vblock}

In the proof of Theorem \ref{thm:hyp_bootstrap}, we make use of triangular arrays constructed in a certain manner: For a given strictly stationary sequence $(U_k)_{k \in \mathbb{N}}$, we define a triangular array $\tilde{U}$ by taking its $n$-th row as $N = N(n)$ iid copies of $(U_1, \ldots, U_d)$, $d = d(n)$ such that $n/(Nd) \to 1$ as $n \to \infty$. If we are then interested in the limiting distributions of $V$-statistics with sample data $U_1, \ldots, U_n$, compared to those with sample data $\tilde{U}_{1,n}, \ldots, \tilde{U}_{Nd, n}$, we can use the following theorem.

\begin{theorem}
  \label{thm:vschlange_d1}
  Suppose the conditions of Theorem \ref{thm:zschlange} are satisfied. Furthermore, assume $g$ to have finite $(3+\varepsilon)$-moments with respect to $\xi^2$ and $\alpha(n) = \mathcal{O}(n^{-r})$ with $r > 2 + 6\varepsilon^{-1}$. Then, with $\tilde{V} = V_g(\tilde{U}_{1,n}, \ldots, \tilde{U}_{Nd,n})$, if $d \to \infty$ for $n \to \infty$, it holds that
  $$
    d_1\left(n\tilde{V}, \sum_{k=1}^\infty \lambda_k \zeta_k^2\right) \xrightarrow[n \to \infty]{} 0,
  $$
  where $\lambda_k$ and $\zeta_k$ are the objects from Theorem \ref{thm:zschlange}.
\end{theorem}

As a starting point, we will generalise Theorem 2 in \cite{kroll} to triangular arrays. We will have to assume stronger moment conditions, but the proofs will be mostly similar. Note that Theorem 2 in \cite{kroll} is a weaker result than Theorem \ref{thm:zschlange} since it only gives us weak convergence instead of convergence in $d_1$. The only real difference from the original proof of Theorem 2 in \cite{kroll} lies in the CLT that we use.

\begin{theorem}
  \label{thm:vschlangeschwach}
  Under the conditions of Theorem \ref{thm:vschlange_d1}, it holds that
  $$
    n\tilde{V} \xrightarrow[n \to \infty]{\mathcal{D}} \zeta := \sum_{k=1}^\infty \lambda_k \zeta_k^2,
  $$
  where $\lambda_k$ and $\zeta_k$ are the objects from Theorem \ref{thm:zschlange} (i).
\end{theorem}
\begin{proof}
  As shown in the proof of Theorem 2 in \cite{kroll}, we have
  $$
    g(s, s') = \sum_{k=1}^\infty \lambda_k \varphi_k(s)\varphi_k(s')
  $$
  for $\xi^2$-almost all $s, s'$. The $\varphi_k$ are centred and form an orthonormal system in $L_2(\mathcal{U}, \xi)$. Let $c_1, \ldots, c_K$ be any selection of real constants. We define the following objects:
  \begin{align*}
    \vartheta_{t,n} & := \sum_{k=1}^K c_k \varphi_k\left(\tilde{U}_{t,n}\right),                   \\
    \zeta_{n,k}     & := \frac{1}{\sqrt{Nd}}\sum_{t=1}^{Nd} \varphi_k\left(\tilde{U}_{t,n}\right), \\
    S_{i,n}         & := \frac{1}{\sqrt{Nd}}\sum_{t=1}^{i} \vartheta_{t,n},                        \\
    V_{i,n}         & := \mathrm{Var}(S_{i,n}).
  \end{align*}
  Note that $Nd V^{(K)} = \sum_{k=1}^K \lambda_k \zeta_{n,k}^2$ and $\sum_{k=1}^K c_k \zeta_{n,k} = S_{Nd, n}$. Using the Cramér-Wold-device, we show that $(\zeta_{n,k})_{1 \leq k \leq K}$ converges in distribution to $(\zeta_1, \ldots, \zeta_K)$.

  First, because the dependence structure within each block is the same as in the original sequence $(U_k)_{k \in \mathbb{N}}$, and the blocks themselves are independent from each other, we have that
  \begin{align*}
    V_{i,n} & = \mathrm{Var}\left(\frac{1}{\sqrt{Nd}}\sum_{t=1}^d \vartheta_{t,n}\right) + \mathrm{Var}\left(\frac{1}{\sqrt{Nd}}\sum_{t=d+1}^{2d} \vartheta_{t,n}\right) + \ldots + \mathrm{Var}\left(\frac{1}{\sqrt{Nd}}\sum_{t=\lfloor i/d \rfloor d + 1}^i \vartheta_{t,n}\right) \\
            & = \left\lfloor \frac{i}{d} \right\rfloor \mathrm{Var}\left(\frac{1}{\sqrt{Nd}}\sum_{t=1}^d \vartheta_{t,n}\right) + \mathrm{Var}\left(\frac{1}{\sqrt{Nd}}\sum_{t=\lfloor i/d \rfloor d + 1}^i \vartheta_{t,n}\right)                                                   \\
            & = \left\lfloor \frac{i}{d}\right\rfloor \mathcal{O}\left(\frac{1}{N}\right) + \mathcal{O}\left(\frac{i \mod d}{Nd}\right).
  \end{align*}
  The rate of growth in the last line follows exactly as in the proof of Theorem 2 in \cite{kroll}, since, as mentioned before, the dependence structure within each block is identical to that of the original sequence $U$. Therefore, $V_{Nd, n}$ converges and
  $$
    V_{i,n} = \mathcal{O}(N) \Theta\left(\frac{1}{N}\right) + \mathcal{O}\left(\frac{1}{N}\right) = \mathcal{O}(1),
  $$
  where the constants involved are independent of $i$. From this it follows that
  \begin{equation}
    \label{eq:varianzlimsup}
    \limsup_{n \to \infty} \max_{1 \leq i \leq Nd} \frac{V_{i,n}}{V_{Nd,n}} < \infty.
  \end{equation}
  For $V_{Nd, n}$ in particular, we get
  \begin{equation}
    \label{eq:varianzgesamt}
    V_{Nd, n} = N \,\mathrm{Var}\left(\frac{1}{\sqrt{Nd}} \sum_{t=1}^d \vartheta_{t,n}\right) = \mathrm{Var}\left(\frac{1}{\sqrt{d}} \sum_{t=1}^d \vartheta_{t,n}\right) \xrightarrow[n \to \infty]{} \sigma^2.
  \end{equation}

  Define the functions
  $$
    Q_{t,n}(u) := \sup\left\{t ~\Big|~ \mathbb{P}\left((Nd)^{-1/2}\vartheta_{t,n} > t\right) > u\right\},
  $$
  and
  $$
    \alpha^{-1}(u) := \max\{k \in \mathbb{N} ~|~ \alpha(k) > u\}.
  $$
  Now, if for some $0 < u < 1$, $t \geq 0$ is chosen such that $\mathbb{P}\left(\left|(Nd)^{-1/2} \vartheta_{t,n}\right| > t\right) > u$, it follows by Markov's inequality that
  $$
    u < t^{-p} \left\|(Nd)^{-1/2} \vartheta_{t,n}\right\|_{L^p}^p
  $$
  for any $p \geq 1$, or, equivalently,
  $$
    t < \left\|(Nd)^{-1/2} \vartheta_{t,n}\right\|_{L^p} u^{-1/p},
  $$
  and so
  \begin{equation}
    \label{eq:Qtn_bound}
    Q_{t,n} \leq \left\|(Nd)^{-1/2}\vartheta_{t,n}\right\|_{L_p} u^{-1/p}
  \end{equation}
  for any $p \geq 1$. Furthermore, we have that $\alpha^{-1}(u) \leq 2c_0 u^{-1/r}$ for some constant $c_0 > 0$, since $\alpha(n) = \mathcal{O}(n^{-r})$. Lastly, note that $\|\vartheta_{t,n}\|_{L_p}$ is independent of $t$ and $n$ due to the stationarity of $(U_k)_{k \in \mathbb{N}}$. Therefore,
  $$
    \int_0^1 \alpha^{-2}\left(\frac{u}{2}\right) Q_{t,n}^3(u) ~\mathrm{d}u \leq \mathrm{const} \cdot (Nd)^{-\frac{3}{2}} \int_0^1  u^{-\frac{2}{r} - \frac{3}{p}} ~\mathrm{d}u = \mathcal{O}\left((Nd)^{-\frac{3}{2}}\right),
  $$
  where we have used the special choice $p = 3 + \varepsilon$. We get
  \begin{align*}
    V_{Nd, n}^{-\frac{3}{2}} & \sum_{t=1}^{Nd} \int_0^1 \alpha^{-1}\left(\frac{u}{2}\right) Q_{t,n}^2(u) \min\left\{\alpha^{-1}\left(\frac{u}{2}\right)Q_{t,n}(u), \sqrt{V_{Nd,n}}\right\} ~\mathrm{d}u \\
                             & \leq V_{Nd,n}^{-\frac{3}{2}} Nd \int_0^1 \alpha^{-2}\left(\frac{u}{2}\right) Q_{t,n}^3(u) ~\mathrm{d}u = \mathcal{O}\left(\frac{1}{\sqrt{Nd}}\right),
  \end{align*}
  and from this and Eq.\@ \eqref{eq:varianzlimsup}, Corollary 1 in \cite{rio_lindeberg} gives us
  $$
    S_{Nd, n} V_{Nd,n}^{-1} \xrightarrow[n \to \infty]{\mathcal{D}} \mathcal{N}(0,1).
  $$
  Note that the limit $\sigma^2$ in Eq.\@ \eqref{eq:varianzgesamt} is exactly the variance of $\sum_{k=1}^K c_k \zeta_k$, where $\zeta_k$ are centred Gaussian random variables with their covariance function given in Eq.\@ \eqref{eq:kovarianz}, and thus
  $$
    \sum_{k=1}^K c_k \zeta_{n,k} = S_{Nd, n} \xrightarrow[n \to \infty]{\mathcal{D}} \mathcal{N}(0, \sigma^2) = \mathcal{L}\left(\sum_{k=1}^K c_k \zeta_k\right).
  $$
  By the Cramér-Wold-device it follows that $(\zeta_{n,k})_{1\leq k \leq K} \xrightarrow[n \to \infty]{\mathcal{D}} (\zeta_k)_{1 \leq k \leq K}$. The continuous mapping theorem then gives us
  \begin{equation}
    \label{eq:billingsley1}
    Nd \, \tilde{V}^{(K)} = \sum_{k=1}^K \lambda_k \zeta_{n,k}^2 \xrightarrow[n\to \infty]{\mathcal{D}} \sum_{k=1}^K \lambda_k \zeta_k^2 =: \zeta^{(K)}.
  \end{equation}

  We now show that
  \begin{equation}
    \label{eq:billingsley3}
    \lim_{K \to \infty}\limsup_{n \to \infty} \mathbb{E}\left|Nd\,\tilde{V} - Nd\,\tilde{V}^{(K)}\right| = 0.
  \end{equation}
  This can be achieved in much the same way as in the original proof of Theorem 2 in \cite{kroll}. Let us introduce the Hilbert-space $H$ of all real-valued sequences $(a_k)_{k \in \mathbb{N}}$ such that $\sum_k \lambda_k a_k^2 < \infty$, with the inner product $\langle (a_k), (b_k)\rangle_H := \sum_k \lambda_k a_k b_k$. If we write $T_K(\tilde{U}_{t,n}) := (0^K, (\varphi_k(\tilde{U}_{t,n}))_{k > K})$, where $0^K$ is the $K$-dimensional zero vector, then
  \begin{align*}
    \mathbb{E}\left|Nd \,\tilde{V} - Nd\, \tilde{V}^{(K)}\right| = \frac{1}{Nd}\sum_{s,t=1}^{Nd} \mathrm{Cov}\left(T_K\left(\tilde{U}_{s,n}\right), T_K\left(\tilde{U}_{t,n}\right)\right).
  \end{align*}
  With the same argument regarding the block structure as before -- the dependence is identical within any given block, and the blocks are independent from each other --, we can instead examine the sum
  $$
    \frac{1}{Nd}\sum_{s,t=1}^{Nd} \mathrm{Cov}(T_K(U_s), T_K(U_t)).
  $$
  Now we can proceed exactly as in the proof of Theorem Theorem 2 in \cite{kroll} to show that
  $$
    \lim_{K\to \infty}\limsup_{Nd \to \infty} \frac{1}{Nd}\sum_{s,t=1}^{Nd} \mathrm{Cov}(T_K(U_s), T_K(U_t)) = 0.
  $$
  This implies Eq.\@ \eqref{eq:billingsley3}, because $Nd \xrightarrow[n \to \infty]{} \infty$.

  By Theorem 2 in \cite{dehling_durieu_volny:2009} , Eqs.\@ \eqref{eq:billingsley1} and \eqref{eq:billingsley3} imply
  $$
    Nd \,\tilde{V} \xrightarrow[n \to \infty]{\mathcal{D}} \zeta.
  $$
  Because $n/(Nd) \to 1$ as $n \to \infty$, it follows that
  $$
    n\tilde{V} = \frac{n}{Nd} Nd \, \tilde{V} \xrightarrow[n \to \infty]{\mathcal{D}} \zeta,
  $$
  by Slutsky's Theorem.
\end{proof}

Theorem 2 in \cite{kroll} together with Theorem \ref{thm:vschlangeschwach} give us that both $V$-statistics converge weakly to the same limiting distribution $\zeta$.

\begin{corollary}
  \label{cor:wasserstein}
  The convergences in Theorem 2 in \cite{kroll} and Theorem \ref{thm:vschlangeschwach} also hold in the Wasserstein metric $d_1$.
\end{corollary}
\begin{proof}
  By Theorem 6.9 in \cite{villani:optimaltransport}, convergence in $d_p$ is equivalent to weak convergence and uniform $p$-integrability. We therefore only need to show uniform integrability, for which we will borrow concepts from the proof of Theorem 2 in \cite{kroll}. We will show the uniform integrability by proving that the $(1+\alpha)$-moments of $nV$ are uniformly bounded for some $\alpha > 0$.

  With $H$ once again being the Hilbert space of all real-valued sequences $(a_k)_{k \in \mathbb{N}}$ such that $\sum_{k=1}^\infty \lambda_k a_k^2 < \infty$, where $\lambda_k$ are the eigenvalues from Theorem \ref{thm:vschlangeschwach}, equipped with the weighted inner product $\langle (a_k)_{k \in \mathbb{N}}, (b_k)_{k \in \mathbb{N}}\rangle_H := \sum_{k \in \mathbb{N}} \lambda_k a_k b_k$, we define the operator $T$ by $T(s) := (\varphi_k(s))_{k \in \mathbb{N}} \in H$. Then it holds that
  \begin{align*}
    \mathbb{E}\left[\|T(U_1)\|_H^{2+\varepsilon}\right] & = \mathbb{E}\left[\left(\sum_{k=1}^\infty \lambda_k \varphi_k(U_1)^2\right)^\frac{2+\varepsilon}{2}\right] = \mathbb{E}\left[\left(\sum_{k=1}^\infty \lambda_k \varphi_k(U_1)^2\right)^{1+\varepsilon/2}\right] \\
                                                        & = \mathbb{E}\left[g(U_1, U_1)^{1+ \varepsilon/2}\right] < \infty.
  \end{align*}
  Furthermore, it holds that
  \begin{align*}
    \mathbb{E}\left[\left\|\frac{1}{\sqrt{n}}\sum_{s=1}^n T(U_s)\right\|_H^2\right] & = \mathbb{E}\left[\left\langle\frac{1}{\sqrt{n}}\sum_{s=1}^n T(U_s), \frac{1}{\sqrt{n}}\sum_{t=1}^n T(U_t)\right\rangle_H\right] \\
                                                                                    & = \frac{1}{n} \sum_{s,t=1}^n \mathbb{E}\langle T(U_s), T(U_t)\rangle_H                                                           \\
                                                                                    & \leq \frac{15}{n} \sum_{s,t=1}^n \alpha(|s-t|)^{\varepsilon/(2+\varepsilon)} \|T(U_1)\|_{L_{2+\varepsilon}}^2
  \end{align*}
  by Lemma 2.2 in \cite{dehlingvector}. The partial sum of the mixing coefficient converges by assumption on the growth rate of the mixing coefficients, and so the entire bound converges to some finite constant which we will call $c$. Lemma 2.1 in \cite{dehlingvector} now gives us that there exists some $\alpha > 0$ such that
  $$
    \mathbb{E}\left[\left\|\sum_{s=1}^n T(U_s)\right\|_H^{2+2\alpha}\right] \leq c_0 n^{1+\alpha} \left(c + \mathbb{E}\left[g(U_1, U_1)^{1+ \varepsilon/2}\right]\right)
  $$
  for all $n \in \mathbb{N}$, where $c_0$ is some constant depending only on $\varepsilon$ and the mixing coefficients. Thus,
  \begin{align*}
    \mathbb{E}\left[|nV|^{1+\alpha}\right] & = \mathbb{E}\left[\left(\sum_{k=1}^\infty \lambda_k \left(\frac{1}{\sqrt{n}} \sum_{s=1}^n \varphi_k(U_s)\right)^2\right)^{1+\alpha}\right] = \mathbb{E}\left[\left\| \frac{1}{\sqrt{n}}\sum_{s=1}^n T(U_s) \right\|_H^{2+2\alpha}\right] \\
                                           & \leq  n^{-(1+\alpha)} \, c_0 n^{1+\alpha} \left(c + \mathbb{E}\left[g(U_1, U_1)^{1+ \varepsilon/2}\right]\right)                                                                                                                         \\
                                           & = c_0 \left(c + \mathbb{E}\left[g(U_1, U_1)^{1+ \varepsilon/2}\right]\right),
  \end{align*}
  and so the $(1+\alpha)$-moments of $nV$ are uniformly bounded, which implies uniform integrability. This in turn proves convergence in $d_1$ by Theorem 6.9 in \cite{villani:optimaltransport}. The generalisation of Theorem \ref{thm:zschlange} (ii) can be proven in the same way.
\end{proof}

Theorems \ref{thm:zschlange} and \ref{thm:vschlange_d1} now follow from Theorem 2 in \cite{kroll}, Theorem \ref{thm:vschlangeschwach} and Corollary \ref{cor:wasserstein}.

\begin{proof}[Proof of Theorem \ref{thm:asymptotik}]
  In the proof of Theorem 3 in \cite{kroll}, it is shown that $h_2(z,z';\theta)$ fulfils the assumptions of Theorem \ref{thm:zschlange}. Therefore,
  \begin{align*}
    d_1\left(nV_n^{(2)}(h;\theta), \zeta\right)         & \xrightarrow[n \to \infty]{} 0, \\
    d_1\left(n\tilde{V}_n^{(2)}(h;\theta), \zeta\right) & \xrightarrow[n \to \infty]{} 0.
  \end{align*}

  It remains to show that
  \begin{equation}
    \label{eq:untenhoeffding2}
    \mathbb{E}\left[\left(nV - 15nV_n^{(2)}(h;\theta)\right)^2\right] \xrightarrow[n \to \infty]{} 0,
  \end{equation}
  and
  \begin{equation}
    \label{eq:untenhoeffding2schlange}
    \mathbb{E}\left[\left(n\tilde{V} - 15n\tilde{V}_n^{(2)}(h;\theta)\right)^2\right] \xrightarrow[n \to \infty]{} 0.
  \end{equation}
  However, these two convergences follow by Lemma \ref{lem:hoeffdingdominiert}.
\end{proof}

\section{Additional Simulations}
\label{app:sims_hypothesentest}
In this Appendix, we give some more simulation results.

\begin{table}[h]
  \begin{tabular}{cccccccccccc}
          &     &     &     &     &     &     & $\rho$                         \\

    \cline{4-12}                                                               \\
    $\nu$ & $n$ & $d$ & 0.1 & 0.2 & 0.3 & 0.4 & 0.5    & 0.6 & 0.7 & 0.8 & 0.9 \\
    \cline{1-12}                                                               \\
    \\
    3     & 100 & 2   &
    0.107 &
    0.214 &
    0.406 &
    0.621 &
    0.805 &
    0.928 &
    0.978 &
    0.996 &
    1.000                                                                      \\
          &     & 4   &
    0.096 &
    0.199 &
    0.364 &
    0.588 &
    0.763 &
    0.914 &
    0.972 &
    0.996 &
    0.999                                                                      \\
          &     & 10  &
    0.090 &
    0.18  &
    0.364 &
    0.571 &
    0.768 &
    0.904 &
    0.974 &
    0.994 &
    0.999                                                                      \\
    \\
          & 300 & 2   &
    0.172 &
    0.444 &
    0.754 &
    0.946 &
    0.988 &
    0.998 &
    1.000 &
    1.000 &
    1.000                                                                      \\
          &     & 6   &
    0.138 &
    0.388 &
    0.723 &
    0.921 &
    0.985 &
    0.998 &
    0.999 &
    1.000 &
    1.000                                                                      \\
          &     & 17  &
    0.134 &
    0.378 &
    0.704 &
    0.923 &
    0.981 &
    0.996 &
    0.999 &
    1.000 &
    1.000                                                                      \\
    \\
    \\
    5     & 100 &
    2     &
    0.138 &
    0.272 &
    0.517 &
    0.741 &
    0.887 &
    0.976 &
    0.996 &
    1.000 &
    1.000                                                                      \\
          &     & 4   &
    0.113 &
    0.257 &
    0.477 &
    0.691 &
    0.878 &
    0.965 &
    0.994 &
    1.000 &
    1.000                                                                      \\
          &     & 10  &
    0.114 &
    0.229 &
    0.437 &
    0.665 &
    0.864 &
    0.964 &
    0.995 &
    1.000 &
    1.000                                                                      \\
    \\
          & 300 & 2   &
    0.223 &
    0.586 &
    0.883 &
    0.989 &
    0.999 &
    1.000 &
    1.000 &
    1.000 &
    1.000                                                                      \\
          &     & 6   &
    0.192 &
    0.530 &
    0.861 &
    0.983 &
    1.000 &
    1.000 &
    1.000 &
    1.000 &
    1.000                                                                      \\
          &     & 17  &
    0.177 &
    0.518 &
    0.845 &
    0.979 &
    0.999 &
    1.000 &
    1.000 &
    1.000 &
    1.000                                                                      \\
  \end{tabular}
  \caption[Rejection rates of the test based on Pearson's correlation for two dependent processes with $t_\nu$ marginals]{Rejection rates of the test based on Pearson's correlation for  $\mathrm{VAR}(1)$ processes with parameters $a=0.5$, $b=0$, and $r$, covariance matrix $\Gamma(0)$ as in Eq.\@ \eqref{eq:corr_matrix}, and length $n$. As block length we choose $d\in \{2, n^{\frac{1}{3}}, \sqrt{n}\}$. We consider a quantile transformation that yields a $t$-distribution with $\nu$ degrees of freedom. As block length we choose $d\in \{2, n^{\frac{1}{3}}, \sqrt{n}\}$.}
  \label{table:app_Pearson_t}
\end{table}

\begin{table}
  \begin{tabular}{cccccccccccc}
          &       &     &     &     &     &     & $\rho$                         \\
    \cline{4-12}                                                                 \\
    $\nu$ & $n$   & $d$ & 0.1 & 0.2 & 0.3 & 0.4 & 0.5    & 0.6 & 0.7 & 0.8 & 0.9 \\
    \cline{1-12}                                                                 \\
    \\
    3     & 100
          & 2
          & 0.112
          & 0.233
          & 0.439
          & 0.668
          & 0.861
          & 0.960
          & 0.996
          & 0.999
          & 1.000                                                                \\
          &       & 4
          & 0.091
          & 0.195
          & 0.395
          & 0.643
          & 0.830
          & 0.949
          & 0.992
          & 1.000
          & 1.000                                                                \\
          &       & 10
          & 0.081
          & 0.188
          & 0.381
          & 0.625
          & 0.820
          & 0.944
          & 0.992
          & 0.999
          & 1.000                                                                \\
    \\
          & 300   & 2   &
    0.208 &
    0.566 &
    0.874 &
    0.986 &
    0.999 &
    1.000 &
    1.000 &
    1.000 &
    1.000                                                                        \\
          &       & 6   &
    0.163 &
    0.507 &
    0.842 &
    0.980 &
    0.999 &
    1.000 &
    1.000 &
    1.000 &
    1.000                                                                        \\
          &       & 17  &
    0.150 &
    0.470 &
    0.814 &
    0.973 &
    0.999 &
    1.000 &
    1.000 &
    1.000 &
    1.000
    \\
    \\
    \\
    5     & 100   & 2   &
    0.128 &
    0.272 &
    0.491 &
    0.719 &
    0.881 &
    0.975 &
    0.996 &
    1.000 &
    1.000                                                                        \\
          &       & 4   &
    0.102 &
    0.239 &
    0.433 &
    0.683 &
    0.874 &
    0.967 &
    0.994 &
    0.999 &
    1.000                                                                        \\
          &       & 10  &
    0.090 &
    0.216 &
    0.423 &
    0.660 &
    0.851 &
    0.959 &
    0.994 &
    1.000 &
    1.000                                                                        \\
    \\
          & 300   & 2   &
    0.233 &
    0.607 &
    0.891 &
    0.989 &
    1.000 &
    1.000 &
    1.000 &
    1.000 &
    1.000                                                                        \\
          &       &
    6     &
    0.184 &
    0.538 &
    0.865 &
    0.985 &
    0.999 &
    1.000 &
    1.000 &
    1.000 &
    1.000                                                                        \\
          &       & 17
          & 0.164
          & 0.511
          & 0.846
          & 0.979
          & 1.000
          & 1.000
          & 1.000
          & 1.000
          & 1.000
  \end{tabular}
  \caption[Rejection rates of the test based on the distance covariance for two dependent processes with $t_\nu$ marginals]{Rejection rates of the test based on distance covariance for transformed $\mathrm{VAR}(1)$ processes with parameters $a=0.5$, $b=0$, and $r$, covariance matrix $\Gamma(0)$ as in Eq.\@ \eqref{eq:corr_matrix}, and length $n$. We consider a quantile transformation that yields a $t$-distribution with $\nu$ degrees of freedom. As block length we choose $d\in \{2, n^{\frac{1}{3}}, \sqrt{n}\}$.}
  \label{table:app_dist_corr_t}
\end{table}

\begin{table}
  \begin{tabular}{cccccccccccc}
          &       &     &     &     &     &     & $\rho$                         \\
    \cline{4-12}                                                                 \\
    $\nu$ & $n$   & $d$ & 0.1 & 0.2 & 0.3 & 0.4 & 0.5    & 0.6 & 0.7 & 0.8 & 0.9 \\
    \cline{1-12}                                                                 \\
    \\
    3     & 100   & 2   &
    0.055 &
    0.056 &
    0.057 &
    0.063 &
    0.069 &
    0.086 &
    0.092 &
    0.105 &
    0.092                                                                        \\
          &       & 4   &
    0.057 &
    0.057 &
    0.063 &
    0.062 &
    0.065 &
    0.066 &
    0.073 &
    0.072 &
    0.072                                                                        \\
          &       & 10  &
    0.055 &
    0.061 &
    0.054 &
    0.054 &
    0.062 &
    0.060 &
    0.069 &
    0.075 &
    0.070                                                                        \\
    \\
          & 300   & 2   &
    0.054 &
    0.055 &
    0.063 &
    0.068 &
    0.065 &
    0.078 &
    0.085 &
    0.083 &
    0.061                                                                        \\
          &       & 6
          & 0.050 &
    0.053 &
    0.057 &
    0.055 &
    0.061 &
    0.059 &
    0.059 &
    0.067 &
    0.050                                                                        \\
          &       & 17
          & 0.049 &
    0.045 &
    0.046 &
    0.056 &
    0.049 &
    0.055 &
    0.060 &
    0.053 &
    0.050
    \\
    \\
    \\
    5     & 100   &
    2     &
    0.057 &
    0.052 &
    0.062 &
    0.072 &
    0.090
          &
    0.099 &
    0.125 &
    0.157 &
    0.163                                                                        \\
          &       & 4   &
    0.049 &
    0.055 &
    0.065 &
    0.061 &
    0.066 &
    0.080 &
    0.091 &
    0.110 &
    0.116                                                                        \\
          &       & 10  &
    0.049 &
    0.054 &
    0.057 &
    0.060 &
    0.061 &
    0.069 &
    0.071 &
    0.089 &
    0.090                                                                        \\
    \\
          & 300   & 2   &
    0.052 &
    0.059 &
    0.063 &
    0.077 &
    0.077 &
    0.111 &
    0.138 &
    0.153 &
    0.122                                                                        \\
          &       & 6   &
    0.050 &
    0.055 &
    0.054 &
    0.060 &
    0.064 &
    0.067 &
    0.078 &
    0.079 &
    0.077                                                                        \\
          &       & 17  &
    0.053 &
    0.049 &
    0.048 &
    0.052 &
    0.058 &
    0.059 &
    0.054 &
    0.058 &
    0.059
  \end{tabular}
  \caption[Rejection rates of the test based on Pearson's correlation for two independent processes with $t_\nu$ marginals]{Rejection rates of the test based on Pearson's correlation for two independent AR(1) processes with parameter $\rho$  and length $n$. We consider a quantile transformation that yields a t-distribution with $\nu$ degrees of freedom. As block length we choose $d\in \{2, n^{\frac{1}{3}}, \sqrt{n}\}$.}
  \label{table:app_Pearson_ind_t}
\end{table}

\begin{table}
  \begin{tabular}{cccccccccccc}
          &     &     &       &     &     &     & $\rho$                         \\
    \cline{4-12}                                                                 \\
    $\nu$ & $n$ & $d$ & 0.1   & 0.2 & 0.3 & 0.4 & 0.5    & 0.6 & 0.7 & 0.8 & 0.9 \\
    \cline{1-12}                                                                 \\
    \\
    3     & 100 & 2   &
    0.043 &
    0.047 &
    0.048 &
    0.054 &
    0.065 &
    0.092 &
    0.111 &
    0.133 &
    0.125
    \\
          &     & 4   &
    0.050 &
    0.046 &
    0.049 &
    0.048 &
    0.050 &
    0.060 &
    0.078 &
    0.092 &
    0.085
    \\
          &     & 10  &
    0.051 &
    0.049 &
    0.046 &
    0.047 &
    0.049 &
    0.049 &
    0.063 &
    0.069 &
    0.075
    \\
    \\
          & 300 & 2   &
    0.045 &
    0.051 &
    0.058 &
    0.062 &
    0.082 &
    0.097 &
    0.128 &
    0.134 &
    0.092                                                                        \\
          &     & 6   &
    0.046 &
    0.050 &
    0.050 &
    0.049 &
    0.058 &
    0.060 &
    0.065 &
    0.068 &
    0.060                                                                        \\
          &     & 17  &
    0.043 &
    0.043 &
    0.044 &
    0.042 &
    0.047 &
    0.045 &
    0.048 &
    0.050 &
    0.056
    \\
    \\
    \\
    5     & 100 &
    2     &
    0.041 &
    0.055 &
    0.058 &
    0.065 &
    0.082 &
    0.109 &
    0.145 &
    0.235 &
    0.256
    \\
          &     & 4   &
    0.041 &
    0.054 &
    0.053 &
    0.055 &
    0.061 &
    0.075 &
    0.085 &
    0.125 &
    0.160
    \\
          &     & 10  &
    0.041 &
    0.052 &
    0.051 &
    0.051 &
    0.057 &
    0.064 &
    0.063 &
    0.078 &
    0.098
    \\
    \\
          & 300 & 2   &
    0.051 &
    0.051 &
    0.058 &
    0.075 &
    0.095 &
    0.122 &
    0.165 &
    0.258 &
    0.274
    \\
          &     & 6   & 0.054 &
    0.047 &
    0.054 &
    0.061 &
    0.066 &
    0.071 &
    0.073 &
    0.101 &
    0.120                                                                        \\
          &     & 17  &
    0.047 &
    0.043 &
    0.046 &
    0.051 &
    0.050 &
    0.054 &
    0.054 &
    0.061 &
    0.075
  \end{tabular}
  \caption[Rejection rates of the test based on the distance covariance for two independent processes with $t_\nu$ marginals]{Rejection rates of the test based on distance covariance for two independent AR(1) processes with parameter $\rho$ and length $n$. We consider a quantile transformation that yields a t-distribution with $\nu$ degrees of freedom. As block length we choose $d\in \{2, n^{\frac{1}{3}}, \sqrt{n}\}$.}
  \label{table:app_dist_corr_ind_t}
\end{table}

\begin{table}[h]
  \begin{tabular}{cccccccccccc}
             &     &     &       &     &     &     & $\rho$                         \\
    \cline{4-12}                                                                    \\
    $\alpha$ & $n$ & $d$ & 0.1   & 0.2 & 0.3 & 0.4 & 0.5    & 0.6 & 0.7 & 0.8 & 0.9 \\
    \cline{1-12}
    \\
    3        & 100 &
    2        &
    0.043    &
    0.065    &
    0.087    &
    0.111    &
    0.142    &
    0.196    &
    0.259    &
    0.328    &
    0.408                                                                           \\
             &     & 4   & 0.047 &
    0.062    &
    0.078    &
    0.109    &
    0.140    &
    0.196    &
    0.256    &
    0.331    &
    0.414                                                                           \\
             &     & 10  &
    0.05     &
    0.059    &
    0.080    &
    0.112    &
    0.135    &
    0.195    &
    0.242    &
    0.332    &
    0.418                                                                           \\
    \\
             & 300 &
    2        &
    0.058    &
    0.086    &
    0.117    &
    0.155    &
    0.211    &
    0.274    &
    0.383    &
    0.481    &
    0.618                                                                           \\
             &     & 6   &
    0.058    &
    0.089    &
    0.109    &
    0.147    &
    0.208    &
    0.289    &
    0.368    &
    0.488    &
    0.625                                                                           \\
             &     & 17  &
    0.058    &
    0.075    &
    0.111    &
    0.148    &
    0.211    &
    0.274    &
    0.361    &
    0.482    &
    0.599                                                                           \\
    \\
    \\
    5        & 100 &
    2        &
    0.048    &
    0.066    &
    0.095    &
    0.124    &
    0.162    &
    0.219    &
    0.266    &
    0.332    &
    0.422                                                                           \\
             &     & 4   &
    0.053    &
    0.063    &
    0.089    &
    0.107    &
    0.159    &
    0.202    &
    0.261    &
    0.328    &
    0.405                                                                           \\
             &     & 10  &
    0.050    &
    0.072    &
    0.088    &
    0.122    &
    0.152    &
    0.197    &
    0.265    &
    0.329    &
    0.413                                                                           \\
    \\
             & 300 &
    2        &
    0.070    &
    0.086    &
    0.118    &
    0.167    &
    0.226    &
    0.301    &
    0.406    &
    0.511    &
    0.628                                                                           \\
             &     & 6   &
    0.066    &
    0.088    &
    0.125    &
    0.164    &
    0.231    &
    0.294    &
    0.393    &
    0.503    &
    0.621                                                                           \\
             &     & 17  &
    0.063    &
    0.095    &
    0.122    &
    0.157    &
    0.235    &
    0.278    &
    0.388    &
    0.492    &
    0.614
  \end{tabular}
  \caption[Rejection rates of the test based on Pearson's correlation for stochastic volatility  time series]{Rejection rates of the test based on Pearson's correlation for stochastic volatility  time series, i.e.\@ for $X_j=(\xi_j-\mathbb{E}\xi_j)\exp(\eta_j)$, $j=1, \ldots, n$ where $\eta_j$, $j=1, \ldots, n$, stems from an AR(1) process with parameter $\rho$ and $\xi_j$, $j=1, \ldots, n$, from iid Pareto distributed random variables with shape parameter $\alpha$.}
  \label{table:app_Pearson_SV}
\end{table}

\begin{table}
  \begin{tabular}{cccccccccccc}
             &     &     &     &     &     & $\rho$                               \\
    \cline{4-12}                                                                  \\
    $\alpha$ & $n$ & $d$ & 0.1 & 0.2 & 0.3 & 0.4    & 0.5 & 0.6 & 0.7 & 0.8 & 0.9 \\
    \cline{1-12}                                                                  \\
    \\
    3        & 100 & 2   &
    0.040    &
    0.069    &
    0.117    &
    0.193    &
    0.316    &
    0.471    &
    0.648    &
    0.794    &
    0.894
    \\
             &     & 4   &
    0.036    &
    0.062    &
    0.118    &
    0.191    &
    0.301    &
    0.451    &
    0.614    &
    0.782    &
    0.890                                                                         \\
             &     & 10  &
    0.042    &
    0.076    &
    0.100    &
    0.172    &
    0.283    &
    0.425    &
    0.591    &
    0.753    &
    0.862
    \\
    \\
             & 300 & 2   &
    0.063    &
    0.127    &
    0.273    &
    0.510    &
    0.770    &
    0.917    &
    0.976    &
    0.996    &
    0.999                                                                         \\
             &     & 6   &
    0.064    &
    0.139    &
    0.283    &
    0.502    &
    0.755    &
    0.913    &
    0.977    &
    0.992    &
    0.999                                                                         \\
             &     & 17  &
    0.062    &
    0.123    &
    0.260    &
    0.487    &
    0.723    &
    0.889    &
    0.969    &
    0.991    &
    0.998                                                                         \\
    \\
    \\
    5        & 100 &
    2        &
    0.042    &
    0.075    &
    0.121    &
    0.190    &
    0.325    &
    0.488    &
    0.649    &
    0.817    &
    0.914                                                                         \\
             &     & 4   &
    0.045    &
    0.068    &
    0.114    &
    0.190    &
    0.308    &
    0.475    &
    0.626    &
    0.803    &
    0.908                                                                         \\
             &     & 10  &
    0.041    &
    0.072    &
    0.110    &
    0.184    &
    0.303    &
    0.433    &
    0.602    &
    0.764    &
    0.866                                                                         \\
    \\
             & 300 & 2   &
    0.071    &
    0.138    &
    0.300    &
    0.553    &
    0.793    &
    0.942    &
    0.981    &
    0.998    &
    0.999                                                                         \\
             &     & 6   &
    0.072    &
    0.133    &
    0.300    &
    0.536    &
    0.794    &
    0.937    &
    0.981    &
    0.997    &
    1.000
    \\
             &     & 17  &
    0.064    &
    0.139    &
    0.280    &
    0.516    &
    0.772    &
    0.922    &
    0.980    &
    0.996    &
    1.000                                                                         \\
  \end{tabular}
  \caption[Rejection rates of the test based on distance covariance for stochastic volatility time series]{Rejection rates of the test based on distance covariance for stochastic volatility time series, i.e.\@ for $X_j=(\xi_j-\mathbb{E}\xi_j)\exp(\eta_j)$, $j=1, \ldots, n$ where $\eta_j$, $j=1, \ldots, n$, stems from an AR(1) process with parameter $\rho$ and $\xi_j$, $j=1, \ldots, n$, from iid Pareto distributed random variables with shape parameter $\alpha$.}
  \label{table:app_dist_corr_SV}
\end{table}

\begin{table}
  \begin{tabular}{ccccccccccc}
          &     &     &     &     &     & $\rho$                         \\

    \cline{3-11}                                                         \\
    $n$   & $d$ & 0.1 & 0.2 & 0.3 & 0.4 & 0.5    & 0.6 & 0.7 & 0.8 & 0.9 \\
    \cline{1-11}                                                         \\
    \\
    100   & 2   &
    0.049 &
    0.057 &
    0.063 &
    0.078 &
    0.099 &
    0.119 &
    0.162 &
    0.246 &
    0.382
    \\
          & 4   &
    0.053 &
    0.054 &
    0.061 &
    0.065 &
    0.075 &
    0.084 &
    0.110 &
    0.145 &
    0.256                                                                \\
          & 10  &
    0.048 &
    0.055 &
    0.053 &
    0.059 &
    0.066 &
    0.073 &
    0.080 &
    0.105 &
    0.148                                                                \\
    \\
    300   & 2   &
    0.051 &
    0.060 &
    0.063 &
    0.066 &
    0.087 &
    0.118 &
    0.167 &
    0.243 &
    0.392                                                                \\
          & 6   &
    0.053 &
    0.052 &
    0.058 &
    0.068 &
    0.065 &
    0.071 &
    0.086 &
    0.111 &
    0.200                                                                \\
          & 17  &
    0.054 &
    0.045 &
    0.051 &
    0.060 &
    0.056 &
    0.058 &
    0.068 &
    0.074 &
    0.098
  \end{tabular}
  \caption[Rejection rates of the test based on Pearson's correlation for two independent AR(1) processes]{Rejection rates of the test based on Pearson's correlation for two independent AR(1) processes with parameter $\rho$  and length $n$. As block length we choose $d\in \{2, n^{\frac{1}{3}}, \sqrt{n}\}$.}
  \label{table:app_Pearson_ind}
\end{table}

\begin{table}
  \begin{tabular}{ccccccccccc}
          &     &     &     &     &     & $\rho$                         \\
    \cline{3-11}                                                         \\
    $n$   & $d$ & 0.1 & 0.2 & 0.3 & 0.4 & 0.5    & 0.6 & 0.7 & 0.8 & 0.9 \\
    \cline{1-11}                                                         \\
    \\
    100   & 2   &
    0.051 &
    0.055 &
    0.066 &
    0.071 &
    0.087 &
    0.112 &
    0.163 &
    0.26  &
    0.410                                                                \\
          & 4   &
    0.053 &
    0.053 &
    0.061 &
    0.061 &
    0.069 &
    0.077 &
    0.099 &
    0.152 &
    0.282                                                                \\
          & 10  &
    0.050 &
    0.057 &
    0.054 &
    0.055 &
    0.067 &
    0.068 &
    0.080 &
    0.097 &
    0.148                                                                \\
    \\
    300   & 2   &
    0.060 &
    0.056 &
    0.064 &
    0.068 &
    0.083 &
    0.119 &
    0.170 &
    0.282 &
    0.530                                                                \\
          & 6   &
    0.053 &
    0.051 &
    0.054 &
    0.054 &
    0.061 &
    0.065 &
    0.090 &
    0.113 &
    0.215                                                                \\
          & 17  &
    0.049 &
    0.048 &
    0.052 &
    0.058 &
    0.053 &
    0.051 &
    0.061 &
    0.065 &
    0.101
  \end{tabular}
  \caption[Rejection rates of the test based on distance covariance for two independent AR(1) processes]{Rejection rates of the test based on distance covariance for two independent AR(1) processes with parameter $\rho$ and length $n$. As block length we choose $d\in \{2, n^{\frac{1}{3}}, \sqrt{n}\}$.}
  \label{table:app_dist_corr_ind}
\end{table}

\begin{table}
  \begin{tabular}{ccccccccccc}
          &       &     &     &     &     & $\rho$                         \\
    \cline{3-11}                                                           \\
    $n$   & $d$   & 0.1 & 0.2 & 0.3 & 0.4 & 0.5    & 0.6 & 0.7 & 0.8 & 0.9 \\
    \cline{1-11}                                                           \\
    \\
    100
          & 2
          & 0.146
          & 0.306
          & 0.540
          & 0.767
          & 0.907
          & 0.981
          & 0.997
          & 1.000
          & 1.000                                                          \\
          & 4
          & 0.122
          & 0.274
          & 0.496
          & 0.737
          & 0.893
          & 0.976
          & 0.996
          & 1.000
          & 1.000                                                          \\
          & 10
          & 0.114
          & 0.263
          & 0.478
          & 0.711
          & 0.884
          & 0.972
          & 0.996
          & 1.000
          & 1.000                                                          \\
    \\
    300
          & 2
          & 0.251
          & 0.647
          & 0.916
          & 0.995
          & 1.000
          & 1.000
          & 1.000
          & 1.000
          & 1.000                                                          \\
          & 6     &
    0.203 &
    0.583 &
    0.886 &
    0.991 &
    0.999 &
    1.000 &
    1.000 &
    1.000 &
    1.000                                                                  \\
          & 17    &
    0.179 &
    0.545 &
    0.869 &
    0.990 &
    0.999 &
    1.000 &
    1.000 &
    1.000 &
    1.000                                                                  \\
  \end{tabular}
  \caption[Rejection rates of the test based on Pearson's correlation for $\mathrm{VAR}(1)$ processes]{Rejection rates of the test based on Pearson's correlation for $\mathrm{VAR}(1)$ processes with parameters $a=0.5$, $b=0$, and $r$, correlation matrix $\Gamma(0)$ as in Eq.\@ \eqref{eq:corr_matrix}, and length $n$. As block length we choose $d\in \{2, n^{\frac{1}{3}}, \sqrt{n}\}$.}
  \label{table:app_Pearson}
\end{table}

\begin{table}
  \begin{tabular}{ccccccccccc}
          &       &     &     &     &     & $\rho$                         \\
    \cline{3-11}                                                           \\
    $n$   & $d$   & 0.1 & 0.2 & 0.3 & 0.4 & 0.5    & 0.6 & 0.7 & 0.8 & 0.9 \\
    \cline{1-11}                                                           \\
    \\
    100   & 2     &
    0.138 &
    0.275 &
    0.495 &
    0.724 &
    0.883 &
    0.970 &
    0.994 &
    1.000 &
    1.000                                                                  \\
          & 4     &
    0.113 &
    0.244 &
    0.454 &
    0.690 &
    0.863 &
    0.965 &
    0.994 &
    1.000 &
    1.000                                                                  \\
          & 10    &
    0.096 &
    0.226 &
    0.431 &
    0.667 &
    0.849 &
    0.959 &
    0.992 &
    1.000 &
    1.000                                                                  \\
    \\
    300
          & 2     &
    0.243
          & 0.584
          & 0.888
          & 0.989
          & 1.000
          & 1.000
          & 1.000
          & 1.000
          & 1.000                                                          \\
          & 6
          & 0.195
          & 0.528
          & 0.851
          & 0.983
          & 1.000
          & 1.000
          & 1.000
          & 1.000
          & 1.000                                                          \\
          & 17    &
    0.162
          & 0.497
          & 0.830
          & 0.982
          & 1.000
          & 1.000
          & 1.000
          & 1.000
          & 1.000
  \end{tabular}
  \caption[Rejection rates of the test based on distance covariance for $\mathrm{VAR}(1)$ processes]{Rejection rates of the test based on distance covariance for $\mathrm{VAR}(1)$ processes with parameters $a=0.5$, $b=0$, and $r$, correlation matrix $\Gamma(0)$ as in Eq.\@ \eqref{eq:corr_matrix}, and length $n$. As block length we choose $d\in \{2, n^{\frac{1}{3}}, \sqrt{n}\}$.}
  \label{table:app_dist_corr}
\end{table}
\end{document}